\documentclass[11pt,reqno]{amsproc}

    \usepackage[sc]{mathpazo}\renewcommand{\mathbf}{\mathbold}
    \normalfont
    \usepackage[T1]{fontenc}

    \usepackage[margin=1in]{geometry}
    \usepackage{comment}
    \usepackage[small,bf]{caption}
    \usepackage{scalerel,nicefrac}
    \usepackage[all,cmtip]{xy}
    \usepackage{tikz-cd}
    \usepackage{tikz} \usetikzlibrary{arrows.meta,calc}

    \usepackage{rotating, setspace}

    \usepackage{bbm}
    \usepackage{listings}
    \usepackage{geometry}
    \usepackage{amsmath,amsthm,amssymb}
    \usepackage[hypertexnames=false]{hyperref}
    \usepackage{color}
    \usepackage{mathtools}
    \usepackage{tikz-cd}
    \usepackage{subfigure}
    \usepackage{shuffle}
    \usepackage{mathrsfs}
    \usepackage{multicol}
    \usepackage{cleveref}
    \usepackage{thmtools}
    \usepackage{longtable}
    \usepackage{booktabs}
    \usepackage[utf8]{inputenc}

    \usepackage{autonum}

    \makeatletter
    \renewcommand{\@secnumfont}{\bfseries}
    \makeatother

    \hypersetup{
        colorlinks=true,
        linkcolor=blue,
        filecolor=blue,      
        }

    \makeatletter
    \def\l@subsection{\@tocline{2}{0pt}{2.5pc}{5pc}{}}
    \makeatother
    \numberwithin{equation}{section}

    \newcommand{\N}{\mathbb{N}}
    
    \newcommand{\R}{\mathbb{R}}

    \newcommand{\E}{\mathbb{E}}
    \newcommand{\U}{\mathbb{U}}
    \newcommand{\PP}{\mathbb{P}}
    \newcommand{\one}{\mathbb{1}}

    \newcommand{\bx}{\mathbf{x}}
    \newcommand{\by}{\mathbf{y}}
    \newcommand{\bz}{\mathbf{z}}

    \newcommand{\bu}{\mathbf{u}}
    \newcommand{\bv}{\mathbf{v}}
    \newcommand{\bw}{\mathbf{w}}

    \newcommand{\cF}{\mathcal{F}}
    \newcommand{\cN}{\mathcal{N}}

    \newcommand{\cE}{\mathcal{E}}
    \newcommand{\cB}{\mathcal{B}}
    \newcommand{\cL}{\mathcal{L}}
    
    \newcommand{\cC}{\mathcal{C}}
    \newcommand{\cD}{\mathcal{D}}
    \newcommand{\cV}{\mathcal{V}}
    \newcommand{\cT}{\mathcal{T}}

    \newcommand{\texp}{\widetilde{\exp}}

    \newcommand{\tp}{\tilde{p}}
    \newcommand{\tL}{\widetilde{L}}
    \newcommand{\talpha}{\tilde{\alpha}}

    \newcommand{\teta}{\tilde{\eta}}
    \newcommand{\hd}{\hat{d}}
    \newcommand{\hdelta}{\hat{\delta}}
    \newcommand{\tw}{\tilde{w}}

    \newcommand{\hzeta}{\hat{\zeta}}

    \newcommand{\ls}{\lesssim}

    \newcommand{\homega}{\hat{\omega}}
    
    \newcommand{\hrho}{\hat{\rho}}

    \newcommand{\vol}{\operatorname{vol}}
    
    \newcommand{\sgn}{\operatorname{sgn}}
    
    \newcommand{\Med}{\operatorname{Med}}
    \newcommand{\inj}{\operatorname{inj}}
    \newcommand{\cx}{\operatorname{cx}}
    \newcommand{\im}{\operatorname{im}}

    \newcommand{\diam}{\operatorname{diam}}

    \newcommand{\Var}{\operatorname{Var}}
    
    \newcommand{\Lip}{\operatorname{Lip}}
    \newcommand{\spann}{\operatorname{span}}

    \newcommand{\dvol}{d\!\vol}

    \newcommand{\ball}[3]{B_{#3}(#1, #2)} %

    \newtheorem{counter}{Counter}[section]
    \newtheorem{lemma}[counter]{Lemma}
    
    \newtheorem{proposition}[counter]{Proposition}
    \newtheorem{theorem}[counter]{Theorem}
    \newtheorem{corollary}[counter]{Corollary}

    \declaretheoremstyle[
    headfont=\color{red}\normalfont\bfseries,
    ]{colored}

    \newcommand{\andd}{\quad \text{and} \quad}

    \theoremstyle{definition}
    \newtheorem{definition}[counter]{Definition}
    
    \newtheorem{remark}[counter]{Remark}

    \theoremstyle{plain}

\title{The Geometry of  Cochains on Sampled Vietoris-Rips Complexes}

\author{Darrick Lee}
\address{School of Mathematics and Maxwell Institute, University of Edinburgh, Edinburgh EH9 3FD, Scotland}
\email{darrick.lee@ed.ac.uk}
\author{Kelly Maggs}
\address{Max Planck Institute of Molecular Cell Biology and Genetics; Dresden, Germany. \newline
\indent Center for Systems Biology, Dresden, Germany.\newline
\indent Faculty of Mathematics, Technische Universitat Dresden, Germany. 
}
\email{maggskel@gmail.com}

\begin{document}

\begin{abstract}
We study the geometry of simplicial cochains on Vietoris--Rips complexes built from i.i.d.\ samples of a compact embedded manifold. By integrating differential forms over affine simplices, we define a cochain map from smooth forms to simplicial cochains and equip the latter with kernel-weighted inner products determined by ambient pairwise distances. At scales where the sampled complex recovers the homotopy type of the manifold, we prove quantitative high-probability convergence of these inner products and, under uniform sampling, of the associated codifferential energies to their continuum counterparts. As consequences, we obtain spectral upper bounds for the discrete Hodge Laplacian and, for orientable manifolds, convergence of harmonic representatives of discretized harmonic one-forms and consistency of harmonic smoothing for circular coordinates.
\end{abstract}

\maketitle

{\footnotesize
\tableofcontents
}

\section{Introduction}

The Vietoris-Rips (VR) complex occupies a central role in topological data analysis (TDA) as it gives a concrete simplicial complex built from point-cloud (or finite metric) data at a scale parameter $\epsilon > 0$. Under manifold sampling assumptions, one has rigorous topological guarantees that the VR complex built from point samples is homotopy equivalent to the underlying topological space~\cite{hausmann1995vietorisrips,latschev_vietoris-rips_2001,majhi_demystifying_2024}. At the same time, Laplacian-based methods on graphs and simplicial complexes have become important tools for extracting geometric and harmonic information from data. This approach uses combinatorial Hodge theory, and requires an inner product structure on the simplicial cochain complex of the discrete space. Given any such choice of an inner product, Eckmann~\cite{Eckmann1944} developed a combinatorial Hodge Laplacian, whose kernel recovers cohomology, and furthermore yields a Hodge decomposition. While spectral properties of the combinatorial Laplacian are studied in their own right \cite{horak_spectra_2013, Lim-2020}, these Laplacian-based methods have also been utilized in TDA, from defining circular coordinates using harmonic cochains~\cite{deSilva2011} to more recent work on persistent Laplacians \cite{guo-wei-2020,memoli-2022, Chen2022}.

These constructions rely crucially on the choice of an inner product on the cochain complex; the standard choice is to declare the simplex-dual basis orthonormal. On a sampled manifold, however, simplex counts, Euclidean shapes, and sampling density all affect what a cochain norm should approximate. At present, there is limited theory justifying this choice, or more broadly explaining the relationship between the \emph{geometry} of discrete simplicial cochains of VR complexes built from sampled data, and the corresponding continuum objects from the underlying space. In this article, we consider the setting of VR complexes built from a point cloud $X$ sampled i.i.d.~from an embedded compact manifold $M \subset \R^N$. Our main goal is to construct appropriate inner products and develop a de Rham-type discretization procedure from smooth differential $k$-forms to discrete simplicial cochains which is geometrically consistent with the $L^2$ Riemannian geometry of the underlying manifold. Our approach uses only the ambient Euclidean geometry of the sample, through affine simplices and pairwise distances, and does not require a priori knowledge of the intrinsic geometry of the manifold.

Our constructions are motivated by manifold learning methods used to understand various notions of convergence of graph Laplacians built from such sampled data to the Laplace-Beltrami operator on the underlying manifold. Early work~\cite{belkin_towards_2008,coifman_diffusion_2006} built fully connected graphs with weights defined using Gaussian heat kernels leading to qualitative spectral convergence results~\cite{belkin_convergence_2006}. More recent approaches refined these results by considering a broader class of kernels on sparser $\epsilon$-graphs (the $1$-skeleton of the VR complex) and $k$-nearest neighbour graphs, and obtained quantitative convergence results \cite{garcia_trillos_error_2020,calder_improved_2022}. 
Our aim is to extend this perspective to $k$-forms on the full VR cochain complex, focusing on the discretization direction: smooth forms are sent to cochains, but general cochains are not interpolated back to smooth forms. In particular, our work builds on the $\epsilon$-graph approaches. Beyond computational sparsity, these are also the finite-sample scales at which reconstruction theorems identify the VR complex with $M$. \medskip

\textbf{Contributions and Outline.} Let $X = (x_1, \ldots, x_n) \subset \R^N$ be a finite collection of i.i.d.~samples from a compact embedded manifold $M \subset \R^N$, and let $X_\epsilon$ denote the VR complex at scale $\epsilon>0$. In \Cref{sec:geometric_prelim}, we begin by recalling preliminary geometric constructions used throughout the article, and discuss scale assumptions required for normal coordinate computations and distance comparisons used in our arguments. 

Then in \Cref{sec:discretization}, we define our discretization of an ambient $k$-form $\alpha \in \Omega^k_c(\R^N)$ by integrating it over the affine Euclidean $k$-simplex spanned by each $k$-simplex of $X_{\epsilon}$, writing the resulting simplicial cochain as $q\alpha \in C^k(X_\epsilon)$. This is a cochain map $q: \Omega^\bullet_c(\R^N) \to C^\bullet(X_\epsilon)$. We equip $C^k(X_{\epsilon})$ with a kernel-weighted inner product $\langle \cdot, \cdot\rangle_n$ whose weights depend on the pairwise Euclidean distances among the vertices of a simplex. We obtain smooth and simplicial codifferentials,
\begin{equation}
\delta: \Omega^k(M) \to \Omega^{k-1}(M) \andd \hdelta: C^k(X_\epsilon) \to C^{k-1}(X_\epsilon)
\end{equation}
as the adjoints to the exterior derivative $d: \Omega^{k-1}(M) \to \Omega^k(M)$ (under the $L^2$ inner product) and simplicial coboundary $\hd: C^{k-1}(X_\epsilon) \to C^k(X_\epsilon)$ (under our kernel-weighted inner products).

In \Cref{sec:main_results}, we present our main convergence results for the inner product and the codifferential. Let $U_\epsilon \subset \R^N$ be the $\epsilon$-tubular neighbourhood of $M \subset \R^N$ with nearest point projection $\pi: U_\epsilon \to M$. For the codifferential convergence, our results are stated for \emph{pullback forms} $\alpha = \pi^* \talpha \in \Omega^k(U_\epsilon)$, where $\talpha \in \Omega^k(M)$. At the relevant scales the affine simplices lie in $U_\epsilon$, so this is sufficient to define $q\alpha$.

\begin{theorem}[Informal]
    Let $\cB^k_\R\subset \Omega^k_c(\R^N)$ and $\cB^k_M\subset \Omega^k(M)$ be the unit $\cC^2$-balls. For appropriate choices of $\epsilon_n$ and kernel weights, there exist constants $\sigma_k, \sigma_{\delta, k} > 0$ such that with high probability,
    \begin{enumerate}
        \item \emph{\textbf{(Topological Consistency)}} the VR complex $X_{\epsilon_n}$ is homotopy equivalent to $M$;
        \item \emph{\textbf{(Inner Product Convergence)}} uniformly over $\alpha, \beta \in \cB^k_\R$, we have $\langle q\alpha, q\beta\rangle_n \to \sigma_k \langle \alpha, \beta \rangle_M$;
        \item \emph{\textbf{(Codifferential Convergence)}} in the uniform sampling density setting, uniformly over $\tilde{\alpha}, \tilde{\beta} \in \cB^k_M$, we have $\langle \hdelta q \alpha, \hdelta q \beta \rangle_n \to \sigma_{\delta, k} \langle \delta \tilde{\alpha}, \delta \tilde{\beta} \rangle_M$, where $\alpha = \pi^* \tilde{\alpha}$ and $\beta = \pi^* \tilde{\beta}$.
    \end{enumerate}
\end{theorem}

This result shows that we obtain convergence in probability of \emph{geometric} structures within a bandwidth regime $\epsilon_n$ where $X_{\epsilon_n}$ has the correct homotopy type. The topological consistency result is stated formally in~\Cref{prop:topology_quasi_iso_convergence_scales}. The inner product and codifferential convergence results are stated formally in~\Cref{thm:main_uc} and~\Cref{thm:main_codiff}, where explicit rates are provided given a choice of $\epsilon_n$. Furthermore, these are \emph{uniform} convergence results within the $\cC^2$-ball of differential forms. The proofs of the convergence results are deferred to \Cref{sec:uc_proof} and \Cref{sec:codiff_proof}.

Finally, in \Cref{sec:rips_laplacian}, we discuss some immediate consequences of our results related to the discrete Hodge Laplacian in degrees $0\leq k\leq m-1$ (normalized by a constant $\bar{\sigma}_{\delta,k}$ which arises from our convergence results),
\begin{equation}
    \hat\Delta_{k,n}=\hdelta_n\hd_n+\bar{\sigma}_{\delta,k}^{-1}\hd_n\hdelta_n : C^k(X_\epsilon) \to C^k(X_\epsilon) \quad \text{with eigenvalues} \quad 0 \leq \hat{\lambda}^{k}_{0,n} \leq \hat{\lambda}^k_{1,n} \leq \ldots .
\end{equation}
This is the discrete counterpart to the smooth Hodge Laplacian
\begin{equation}
    \Delta_k = \delta d + d \delta : \Omega^k(M) \to \Omega^k(M) \quad \text{with eigenvalues} \quad 0 \leq \lambda^k_0 \leq \lambda^k_1 \leq \ldots.
\end{equation}
For the following results, we fix a sequence $\epsilon_n$ such that our convergence results hold at the appropriate degrees. First, we provide a spectral upper bound, proved in~\Cref{thm:rips_spectral_upper_bound_sec7}.
\begin{corollary}
    For each fixed $j\in\N_0$, $0\leq k\leq m-1$, and $\eta>0$, we have
    \begin{align}
        \PP\left[\hat{\lambda}^k_{j,n}\leq\lambda^k_j+\eta\right]\to1.
    \end{align}
\end{corollary}

In order to obtain stronger spectral convergence results, one must also develop \emph{interpolation} methods to go from discrete cochains back to continuous differential forms. While this is beyond the scope of the present article, we view this as a promising avenue for future work. Despite this limitation, for orientable $M$, we consider the convergence of harmonic 1-forms by leveraging the known spectral convergence results for the graph Laplacian in~\cite{calder_improved_2022}. This is proved in~\Cref{prop:harmonic_one_form_projection_sec7}.

\begin{corollary} \label{cor:intro_harmonic_convergence}
    Let $M$ be connected and orientable with $m\geq2$ and let $\omega \in \Omega^1(M)$ be a harmonic $1$-form. For each $n$, let $\hat{\omega}_n \in C^1(X_{\epsilon_n})$ be the harmonic part of the pullback discretization $q\omega$ in the discrete Hodge decomposition. Then, $\|q\omega - \hat{\omega}_n\|_n \to 0$ in probability. 
\end{corollary}

As an application, \Cref{cor:circular_coordinate_convergence_sec7} proves convergence of harmonic smoothing in circular coordinates~\cite{deSilva2011}, when persistent cohomology selects the integral class of the limiting harmonic form. \medskip

\textbf{Related Work.}
Lerch and Wahl \cite{lerch_empirical_2025} define an empirical exterior calculus on certain alternating functions over a finite sample (called \emph{empirical forms}) and prove high-probability bounds for the up-Laplacian energy of such forms. Their work is closely aligned with ours in treating differential forms directly at the level of sampled cochains; however their objects and results are different in several important ways. First, their weights are defined using the intrinsic heat kernel on the manifold, rather than from ambient distances between sample points. Second, their empirical forms are defined on all sample tuples, rather than on the simplices of a sampled VR complex, so the construction is not tied to the finite-scale topology guaranteed by reconstruction theorems.
Finally, their main estimate is for the up-Laplacian energy and does not give convergence of the codifferential term needed for the full Hodge Laplacian on the sampled complex.

Other approaches avoid a finite simplicial cochain complex and instead approximate Hodge operators or form pairings directly. Spectral exterior calculus \cite{Berry2020} builds forms from Laplace--Beltrami eigenfunctions and proves spectral convergence for Galerkin approximations of the $1$-Hodge Laplacian. Diffusion-geometry methods \cite{jones2026computing} and neural point-forms \cite{trentini2026neuralpointforms} use carr\'e-du-champ or Gram-field constructions to compute form inner products, Hodge energies, and learned form-comparison matrices from point clouds. L\^e's empirical Hodge Laplacians \cite{le_empirical_2026} estimate tangent projections and curvature terms from a uniform point cloud, build empirical Hodge operators on projected ambient exterior powers, and prove spectral convergence together with recovery of the near-zero harmonic cluster. 
These methods have strong operator-theoretic and computational advantages, especially in avoiding large complexes and, in some cases, obtaining stronger spectral statements. The tradeoff is algebraic: they do not give a finite sampled complex in which cocycles, coboundaries, and chosen cohomology classes are available exactly at reconstruction scales; harmonic information is instead recovered from a near-zero spectral cluster on sampled vector data. In particular, they do not directly formulate convergence statements such as~\Cref{cor:intro_harmonic_convergence}, where $\hat{\omega}_n$ is the harmonic representative of a specified cohomology class in the sampled VR complex.

A different line of work proves Hodge-theoretic convergence from controlled and deterministic intrinsic discretizations, such as covers, triangulations, or finite-element meshes. Mantuano \cite{mantuano_discretization_2008} starts from an intrinsic $\epsilon$-net, builds the nerve of the cover by geodesic balls, and proves a two-sided comparison between the smooth Hodge spectrum and the combinatorial spectrum on the resulting \v{C}ech cochains. Dodziuk's finite-difference approach and finite element exterior calculus \cite{dodziuk_finite-difference_1976,Arnold2010} use Whitney or finite-element forms, cochain projections, and mesh refinement to compare smooth and discrete Hodge theory on controlled complexes. Mantuano's result is especially close in spirit, since it also compares smooth and discrete Hodge theory through discretization and interpolation maps; however, these maps are built intrinsically from the \v{C}ech--de Rham double complex. The main distinction is that their discretizations are built from intrinsic geometric data or prescribed meshes; our complex is the sampled VR complex at various reconstruction scales, and its metric is defined only from ambient Euclidean distances. \medskip

\phantomsection\label{not:global_conventions}
\textbf{Notations and Conventions.}
    Throughout this article, $A\ls B$ means that $A\leq CB$, where $C$ may depend on fixed data but is independent of $n$, $\epsilon$, $p$, and any forms over which a supremum is taken, unless stated otherwise. We write $A\ls_p B$ when dependence on $p$ is allowed, and $A\asymp B$ when both $A\ls B$ and $B\ls A$. A table of notation is provided in~\Cref{apxsec:notation}. \medskip

\textbf{AI Use Statement.}
The authors used OpenAI’s ChatGPT and Anthropic's Claude (to a much lesser extent), as interactive research and writing aids. The research direction, specific results we sought to prove, and the organization of the paper originated with the authors. AI tools were used to explore possible proof arguments, assist with \LaTeX, and help draft and revise portions of the exposition. The authors verified all mathematical claims, proofs, and references, made all final decisions concerning the content of the paper, and take full responsibility for its correctness. Further details are provided in~\Cref{apxsec:ai_use}.\medskip

    \textbf{Acknowledgements.}
    We would like to thank Daniel Tolosa for helpful discussions at the beginning of this project. K.M.~is a member of the Max Planck/Oxford international Center to Center collaboration supported by EPSRC EP/Z531224/1.

    \section{Geometric Preliminaries} \label{sec:geometric_prelim}

    In this article, we work with an $m$-dimensional compact Riemannian manifold $M$ without boundary, smoothly embedded by $\iota: M \hookrightarrow \R^N$. We equip $M$ with the induced metric $g$ where
    \begin{align} \label{eq:riemannian_metric}
        \langle u, v \rangle_x = g_x(u,v) = \langle \iota_* u, \iota_* v \rangle_\R \quad \text{for} \quad x \in M \andd u, v \in T_x M,
    \end{align}
    and $\langle \cdot , \cdot \rangle_\R$ denotes the standard inner product on $\R^N$. We write $\dvol$ for the associated Riemannian volume density and $\vol_M$ for the Borel measure that it induces. Product measures and integrals below are understood in this density sense.

    \subsection{Differential Forms and Hodge Operators} \label{ssec:differential_forms}
    Let $\Omega^k(M)$ be the space of smooth differential $k$-forms on $M$; these are the sections of the vector bundle $\Lambda^k T^* M$. For each $x \in M$, the Riemannian metric~\eqref{eq:riemannian_metric} induces an inner product on $T^*_x M$, and furthermore on the exterior power $\Lambda^k T_x^* M$ defined by
    \begin{align}
        \langle u_1 \wedge \ldots \wedge u_k, v_1 \wedge \ldots \wedge v_k \rangle_x = \det( \langle u_i, v_j\rangle_x).
    \end{align}
    Then, given two differential forms $\alpha, \beta \in \Omega^k(M)$, we can define the $L^2$ inner product by
    \begin{align}
        \langle \alpha, \beta \rangle_M \coloneqq \int_M \langle \alpha_x, \beta_x \rangle_x \, \dvol(x).
    \end{align}
    Note that we use $\langle \cdot , \cdot \rangle_x$ to denote all pointwise inner products and $\langle \cdot , \cdot \rangle_M$ to denote the $L^2$ inner product. We will also need to consider \emph{weighted} $L^2$ inner products. Given a smooth function $f: M \to (0,\infty)$, we define the $f$-weighted $L^2$ inner product by
    \begin{align} \label{eq:weighted_L2}
        \langle \alpha, \beta \rangle_M^f \coloneqq \int_M \langle \alpha_x, \beta_x \rangle_x \, f(x) \dvol(x).
    \end{align}

    Next, let $d: \Omega^k(M) \to \Omega^{k+1}(M)$ be the exterior derivative such that $(\Omega^\bullet(M), d)$ is a cochain complex, $d^2 = 0$. The \emph{codifferential} $\delta: \Omega^{k+1}(M) \to \Omega^k(M)$ is the formal adjoint to $d$, defined by
    \begin{align} \label{eq:smooth_codifferential}
        \langle d\alpha, \beta \rangle_M = \langle \alpha, \delta \beta \rangle_M, \quad \alpha \in \Omega^k(M), \quad \beta \in \Omega^{k+1}(M). 
    \end{align}
    Finally, the \emph{Hodge Laplacian} on $\Omega^k(M)$ is
    \begin{align} \label{eq:smooth_laplacian}
        \Delta_k = d \delta + \delta d : \Omega^k(M) \to \Omega^k(M),
    \end{align}
    and we say that $\omega \in \Omega^k(M)$ is a \emph{harmonic $k$-form} if $\Delta_k \omega = 0$. Further details can be found in~\cite{jost_riemannian_2013}.

    \subsection{Ambient Forms and $\cC^r$ Norms} \label{ssec:ambient_forms}
    While the differential forms in $\Omega^k(M)$ are the main limiting objects in this article, we wish to approximate them without prior knowledge of the manifold $M$. Thus, we primarily work with ambient differential forms. Let $\Omega^k_c(\R^N)$ denote the space of compactly-supported smooth $k$-forms on $\R^N$. Note that there is no loss of generality in considering compactly supported forms, as we primarily consider pullbacks to $M$, which is compact. Given ambient forms $\alpha, \beta \in \Omega^k_c(\R^N)$, we will often use $\langle \alpha, \beta \rangle_M \coloneqq \langle \iota^* \alpha, \iota^* \beta \rangle_M$
    to denote the $L^2$ inner product of the pullback to $M$. \medskip

    We will also use $\cC^r$-norms to state uniform convergence results. Recall that an ambient $k$-form $\alpha \in \Omega^k_c(\R^N)$ can be expressed component-wise as $\alpha = \sum_{|I|=k} \alpha_I \, dz_I$ for multi-indices $I = (i_1 < \ldots < i_k)$. Then, on an open set $U \subset \R^N$, we define the $\cC^r(U)$-norm
    \[
        \|\alpha\|_{\cC^r(U)}
        :=
        \max_{|I|=k}
        \max_{|\mu|\le r}
        \sup_{z\in U}
        |\partial^\mu \alpha_I(z)|.
    \]
    Throughout, $\bar\nabla$ denotes the standard flat connection on $\R^N$ used for ambient derivatives, while $\nabla$ denotes the Levi-Civita connection on $M$ used for intrinsic derivatives.
    For intrinsic forms $\omega \in \Omega^k(M)$, we define
    \[
        \|\omega\|_{\cC^r(M)}
        :=
        \max_{0\le j\le r}
        \sup_{x\in M}
        |\nabla^j \omega_x|_x,
    \]
    where the norm $|\cdot|_x$ is induced by the Riemannian metric $g$.
    \subsection{Reach and Tubular Neighbourhoods} \label{ssec:reach_tubular}

    We shall also use a canonical ambient extension of intrinsic forms to tubular neighbourhoods. First, we consider the notion of \emph{reach}.

    \begin{definition}{\cite[Definition 2.1]{aamari_estimating_2019}}
        For a closed subset $A \subset \R^N$, the \emph{medial axis} $\Med(A)$ of $A$ is the subset of $\R^N$ consisting of points which have at least two nearest neighbours,
        \begin{align}
            \Med(A) = \{ x \in \R^N \, : \, \exists p \neq q \in A, \, \|p-x\| = \|q-x\| = d_{\R}(x,A)\}.
        \end{align}
        The \emph{reach of $A$} is defined as 
        \begin{align}
            \tau_A = d_{\R}(A, \Med(A)) = \inf_{x \in A, \, y \in \Med(A)} d_{\R}(x,y).
    \end{align}
    \end{definition}
    For $\epsilon > 0$, we define the \emph{$\epsilon$-tubular neighbourhood of $M$}, $M \subset U_\epsilon \subset \R^N$ by
    \begin{align}
        U_\epsilon \coloneqq \{ z \in \R^N \, : \, d_\R(z, M) < \epsilon\}. 
    \end{align}
    Then, if $\epsilon < \tau_M$, there exists a unique nearest point projection $\pi: U_\epsilon \to M$, which is smooth since $M$ is smooth~\cite[Theorem 2]{leobacher_existence_2021}. For any $x \in M$, recall that we can decompose the ambient tangent space by $T_x \R^N = T_x M \oplus N_x M$, where $N_x M = (T_x M)^\perp$ is the orthogonal complement of the tangent space of $M$, called the \emph{normal space}. In particular, we have $N_x M = \ker(d\pi_x)$.

    \subsection{Injectivity Radius and Normal Coordinates} \label{ssec:normal_coordinates}
    In our analysis, we generally assume that we only have access to the ambient Euclidean geometry, while our approximation methods often rely on normal coordinates on the manifold. Here, we aim to obtain normal coordinates for open neighbourhoods defined by Euclidean balls. First, we relate Euclidean balls and intrinsic balls. For any $x, y \in M$, we always have $d_{\R}(x,y) \leq d_M(x,y)$, so $\ball{M}{x}{\epsilon} \subset \ball{\R^N}{x}{\epsilon}$.
    The other inclusion is controlled by the reach.

    \begin{lemma}{\cite[Lemma 3]{boissonnat_reach_2019}}
        Let $\tau > 0$ be the reach of a compact manifold embedded in $\R^N$. For any $x,y \in M$ such that $d_{\R}(x,y) < 2\tau$, we have
        \begin{align}
            d_M(x,y)\leq 2 \tau \arcsin\left( d_{\R}(x,y)/(2\tau)\right). 
        \end{align}
    \end{lemma}

    This implies that we have the chain of inclusions
    \begin{align} \label{eq:geodesic_euclidean_balls}
    \ball{M}{x}{\epsilon} \subset \ball{\R^N}{x}{\epsilon} \cap M \subset \ball{M}{x}{S(\epsilon)}, \quad S(\epsilon) = 2 \tau \arcsin(\epsilon/2\tau).
    \end{align} 
    for any $\epsilon < 2 \tau$. Next, we wish to ensure that the Riemannian exponential is well behaved.

    \begin{definition}
        Let $M$ be a Riemannian manifold. We define the \emph{injectivity radius} of $M$ to be
        \begin{align}
            \inj(M) \coloneqq \sup \Big\{ r > 0 \, : \, \exp_x: \ball{T_x M}{0}{r} \to \ball{M}{x}{r} \text{ is a diffeomorphism for all } x \in M\Big\}.
        \end{align}
    \end{definition}

    We combine the above observations about reach and $\inj(M)$ and define a single constant,
    \begin{align}
        \inj_\R(M) \coloneqq \sup\{ \epsilon \in (0, \min\{2 \tau_M, 1\}) \, : \, S(\epsilon) < \inj(M) \}
    \end{align}
    such that we can work with a notion of normal coordinates across the entire manifold. We constrain $\inj_\R(M) \leq 1$ as a technical condition to simplify bounds. On a compact manifold, both the injectivity radius $\inj(M)$ \cite[Lemma 6.16]{lee_introduction_2019-1} and the reach $\tau_M$ \cite[Prop. 14]{thale_50_2008} are positive, which implies that $\inj_\R(M)$ is also positive. 

    \begin{corollary} \label{cor:log}
        Let $\epsilon < \inj_\R(M)$. Then for any $x \in M$, the restriction of the logarithm
        \begin{align} \label{eq:restricted_log}
            \log_x : \ball{\R^N}{x}{\epsilon} \cap M \to T_x M.
        \end{align}
        is a diffeomorphism onto its image. 
    \end{corollary}
    \begin{proof}
        Since $\epsilon < \inj_\R(M)$, then $S(\epsilon) < \inj(M)$ so that $\exp_x : \ball{T_xM}{0}{S(\epsilon)} \to \ball{M}{x}{S(\epsilon)}$ is a diffeomorphism for all $x \in M$ with inverse $\log_x$. The inclusion $\ball{\R^N}{x}{\epsilon} \cap M \subset \ball{M}{x}{S(\epsilon)}$ from \eqref{eq:geodesic_euclidean_balls} shows that the restriction of $\log_x$ to $\ball{\R^N}{x}{\epsilon} \cap M$ is a diffeomorphism onto its image. 
    \end{proof}

    We will often use this in the context of a change of variables to normal coordinates, and we denote the Jacobian of $\exp_x : \ball{T_x M}{0}{\epsilon} \to M$ at $v \in \ball{T_x M}{0}{\epsilon}$ by $J(x,v)$, and recall that
    \begin{align} \label{eq:jacobian_remainder}
        J(x,v) \leq 1 + C_M \|v\|^2. 
    \end{align}
    To obtain comparisons between ambient and intrinsic displacements, we consider the ambient exponential map for $x \in M$, 
    \begin{align} 
        \texp_x  = \iota \circ \exp_x : T_x M  \to \R^N.
    \end{align}
    Recall the third order Taylor expansion of $\texp_x$ as~\cite{monera_taylor_2014}
    \begin{align} \label{eq:exp_taylor}
        \texp_x(v) = x + v + Q(v,v) + r(v), \quad v \in T_x M,
    \end{align}
    where $Q$ is the second fundamental form, so $Q(v,v) \in N_x M$ is valued in the normal space. The remainder satisfies $|r(v)| \leq C_M \|v\|^3$. Then, if $\texp_x(v) = y$, we have
    \begin{align} \label{eq:euclidean_intrinsic_comparison}
        \|v - (y-x)\| \ls \|v\|^2 \andd \big|\|v\|^2-\|y-x\|^2\big| \ls \|v\|^4,
    \end{align}
    where the latter inequality uses the fact that the cubic terms are trivial since $\langle v, Q(v,v) \rangle_\R = 0$. \medskip

    Finally, we fix a global scale parameter for the remainder of the paper. The \emph{convexity radius} is
    \begin{equation} \label{eq:convexity_radius}
        \cx(M) = \sup \{ r \in [0, \infty) \, : \, \ball{M}{x}{r} \text{ is strongly geodesically convex for all } x \in M \}.
    \end{equation}
    Here, strongly geodesically convex means that any two points in the ball are joined by a unique minimizing geodesic contained in the ball. Because $M$ is compact, $\cx(M)>0$~\cite[Section 6.4.2]{petersen_riemannian_2016} and $\cx(M)\leq \inj(M)/2$~\cite{dibble_convexity_2017}. Fix $\Lambda_+>0$ such that the sectional curvatures of $M$ are bounded above by $\Lambda_+$. We then fix a constant $\epsilon_0=\epsilon_0(M)>0$, which depends only on $M$ such that
    \begin{align} \label{eq:global_epsilon_0}
        0<2\epsilon_0<\inj_\R(M) \andd S(2\epsilon_0)<\min\left\{\cx(M),\frac{\pi}{4\sqrt{\Lambda_+}}\right\}.
    \end{align}
    \label{not:global_small_scale}
    All subsequent small-scale estimates will be stated for $0<\epsilon<\epsilon_0$.

    \section{Discretization and Inner Products on VR Complex} \label{sec:discretization}

    We now introduce the discrete analogues of manifolds, differential forms, and inner products which were discussed in the previous section. The primary data that we work with are finite collections of points $X^{(n)} = (x_1, \ldots, x_n) \subset \R^N$, though for simplicity we will often refer to the point cloud as $X = X^{(n)}$. For $\epsilon > 0$, the \emph{$\epsilon$-Vietoris Rips (VR) complex} is a simplicial complex $X_\epsilon$, with $X$ as its vertex set, and $\sigma \subset X$ is a simplex if $\diam_\R(\sigma) < \epsilon$. We give each simplex the orientation induced by the ordering of the sample. The set of such oriented $k$-simplices of $X_\epsilon$ is
    \begin{align}
        S^k_\epsilon(X) \coloneqq \{ \bx_I = (x_{i_0}, \ldots, x_{i_k}) \subset X \, : \, I = (i_0 < \ldots < i_k), \, \diam_\R(\bx_I) < \epsilon\}. 
    \end{align}
    Recall that the simplicial $k$-chains are $C_k(X_\epsilon) \coloneqq \R[S^k_\epsilon(X)]$, and $k$-cochains $C^k(X_\epsilon)$ are the linear maps $a: C_k(X_\epsilon) \to \R$. We extend every cochain alternately to all orderings of a simplex,
    \begin{align} \label{eq:alternating_cochain_convention}
        a(y_{\pi(0)},\ldots,y_{\pi(k)})
        =\sgn(\pi)a(y_0,\ldots,y_k), \qquad \pi\in\Sigma_{k+1},
    \end{align}
    and set its value to zero on tuples with a repeated vertex. With this convention, simplicial cochains $C^\bullet(X_\epsilon)$ are equipped with the usual differential
    \begin{align}
        \hd: C^k(X_\epsilon) \to C^{k+1}(X_\epsilon), \quad \hd a(y_0, \ldots, y_{k+1}) = \sum_{i=0}^{k+1} (-1)^i a(y_0, \ldots, \hat{y}_i, \ldots, y_{k+1}).
    \end{align}
    In this article, we consider the VR complex of a point cloud sampled from an embedded manifold $M \subset \R^N$ such that $X = (x_1, \ldots, x_n) \subset M \subset \R^N$. We emphasize that we use the ambient metric from $\R^N$ to define $X_\epsilon$. In this case, the simplicial complex $X_\epsilon$ acts as our discrete approximation of $M$, and simplicial cochains $C^k(X_\epsilon)$ act as the discrete analogue of differential forms.

    \subsection{Euclidean de Rham Map} \label{ssec:euclidean_derham}
    Given an ambient differential form $\alpha \in \Omega^k_c(\R^N)$, and a point cloud $X \subset \R^N$, our aim is to discretize $\alpha$ as a simplicial cochain on $X_\epsilon$. Let $\Delta^k$ be the standard simplex with the orientation induced by its ordered vertices. For an ordered tuple $\by=(y_0,\ldots,y_k)\in(\R^N)^{k+1}$, define its affine singular-simplex parametrization by
    \begin{align}
        \sigma_{\by}:\Delta^k\longrightarrow\R^N,
        \qquad
        \sigma_{\by}(t_0,\ldots,t_k)=\sum_{i=0}^k t_i y_i.
    \end{align}
    The \emph{discretization map} (or \emph{Euclidean de Rham map}) is
    \begin{align} \label{eq:parametric_derham_map}
        q : \Omega^k_c(\R^N) \to \cC^\infty\left( (\R^N)^{k+1}, \R\right),
        \qquad
        q\alpha(\by) \coloneqq \int_{\Delta^k}\sigma_{\by}^*\alpha.
    \end{align}

    Given a point cloud $X \subset \R^N$ and its associated VR complex $X_\epsilon$, we restrict the de Rham map $q$ to the $k$-simplices $S^k_\epsilon(X)$ in $X_\epsilon$ to obtain a map
    \begin{align} \label{eq:q_cochain_map}
        q_{X, \epsilon} : \Omega^k_c(\R^N) \to C^k(X_\epsilon). 
    \end{align}
    As the point cloud $X$ and $\epsilon$ are usually clear from context, we will often suppress the subscripts and also denote this map by $q$.
    A key property of this discretization is that it is compatible with the smooth and discrete cochain structures.

    \begin{lemma} \label{lem:cochain_map}
        For $X = (x_1, \ldots, x_n) \subset \R^N$ and $\epsilon > 0$, the discretization $q_{X,\epsilon}$ is a cochain map, $\hd q = q d$. 
    \end{lemma}
    \begin{proof}
        The oriented boundary of the standard simplex is the alternating sum of its faces, so the identity follows from Stokes' theorem applied to the  parametrization~\eqref{eq:parametric_derham_map}.
    \end{proof}

    A crucial property for our later approximations is that the discretization map is approximated by evaluation against both ambient and intrinsic displacement vectors based at one of the vertices. 

    \begin{lemma} \label{lem:euclidean_simplex_integral}
        Let $\alpha \in \Omega^k_c(\R^N)$ and $\by = (y_0, \ldots, y_k) \in (\R^N)^{k+1}$. Let $\bw = (y_1 - y_0, \ldots, y_k - y_0)$ be the Euclidean displacement vectors based at $y_0$. Then,
        \begin{align} \label{eq:euclidean_integral_approximation}
            q\alpha(\by) = \frac{\alpha_{y_0}(\bw)}{k!} + \frac{1}{(k+1)!} \sum_{i=1}^k (\bar\nabla_{w_i} \alpha)_{y_0}(\bw) + R(\bw), \quad |R| \ls \|\bar\nabla^2 \alpha\|_\infty \|\bw\|^{k+2}.
        \end{align} where $\bar\nabla$ is the usual covariant derivative on $\R^N$.
    \end{lemma}
    \begin{proof}
        Let $\Delta^k = \{ (u_1, \ldots, u_k) \in \R^k \, : \, u_i\geq 0, \, \sum_{i=1}^k u_i \leq 1 \}$. We parametrize $\Delta(\by)$ by 
        \begin{align}
            \sigma: \Delta^k \to \R^N, \quad \sigma(u_1, \ldots, u_k) = y_0 + \sum_{i=1}^k u_i w_i.
        \end{align}
        Using this parametrization, we obtain
        \begin{align}
            q\alpha(\by) = \int_{\Delta(\by)} \alpha = \int_{\Delta^k} \alpha_{\sigma(u)}(\bw) du.
        \end{align}
        By fixing $\bw = (w_1, \ldots, w_k)$, and treating $\alpha_{\sigma(\cdot)}(\bw) : \Delta^k \to \R$, the second order Taylor expansion at $u=0$ is
        \begin{align}
            \alpha_{\sigma(u)}(\bw) = \alpha_{\sigma(0)}(\bw) + \sum_{i} u_i (\bar\nabla_{w_i} \alpha)_{y_0}(\bw) + R_2(u), \quad \text{where} \quad |R_2(u)| \ls \|\bar\nabla^2 \alpha\|_\infty \|\bw\|^{k+2}.
        \end{align}
        Then, integrating this expression, and using the fact that $\int_{\Delta^k} u_i du = \frac{1}{(k+1)!}$, we obtain~\eqref{eq:euclidean_integral_approximation}.
    \end{proof}

    \begin{lemma} \label{lem:geometric_simplex_integral}
        Let $\alpha \in \Omega^k_c(\R^N)$ and $\by =(y_0, \ldots, y_k) \in M^{k+1} \subset \R^N$ such that $\diam_\R(\by) < \inj_\R(M)$. Let $\bv = (\log_{y_0}(y_1), \ldots, \log_{y_0}(y_k)) \in T_{y_0}^k M$ be the intrinsic displacement vectors based at $y_0$. Then,
        \begin{align}
            q\alpha(\by) = \frac{\alpha_{y_0}(\bv)}{k!} + L(\bv) + R(\bv), \quad |R(\bv)| \ls \| \alpha\|_{\cC^2(\R^N)} \|\bv\|^{k+2}
        \end{align}
        and $L: T^k_x M \to \R$ is homogeneous of degree $k+1$; in other words, $L(\lambda \bv) = \lambda^{k+1} L(\bv)$ for all $\lambda \in \R$.
    \end{lemma}
    \begin{proof}
        By~\Cref{lem:euclidean_simplex_integral} (and using the notation there), we obtain
        \begin{align} \label{eq:simplex_expansion_odd1}
            q\alpha(\by) = \frac{\alpha_{y_0}(\bw)}{k!} + \frac{1}{(k+1)!} \sum_{i=1}^k (\bar\nabla_{w_i} \alpha)_{y_0}(\bw) + R_0(\by). 
        \end{align}
        By~\eqref{eq:euclidean_intrinsic_comparison}, we note that $|R_0(\by)| \ls \|\bar\nabla^2 \alpha\|_\infty \|\bv\|^{k+2}$. Next, using the Taylor expansion for the exponential from~\eqref{eq:exp_taylor}, we get
        \begin{align}
            w_i = \texp_{y_0}(v_i)-y_0 = v_i + Q_{y_0}(v_i, v_i) + r_i, \quad \text{where} \quad \|r_i\| \ls \|v_i\|^3,
        \end{align}
        and $Q_x$ is the second fundamental form of $M$. Applying this to the first term in~\eqref{eq:simplex_expansion_odd1}, we obtain
        \begin{align}
            \alpha_{y_0}(\bw) = \alpha_{y_0}(\bv) + \sum_{i=1}^k \alpha_{y_0}(v_1, \ldots, Q_{y_0}(v_i,v_i), \ldots, v_k) + E_1(\bv),
        \end{align}
        where $|E_1(\bv)| \ls \|\alpha\|_\infty \|\bv\|^{k+2}$. 
        For the second term in~\eqref{eq:simplex_expansion_odd1}, by linearity of $\bar\nabla_w$ in $w$, we get
        \begin{align}
            \sum_{i=1}^k (\bar\nabla_{w_i}\alpha)_{y_0}(\bw) = \sum_{i=1}^k (\bar\nabla_{v_i}\alpha)_{y_0}(\bv) + E_2(\bv),
        \end{align}
        where $|E_2| \ls \|\bar\nabla\alpha\|_\infty \|\bv\|^{k+2}$. Thus, we set
        \begin{align}
            L(\bv) = \frac{1}{k!}\sum_{i=1}^k \alpha_{y_0}(v_1, \ldots, Q_{y_0}(v_i,v_i), \ldots, v_k) + \frac{1}{(k+1)!} \sum_{i=1}^k (\bar\nabla_{v_i}\alpha)_{y_0}(\bv),
        \end{align}
        which is homogeneous of degree $k+1$, and the remainder is $R = R_0 + E_1 + E_2$. 
    \end{proof}

    Both estimates apply at any base vertex after reordering the simplex, with only the orientation sign changed.

    \subsection{Bilinear Forms and Inner Products} \label{ssec:bilinear_forms}

    Next, we will consider a general class of bilinear forms on $C^k(X_\epsilon)$, whose strict members are inner products compatible with the manifold inner product. This is a higher-dimensional analogue of the kernel weights used in graph Laplacian convergence~\cite{coifman_diffusion_2006, garcia_trillos_error_2020,calder_improved_2022}, where edges are assigned weights dependent on the distance between endpoints $\|x - y\|$. Here, we will define weights on $k$-simplices which depend on pairwise squared distances between vertices. For any Euclidean vector space $V$, we define the \emph{pairwise squared distance function} by
    \begin{align}
        p: V^{k+1} \to [0,\infty)^{\binom{k+1}{2}} \quad \text{where} \quad p(\by) = (\|y_i - y_j\|^2)_{i < j},
    \end{align}
    and $\by = (y_0, \ldots, y_k) \in V^{k+1}$. 

    \begin{definition} \label{def:admissible_VR_kernel}
        Let $k \in \N$. A family of maps $\kappa_\epsilon^k : (\R^N)^{k+1} \to \R$ for $\epsilon > 0$ is an \emph{admissible VR kernel of degree $k$} if it has the form
        \begin{align}
            \kappa_\epsilon^k(\by) =  \Theta\left( \frac{p(\by)}{\epsilon^2} \right) \quad \text{where} \quad \Theta : [0,\infty)^{\binom{k+1}{2}} \to [0,\infty)
        \end{align}
        is Lipschitz on $[0,1)^{\binom{k+1}{2}}$, vanishes whenever any coordinate is at least $1$, and $\Theta(p(\by))$ is invariant under the natural action of the permutation group $\Sigma_{k+1}$ on $\by$,
        \begin{align}
            \Theta(p(y_0, \ldots, y_k)) = \Theta(p(y_{\pi(0)}, \ldots, y_{\pi(k)})) \quad \text{for} \quad \pi \in \Sigma_{k+1}.
        \end{align}
        If $\Theta(\bz) > 0$ for all $\bz \in [0,1)^{\binom{k+1}{2}}$, then we say that $\kappa_\epsilon^k$ is an \emph{strict VR kernel of degree $k$}. In the case of $k=0$, we set $\kappa^0 =1$. 
    \end{definition}

    In particular, by definition in terms of the pairwise squared distance function, $\Theta$ is invariant under the diagonal orthogonal action of $A \in O(\R^N)$,
    \begin{align}
        \Theta(p(Ay_0, \ldots, Ay_k)) = \Theta(p(y_0, \ldots, y_k)).
    \end{align}

    One class of admissible VR kernels is given as follows, and will be used to study the codifferential. Let $\theta: [0,\infty) \to [0,\infty)$ be Lipschitz on $[0,1)$ and vanish on $[1,\infty)$. We define
    \begin{align}
        \kappa^k_{\theta, \epsilon} : (\R^N)^{k+1} \to \R \quad \text{by} \quad \kappa^k_{\theta, \epsilon}(\bz) = \theta \left( \frac{\diam_\R(\bz)^2}{\epsilon^2} \right).
    \end{align}
    We use these admissible VR kernels and an additional weighting function to define our desired inner products on simplicial cochains. 

    \begin{definition}
    For $k \in \N_0$, let $w:(\R^N)^{k+1} \to (c_w, \infty)$ be a symmetric smooth function bounded below by $c_w>0$ and $\kappa^k_\epsilon$ be an admissible VR kernel. Let $\epsilon > 0$, $X = (x_1, \ldots, x_n) \subset \R^N$, and $X_\epsilon$ be the corresponding VR complex. We define the following family of bilinear forms, called \emph{$w$-weighted kernel forms} on $C^k(X_\epsilon)$,  
    \begin{align} \label{eq:inner_prod}
        \langle a, b \rangle_{n,\epsilon}^{\kappa,w} \coloneqq \binom{n}{k+1}^{-1} \frac{(k!)^2}{\epsilon^{k(m+2)}}\sum_{\by \in S^k_\epsilon(X)} a(\by) b(\by)  \kappa^k_\epsilon(\by) w(\by), \quad a,b \in C^k(X_\epsilon),
    \end{align}
    where $\by = (y_0, \ldots, y_k)$. If $\kappa^k_\epsilon$ is a strict VR kernel, then we call the above \emph{$w$-weighted kernel inner products}.
    \end{definition}

    \begin{remark} \label{rem:admissible_vs_strict}
        Note that if $\kappa_\epsilon$ is only an admissible VR kernel, then there may exist simplices $\by \in S^k_\epsilon(X)$ such that $\kappa_\epsilon(\by) = 0$, and thus there exist cochains with trivial norm. Once $\kappa_\epsilon$ is a strict VR kernel, then this cannot happen. We note that our convergence results for these pairings only require $\kappa_\epsilon$ to be admissible, and not necessarily strict, whereas our codifferential and discrete Laplacian results require strictness and genuine inner products.
    \end{remark}

    We will often fix $\kappa$ and $w$, and suppress the superscripts by using $\langle \cdot , \cdot \rangle_{n, \epsilon}$ to simplify notation. Additionally, we will omit the superscript on $\kappa$ specifying the degree.

    \subsection{Probabilistic Inner Products as U-Statistics} \label{ssec:probabilistic_inner_products}
    Thus far, we have considered kernel forms on $X_\epsilon$ where $X \subset \R^N$ is a fixed finite point cloud. In our main results, we are concerned with the setting where $X = (x_1, \ldots, x_n)$ is a random collection of points which are i.i.d.~sampled from a measure $\mu$ supported on $M$. We will assume that $\mu$ has a $\cC^2$ density $\rho : M \to (0, \infty)$ with respect to the volume measure $\vol_M$ on $M$, which is bounded above and below,
    \begin{align} \label{eq:density_bound}
        \rho_{-} \leq \rho(x) \leq \rho_{+}
    \end{align}
    for fixed constants $0 < \rho_- < \rho_+$. In particular, given two differential forms $\alpha, \beta \in \Omega^k_c(\R^N)$, we wish to understand the pairing $\langle q\alpha, q\beta \rangle_{n,\epsilon}^{\kappa, w}$ as a $U$-statistic as $n \to \infty$.

    \begin{definition} \label{def:u_statistic}
        Let $h: M^{k+1} \to \R$ be a symmetric function, which we call a \emph{$U$-kernel}. Let $X = (x_1, \ldots, x_n) \in M^{n}$ be i.i.d.~samples from a probability measure $\mu$ on $M$. Then, the \emph{$U$-statistic} corresponding to $h$ is
        \begin{align}
            \U_n[h](X) \coloneqq \binom{n}{k+1}^{-1} \sum_{i_0 < \ldots < i_k} h(x_{i_0}, \ldots, x_{i_k}).
        \end{align}
    \end{definition}

    Fix a smooth symmetric weight $w: M^{k+1} \to (c_w,\infty)$ for some $c_w>0$ and an admissible VR kernel $\kappa_\epsilon$. Given $\alpha, \beta \in \Omega^k_c(\R^N)$, we define the $U$-kernel $h_\epsilon(\alpha, \beta) : M^{k+1} \to \R$ by
    \begin{align} \label{eq:euclidean_u_kernel}
        h_\epsilon(\alpha, \beta)(\by) \coloneqq \frac{(k!)^2}{\epsilon^{k(m+2)}} q\alpha(\by) q\beta(\by) \kappa_\epsilon(\by) w(\by). 
    \end{align}
    Then, by definition of the $U$-statistic and the inner product in~\eqref{eq:inner_prod}, we have
    \begin{align} \label{eq:inner_product_u_statistic}
        \langle q\alpha, q\beta \rangle_{n, \epsilon}^{\kappa, w} = \U_n[h_\epsilon(\alpha,\beta)](X).
    \end{align}
    To understand the limit as $n \to \infty$, we consider the strong law of large numbers for $U$-statistics.

    \begin{theorem}{\cite[Section 5.4, Theorem A]{serfling_approximation_1980}} \label{thm:ustat_slln}
        Let $(Z, \mu)$ be a probability space, and let $h : Z^{k+1} \to \R$ be a symmetric measurable function such that $\E_{\mu^{k+1}}[|h|] < \infty$. Let $(x_i)_{i\in\N}$ be a sequence of i.i.d.~samples from $\mu$, and write $X_n=(x_1,\ldots,x_n)$. Then, for every $n\geq k+1$,
        \begin{align}
            \E_{\mu^n}[\U_n[h](X_n)]=\E_{\mu^{k+1}}[h],
        \end{align}
        and almost surely
        \begin{align}
            \lim_{n \to \infty} \U_n[h](X_n) = \E_{\mu^{k+1}}[h].
        \end{align}
    \end{theorem}

    As an example, let $\alpha, \beta \in \Omega^k_c(\R^N)$. We define the \emph{$(k+1,\epsilon)$-fat diagonal of $M$} with respect to the ambient metric by
    \begin{align} \label{eq:fat_diagonal}
        D^k_\epsilon(M) \coloneqq \{\by \in M^{k+1} \, : \, \diam_\R(\by) < \epsilon\}.
    \end{align}
    Then, for a fixed $\epsilon > 0$, almost surely,
    \begin{align} \label{eq:base_inner_prod_limit}
        \lim_{n \to \infty} \langle q \alpha, q\beta\rangle_{n,\epsilon}^{\kappa,w} = \frac{(k!)^2}{\epsilon^{k(m+2)}} \int_{D^k_\epsilon(M)} q\alpha(\by) q\beta(\by) w(\by) \kappa_\epsilon(\by)\, d\mu^{k+1}(\by) \eqqcolon \E_{\mu^{k+1}}[h_\epsilon].
    \end{align}

    \section{Uniform Convergence of Inner Products and Codifferential} \label{sec:main_results}

    Having constructed kernel-weighted forms on cochains of sampled VR complexes, we will now state our main convergence results, along with immediate corollaries. We postpone the proof of the two main results to~\Cref{sec:uc_proof} and~\Cref{sec:codiff_proof}.

    \subsection{Uniform Convergence of Inner Products}

    We now show that these discrete pairings recover the continuum $L^2$ geometry of differential forms on $M$. The main result of this section is a uniform convergence theorem for the empirical kernel pairings of discretized forms.
    We write
    \begin{align} \label{eq:cB_ball}
        \cB_\R^k := \{\alpha \in \Omega^k_c(\R^N) : \|\alpha\|_{\cC^2(\R^N)} \leq 1\}
    \end{align}
    for the ambient $\cC^2$-unit ball of $k$-forms.

    \medskip
    \noindent\textbf{Euclidean Kernel Moment.}
        For $1\leq k\leq m$, let $V=\R^m$ and $\bu=(u_1,\ldots,u_k)\in V^k$. For an orthonormal coframe $e^1,\ldots,e^m$ of $V$, define the \emph{Euclidean kernel moment}
        \begin{align} \label{eq:inner_product_moment_constant}
            \sigma^\Theta_k
            \coloneqq
            \int_{V^k}
            \left[(e^1\wedge\cdots\wedge e^k)(u_1,\ldots,u_k)\right]^2
            \Theta\big(p(0,\bu)\big)\,d\bu,
        \end{align}
        and set $\sigma^\Theta_0=1$. Orthogonal invariance shows that~\eqref{eq:inner_product_moment_constant} is independent of the chosen coframe, so it depends only on $m$, $k$, and $\Theta$. By nonnegativity, Lipschitz continuity, and orthogonal invariance, $\sigma^\Theta_k>0$ if and only if $\Theta(p(0,\bu))>0$ for some linearly independent tuple $\bu=(u_1,\ldots,u_k)\in V^k$ for which the configuration $(0,\bu)$ lies strictly inside the VR cutoff.
    \medskip

    \begin{theorem} \label{thm:main_uc}
        Let $M\subset\R^N$ be a compact embedded $m$-manifold without boundary, and let $1\leq k\leq m$. Let $\mu$ be a probability measure on $M$ with $\cC^2$ density $\rho$ satisfying~\eqref{eq:density_bound} with respect to $\vol_M$. Let $w:M^{k+1}\to(c_w,\infty)$ be smooth and symmetric for some $c_w>0$, and let $\kappa_\epsilon$ be an admissible VR kernel with defining function $\Theta$. Set
        \begin{align}
        r_{k,m} = \min \left\{ \frac{1}{m+2}, \, \frac{2}{km+2m+2}\right\}
        \quad \text{and fix a sequence} \quad 
        \epsilon_n\asymp n^{-1/(km+2m+2)}.
        \end{align}
        There exist constants $C, c > 0$, depending on $M, k, \rho, w, \Theta$, with the following property. Let $p \in (0,1)$. With probability at least $1 - p$, for $n$ such that
        \begin{align}
            n\geq k+1,\qquad \epsilon_n<\epsilon_0,
            \qquad
            \log(8/p)\leq cn\epsilon_n^{km},
        \end{align}
        we have
        \begin{align}
            \sup_{\alpha, \beta \in \cB_\R^k} |\langle q\alpha, q\beta\rangle^{\kappa,w}_{n, \epsilon_n} - \sigma^\Theta_k \langle \alpha, \beta \rangle_M^{P_\Delta}| \leq C\left(n^{-r_{k,m}} + \sqrt{\frac{\log(1/p)}{n}} + \frac{\log(1/p)}{n^{(2m+2)/(km+2m+2)}}\right),
        \end{align}
        where $P_\Delta(x)=w(x,\ldots,x)\rho(x)^{k+1}$ and $\sigma^\Theta_k$ is the Euclidean kernel moment in~\eqref{eq:inner_product_moment_constant}. 
    \end{theorem}

    We note that for fixed $p$, except in the case of $k=m=1$, the dominant term in the bound is $n^{-2/(km+2m+2)}$. 
    The proof of the degree $k=0$ case is much simpler, and is given in~\Cref{app:degree_zero_inner_product}.

    \begin{proposition}\label{prop:main_uc_degree_zero}
        Let $M\subset\R^N$ be a compact embedded $m$-manifold without boundary. Let $\mu$ be a probability measure on $M$ with $\cC^2$ density $\rho$ satisfying~\eqref{eq:density_bound} with respect to $\vol_M$. Let $w:M\to(c_w,\infty)$ be smooth for some $c_w>0$. There exists $C = C(M, \rho, w) > 0$ with the following property. Let $p\in(0,1)$. With probability at least $1-p$, for $n\geq1$,
        \begin{align}
            \sup_{f,g\in\cB_\R^0}
            \left|
            \frac1n\sum_{i=1}^n f(x_i)g(x_i)w(x_i)
            -\int_M f(x)g(x)w(x)\rho(x)\,\dvol(x)
            \right|
            \leq C\left(
            n^{-1/(m+2)}+\sqrt{\frac{\log(1/p)}{n}}
            \right).
        \end{align}
    \end{proposition}

    \begin{remark}
        In \Cref{thm:main_uc}, the limiting weight is $P_\Delta(x)=w(x,\ldots,x)\rho(x)^{k+1}$. Thus, choosing
        \[
            w(y_0,\ldots,y_k)=\prod_{i=0}^k \rho(y_i)^{-1}
        \]
        cancels the sampling density and recovers the usual $L^2$ inner product on $M$, up to the constant $\sigma^\Theta_k$. As $\rho$ is typically unknown, an empirical weight may instead use a kernel density estimate $\hrho$;  standard uniform consistency results for KDE on manifolds controls the additional error \cite{coifman_diffusion_2006}.
    \end{remark}

    \subsection{Codifferential Convergence}
    \phantomsection\label{not:discrete_codifferential}
    To relate the continuous Hodge Laplacian on $\Omega^k(M)$ and the discrete Hodge Laplacian on $C^k(X_\epsilon)$, we study the relationship between the smooth and discrete codifferentials
    \begin{align}
        \delta: \Omega^k(M) \to \Omega^{k-1}(M) \andd \hdelta: C^k(X_\epsilon) \to C^{k-1}(X_\epsilon),
    \end{align}
    which are defined as the adjoints of the exterior derivative $d: \Omega^{k-1}(M) \to \Omega^k(M)$~\eqref{eq:smooth_codifferential} and the simplicial coboundary map $\hd: C^{k-1}(X_\epsilon) \to C^k(X_\epsilon)$ with respect to the corresponding smooth and discrete inner products. Note that we exclusively work with \emph{strict VR kernels} such that the resulting bilinear form is indeed an inner product (see~\Cref{rem:admissible_vs_strict}). Furthermore, we work in the unweighted uniform-density case, $w \equiv 1$ and $\rho \equiv \rho_0=\vol_M(M)^{-1}$, in order to keep the notation focused on the codifferential argument.

    We also restrict our attention to forms $\alpha \in \Omega^k(U_\epsilon)$ on a tubular neighbourhood $U_\epsilon$ such that $\alpha = \pi^* \talpha$ for some $\talpha \in \Omega^k(M)$; we refer to such forms as \emph{pullback forms}. Along $M$, we identify $\alpha$ with $\talpha$ and write $\delta\alpha$ for the intrinsic codifferential of $\talpha$. Since $\epsilon<\epsilon_0$, the Euclidean simplices considered below lie in the tubular neighbourhood where these pullback forms are defined. Finally, we work with a certain class of strict VR kernels which are bounded below on the open Vietoris--Rips scale. \label{not:uniform_pullback_setting}
    We write
    \begin{align} \label{not:intrinsic_pullback_unit_balls}
        \cB_M^k = \{\talpha \in \Omega^k(M): \|\talpha\|_{\cC^2(M)}\leq 1\} \andd \pi^*\cB_M^k = \{\pi^*\talpha: \talpha\in \cB_M^k\}
    \end{align}
    for the intrinsic $\cC^2$-unit ball of $k$-forms and the corresponding class of pullback forms.

    \begin{definition} \label{def:codiff_kernel}
        \phantomsection\label{not:codiff_kernel_table}
        For any finite tuple $\bz$ in a Euclidean vector space $V$, let
        \begin{align} \label{eq:s_squared_diameter_convention}
            s(\bz)\coloneqq\diam_V(\bz)^2,
        \end{align}
        with the convention that the diameter of a singleton is zero. Let $\kappa_{\theta,\epsilon}^0 = 1$. For $k \geq 1$, let $\theta: [0,\infty) \to [0,\infty)$ be a non-increasing function, Lipschitz on $[0,1)$, equal to zero on $[1,\infty)$, satisfying $\theta(0)=1$ and $c_\theta = \inf_{t \in [0,1)} \theta(t) > 0$. For $k\geq 1$ and $\bz=(z_0,\ldots,z_k)\in (\R^N)^{k+1}$, define
        \begin{align} \label{eq:codiff_kappa_def}
            \kappa^k_{\theta,\epsilon}(\bz) = \theta\left(\frac{s(\bz)}{\epsilon^2}\right).
        \end{align}
    \end{definition}
    For each $k$, the kernels $\kappa^k_{\theta,\epsilon}$ are strict VR kernels because $\theta$ is bounded below on $[0,1)$. 

    \medskip
    \noindent\textbf{Euclidean Codifferential Moment.}
        For $1\leq k\leq m$, let $V=\R^m$ and equip
        \begin{align}
            H_{m,k}=\left\{\bu=(u_1,\ldots,u_k)\in V^k:\sum_{i=1}^k u_i=0\right\}
        \end{align}
        with its induced Euclidean measure. For $\bu\in H_{m,k}$ and $w\in V$ we define $s(\bu)=\diam_V(\bu)^2$ and $s(w,\bu)=\diam_V(w,\bu)^2$.
        For $s(\bu)<1$, let $e^\perp\in\spann\{u_1,\ldots,u_k\}^\perp$ be any unit vector and define
        \begin{align}
            \lambda_\theta(\bu)
            \coloneqq
            \frac{1}{\sqrt{\theta(s(\bu))}}
            \int_{\{w\in V:s(w,\bu)<1\}}
            \langle w,e^\perp\rangle^2\theta(s(w,\bu))\,dw.
        \end{align}
        This quantity is independent of the choice of $e^\perp$. For $k \geq 2$, we define 
        \begin{align} \label{eq:codifferential_moment_constant}
            \sigma^\theta_{\delta,k}
            \coloneqq
            \frac{k^{(m+4)/2}}{\binom{m}{k-1}}
            \int_{\{\bu\in H_{m,k}:s(\bu)<1\}}
            \left\|(u_2-u_1)\wedge\cdots\wedge(u_k-u_1)\right\|^2
            \lambda_\theta(\bu)^2\,d\bu.
        \end{align}
        This constant depends only on $m$, $k$, and $\theta$. Furthermore, since $H_{m,1}=\{0\}$ and $\theta(0)=1$, we set
        \begin{align}
            \sigma^\theta_{\delta,1}
            =
            \left(\int_{\|w\|<1}\langle w,e\rangle^2\theta(\|w\|^2)\,dw\right)^2
        \end{align}
        for any unit vector $e\in\R^m$, consistently with $\kappa^0_{\theta,\epsilon}=1$.
    \medskip

    The main result is the following uniform convergence for the codifferential.

    \begin{theorem} \label{thm:main_codiff}
        Let $M\subset\R^N$ be a compact embedded $m$-manifold without boundary, let $1\leq k\leq m$, and assume that the sampling density is uniform, $\rho\equiv \rho_0=\vol_M(M)^{-1}$. Let $\theta$ satisfy~\Cref{def:codiff_kernel}. Set
        \begin{align}
            r^\delta_{k,m}=\frac{2}{2mk+5m+6}
            \quad \text{and fix a sequence} \quad 
            \epsilon_n\asymp n^{-r^\delta_{k,m}}.
        \end{align}
        Let $\sigma^\theta_{\delta,k}$ be the Euclidean codifferential moment in~\eqref{eq:codifferential_moment_constant}.
        There exist constants $C,c > 0$, depending on $M, k, \theta$, with the following property. Let $p\in(0,1)$. With probability at least $1-p$, for $n$ such that
        \begin{align}
            n\geq\max\{2k,k+2\},\qquad
            \epsilon_n<\epsilon_0,\qquad
            \log(12/p)\leq cn\epsilon_n^{(k+1)m},
        \end{align}
        writing $\alpha = \pi^* \talpha$ and $\beta = \pi^* \tilde{\beta}$ for the corresponding pullback forms, and setting $\kappa = \kappa^{k-1}_{\theta, \epsilon_n}$, we have
        \begin{align}
            \sup_{\tilde\alpha,\tilde\beta\in\cB_M^k}
            \left|\left\langle \hdelta q\alpha,\hdelta q\beta\right\rangle^{\kappa}_{n,\epsilon_n}-\rho_0^{k+2}\sigma^\theta_{\delta,k}\langle\delta\tilde\alpha,\delta\tilde\beta\rangle_M\right| \leq C\left(n^{-r^\delta_{k,m}}+\frac{\sqrt{\log(1/p)}}{n^{1/2-r^\delta_{k,m}}}+\frac{\log(1/p)}{n^{(3m+2)/(2mk+5m+6)}}\right).
        \end{align}
    \end{theorem}

    \begin{remark}
        By bilinearity of the inner products, our main results provide uniform estimates. In particular, if $R_{n,p}$ denotes the relevant rate in~\Cref{thm:main_uc}, then for all $\alpha,\beta\in\Omega^k_c(\R^N)$,
        \begin{align}
            |\langle q\alpha,q\beta\rangle^{\kappa,w}_{n,\epsilon_n}-\sigma^\Theta_k\langle\alpha,\beta\rangle_M^{P_\Delta}|
            \ls R_{n,p}\|\alpha\|_{\cC^2(\R^N)}\|\beta\|_{\cC^2(\R^N)}.
        \end{align}
        and if $R^\delta_{n,p}$ denotes the rate in~\Cref{thm:main_codiff}, then all $\tilde\alpha,\tilde\beta\in\Omega^k(M)$ and their pullbacks satisfy
        \begin{align}
            \left|\left\langle\hdelta q\alpha,\hdelta q\beta\right\rangle^\kappa_{n,\epsilon_n}
            -\rho_0^{k+2}\sigma^\theta_{\delta,k}\langle\delta\tilde\alpha,\delta\tilde\beta\rangle_M\right|
            \ls R^\delta_{n,p}\|\tilde\alpha\|_{\cC^2(M)}\|\tilde\beta\|_{\cC^2(M)}.
        \end{align}
    \end{remark}

    The optimized bandwidths in~\Cref{thm:main_uc,thm:main_codiff} are generally different, whereas the applications in~\Cref{sec:rips_laplacian} require simultaneous convergence in probability of the inner products in adjacent degrees and the codifferential pairing at a common bandwidth; the following corollary (proved in~\Cref{sec:codiff_proof}) records this consequence of the unoptimized estimates in their proofs.

    \begin{corollary} \label{cor:common_bandwidth_convergence}
        Fix $0\leq k\leq m-1$ and assume the hypotheses of~\Cref{thm:main_codiff}, using $w\equiv1$ and the kernels $\kappa^r_{\theta,\epsilon}$ in each degree. Let
        \begin{align}
            \epsilon_n=n^{-a}, \qquad 0<a<\frac{1}{m(k+2)+2}.
        \end{align}
        Then, at the common bandwidth $\epsilon_n$, the degree $k$ and $(k+1)$ inner products and, when $k\geq1$, the degree-$k$ codifferential pairing converge simultaneously in probability to the respective continuum limits identified in~\Cref{thm:main_uc},~\Cref{prop:main_uc_degree_zero}, and~\Cref{thm:main_codiff} uniformly over the corresponding unit balls.
    \end{corollary}

    \subsection{Topology of the VR Complex} \label{ssec:rips_topology}

    The convergence results above provide analytic statements about the geometric properties of cochains on the sampled VR complex. A key advantage of our discretization framework is the realization of these discretized forms as cochains on an underlying simplicial complex. In particular, we show in this section that the scale parameters $\epsilon_n$ required for convergence result in VR complexes $X_{\epsilon_n}$ which are homotopy equivalent to $M$; and furthermore, that the discretization map $q$ restricted to $M$ yields a quasi-isomorphism. 

    There has been a series of works which have studied the topological properties of VR complexes, beginning with Hausmann~\cite{hausmann1995vietorisrips}. For $\epsilon > 0$, let $M_\epsilon$ denote the \emph{intrinsic} Vietoris-Rips complex of $M$, a simplicial complex with $M$ as its vertex set and $\sigma \subset M$ with $|\sigma| = k+1$ is a $k$-simplex if $\diam_M(\sigma) < \epsilon$. By placing a total order on the elements of $M$, Hausmann defined a map $T: |M_\epsilon| \to M$, where $|M_\epsilon|$ is the geometric realization of $M_\epsilon$ as follows. On an ordered simplex $\sigma=[p_0,\ldots,p_k]$, the map is defined inductively as the geodesic cone from the first vertex $p_0$. If $z=\sum_{i=0}^k t_i p_i$ and $t_0<1$, set $z'=(1-t_0)^{-1}\sum_{i=1}^k t_i p_i$ and let $T(z)=\gamma(1-t_0)$, where $\gamma$ is the minimizing geodesic from $p_0$ to $T(z')$; if $t_0=1$, set $T(z)=p_0$. We consider the statement of Hausmann's theorem from~\cite[Theorem 4]{majhi_demystifying_2024}, which uses the convexity radius from~\eqref{eq:convexity_radius}.
    
    \begin{theorem}{\cite{hausmann1995vietorisrips}}
        For any $0 < \epsilon < \cx(M)$, the map $T$ is a homotopy equivalence.
    \end{theorem}

    This result provides homotopy equivalence for the intrinsic VR complex. Later on,~\cite{latschev_vietoris-rips_2001} provided an analogous result for VR complexes built from a finite sample $X$ with small Hausdorff distance from $M$, and the specific scale parameters required were quantified in~\cite{majhi_demystifying_2024}. For our results, we consider the ambient affine realization $A: |X_\epsilon| \to \R^N$ of our sampled VR complex. We note that if $\epsilon < \epsilon_0$, then $A(|X_\epsilon|) \subset U_\epsilon$, where $U_\epsilon$ is the tubular neighbourhood of $M$. Define
    \begin{equation}
        F = \pi \circ A : |X_\epsilon| \to M,
    \end{equation}
    and we will adapt the result from~\cite{majhi_demystifying_2024} to our setting.

    \begin{lemma} \label{lem:affine_projection_homotopy_equivalence}
        Assume that $M$ is connected.  Fix $0<\zeta<1/14$.  Suppose $h=d_H^\R(M,X)<\zeta\epsilon$, where $d_H^\R$ is the Hausdorff distance in $\R^N$ and  $\epsilon<\epsilon_0$. Furthermore, suppose $\epsilon \leq A_\zeta \tau_M$, where $A_\zeta > 0$ is an explicit constant given in~\cite[Theorem 18]{majhi_demystifying_2024}.
        Then $F:|X_\epsilon|\to M$ is a homotopy equivalence.
    \end{lemma}

    \begin{proof}
        The Hausdorff-distance condition induces a vertex map $\phi:M\to X$ satisfying $\|p-\phi(p)\|<\zeta\epsilon$, not necessarily continuous or unique. This extends to a simplicial map
        \begin{align}
            \Phi:M_{(1-2\zeta)\epsilon}\to X_\epsilon,
        \end{align}
        whose realization $|\Phi|$ is a homotopy equivalence~\cite[Theorem~18]{majhi_demystifying_2024}  and where $M_{(1-2\zeta)\epsilon}$ is the VR complex of $M$ at scale $(1-2\zeta)\epsilon$.  Let $T:|M_{(1-2\zeta)\epsilon}|\to M$ be Hausmann's geodesic realization. Our aim is to show that $F\circ |\Phi|\simeq T$.

        Let $\sigma=[p_0,\ldots,p_k]$ be a simplex of $M_{(1-2\zeta)\epsilon}$.  The map $T$ sends $|\sigma|$ into $\ball{M}{p_0}{(1-2\zeta)\epsilon}\subset \ball{M}{p_0}{S(2\epsilon)}$.  For $z\in |\sigma|$, $A(|\Phi|(z))$ lies in the convex hull of the points $\phi(p_i)$, and
        \begin{align}
            \|\phi(p_i)-p_0\|\leq \|\phi(p_i)-p_i\|+d_M(p_i,p_0)<(1-\zeta)\epsilon.
        \end{align}
        Hence $A(|\Phi|(z))\in \ball{\R^N}{p_0}{(1-\zeta)\epsilon}$.  Because the projection moves $A(|\Phi|(z))$ by less than $\epsilon$, we have $\|F(|\Phi|(z))-p_0\|<2\epsilon$.  By~\eqref{eq:geodesic_euclidean_balls}, $F(|\Phi|(z))\in \ball{M}{p_0}{S(2\epsilon)}$.  As we assume $S(2\epsilon)<\cx(M)$, this ball is strongly geodesically convex.  The unique minimizing geodesics from $F(|\Phi|(z))$ to $T(z)$ therefore define simplexwise homotopies which agree on faces, giving $F\circ |\Phi|\simeq T$.  Finally, since $|\Phi|$ and $T$ are homotopy equivalences, so is $F$.
    \end{proof}

    Now we can show that $F$ is a homotopy equivalence with high probability within the convergence regimes of~\Cref{thm:main_uc} and~\Cref{thm:main_codiff}. This is done by using an intrinsic net and a union bound to show that $d_H^\R(M,X^{(n)})<\zeta\epsilon_n$ with probability tending to one, then applying \Cref{lem:affine_projection_homotopy_equivalence} to give the homotopy equivalence.
    \begin{proposition} \label{prop:topology_quasi_iso_convergence_scales}
        Assume that $M$ is connected.  Let $X^{(n)}=(x_1,\ldots,x_n)\subset M$ be sampled i.i.d.~from a density satisfying~\eqref{eq:density_bound}.  Let $\epsilon_n\to0$ be deterministic and suppose that
        \begin{align} \label{eq:topology_bandwidth_condition_sec43}
            \frac{n\epsilon_n^m}{\log(1/\epsilon_n)}\to\infty.
        \end{align}
        Let $A_n:|X^{(n)}_{\epsilon_n}|\to \R^N$ be the affine realization map. For all sufficiently large $n$, its image lies in $U_{\epsilon_n}$, and we set $F_n=\pi\circ A_n:|X^{(n)}_{\epsilon_n}|\to M$.  Then, %
        \begin{align}
            \lim_{n \to \infty} \PP\left[F_n \text{ is a homotopy equivalence}\right]
            =1.
        \end{align}
        On this event, $|X^{(n)}_{\epsilon_n}|\simeq M$ and the following pullback discretization is a quasi-isomorphism,
        \begin{align}
            q_n^\pi:\Omega^\bullet(M)\to C^\bullet(X^{(n)}_{\epsilon_n}),
            \qquad
            (q_n^\pi\omega)(\sigma)=\int_{\Delta^k}\sigma_\sigma^*\pi^*\omega,
        \end{align}
    \end{proposition}

    \begin{proof}
        Fix $0<\zeta<1/14$. By compactness, uniformly for $x\in M$ and $0<r<\inj(M)$, we have $r^m\ls\vol_M(\ball{M}{x}{r})$, and we may take an intrinsic $r/2$-net $z_1,\ldots,z_L$ with $L\ls r^{-m}$. The lower density bound gives a single constant $c>0$, depending only on $M$ and $\rho_-$, such that $\mu(\ball{M}{z_i}{r/2})\geq cr^m$ for every $i$ and every such $r$. If every such ball contains a sample point, then for every $p\in M$ there are $z_i$ and $x_j\in X^{(n)}$ with $d_M(p,z_i)<r/2$ and $d_M(z_i,x_j)<r/2$. Since $d_\R\leq d_M$ on $M$, this implies $d_H^\R(M,X^{(n)})<r$. By independence, we have
        \begin{equation}
        \PP\left[X^{(n)}\cap \ball{M}{z_i}{r/2}=\emptyset\right] = (1 - \mu(\ball{M}{z_i}{r/2}))^n \leq \exp(-n\mu(\ball{M}{z_i}{r/2}))\leq \exp(-c n r^m).
        \end{equation}
        Then, by a union bound, we obtain
        \begin{align} \label{eq:homotopy_equiv_probability}
            \PP\left[d_H^\R(M,X^{(n)})\geq r\right]
            \leq \sum_{i=1}^L \PP\left[X^{(n)}\cap \ball{M}{z_i}{r/2}=\emptyset\right]
            \ls r^{-m}\exp(-cnr^m).
        \end{align}
        Set $r=\zeta\epsilon_n$, and since $\epsilon_n\to0$, for all sufficiently large $n$ we have $\epsilon_n, r<\epsilon_0$, $\epsilon_n\leq A_\zeta\tau_M$, and $A_n(|X^{(n)}_{\epsilon_n}|)\subset U_{\epsilon_n}$, so $F_n$ is defined.  On the complementary event, the final hypothesis of~\Cref{lem:affine_projection_homotopy_equivalence} that $d_{H}^\R(M, X^{(n)}) < \zeta\epsilon_n$ holds, so $F_n$ is a homotopy equivalence. Then,~\eqref{eq:topology_bandwidth_condition_sec43} implies that the probability in~\eqref{eq:homotopy_equiv_probability} converges to $0$.

        For the quasi-isomorphism, let $\sigma$ be an oriented simplex of $X^{(n)}_{\epsilon_n}$, and let $\iota_\sigma:\Delta^k\to |X^{(n)}_{\epsilon_n}|$ be its standard affine parametrization. By definition, $A_n\circ\iota_\sigma=\sigma_\sigma$, and hence $F_n\circ\iota_\sigma=\pi\circ\sigma_\sigma$, so
        \begin{align}
            (q_n^\pi\omega)(\sigma)=\int_{\Delta^k}\sigma_\sigma^*\pi^*\omega
            =\int_{\Delta^k}\iota_\sigma^*F_n^*\omega.
        \end{align}
        Thus $q_n^\pi$ is the simplicial de Rham integration map applied to $F_n^*\omega$.  Because $F_n$ is a homotopy equivalence, $F_n^*$ is a quasi-isomorphism; and  since simplicial de Rham integration is a quasi-isomorphism by the simplicial de Rham theorem~\cite{Whitney1957}, so is $q_n^\pi$.
    \end{proof}

    Finally, \Cref{prop:topology_quasi_iso_convergence_scales} applies at the $\epsilon$ scales used in the analytic convergence theorems. If $\epsilon_n\asymp n^{-a}$, then~\eqref{eq:topology_bandwidth_condition_sec43} is equivalent to $am<1$. The choices in~\Cref{thm:main_uc} and~\Cref{thm:main_codiff} have $a=1/(km+2m+2)$ and $a=2/(2mk+5m+6)$, respectively, and both satisfy $am<1$. Thus, for connected $M$ and fixed $p \in (0,1)$, both conclusions hold simultaneously with high probability. 

    \section{The Vietoris-Rips Laplacian} \label{sec:rips_laplacian}

    The Laplacian on the VR complex is the spectral object attached to the cochain geometry constructed in the previous sections.  The estimates in \Cref{thm:main_uc,thm:main_codiff} give high-probability error bounds for the normalized discrete inner products and codifferential pairings after the prescribed choice of bandwidth.  We record the resulting quantitative consequences for eigenvalue bounds, harmonic representatives, and circular coordinates.

    Throughout this section we work in the unweighted uniform setting of the main codifferential \Cref{thm:main_codiff}: $w\equiv 1$ and $\rho\equiv\rho_0=\vol_M(M)^{-1}$. We denote the point cloud by $X^{(n)}=(x_1,\ldots,x_n)\subset M$.  When $\alpha \in \Omega^k(M)$ is an intrinsic form, we write $q \alpha \coloneq q_n^\pi \alpha$ to simplify notation. For the ambient inner-product estimates, we use a fixed $\cC^2$-bounded linear compactly supported extension of $\pi^*\alpha$ that agrees with it on $U_{\epsilon_0}$; this does not change $q_n^\pi\alpha$ when $\epsilon_n<\epsilon_0$. We fix $\theta$ and kernels $\kappa^k_{\theta, \epsilon}$ satisfying \Cref{def:codiff_kernel}, but consider \emph{normalized inner products} \label{not:normalized_inner_products} only for $0\leq k\leq m$ (recalling that $\kappa^0 = 1$)
    \begin{equation}
        \langle a, b \rangle_{n,\epsilon} \coloneqq  \frac{1}{\rho_0^{k+1} \sigma_k^\theta}\binom{n}{k+1}^{-1} \frac{(k!)^2}{\epsilon^{k(m+2)}}\sum_{\by \in S^k_\epsilon(X^{(n)})} a(\by) b(\by)  \kappa^k_{\theta,\epsilon}(\by) , \quad a,b \in C^k(X^{(n)}_\epsilon),
    \end{equation}
    where, for $1\leq k\leq m$, $\sigma_k^\theta$ denotes the moment $\sigma_k^\Theta$ in~\eqref{eq:inner_product_moment_constant} for $\Theta((t_{ij}))=\theta(\max_{i<j}t_{ij})$, and $\sigma_0^\theta=1$.
    When $\epsilon=\epsilon_n$, we write $\langle a,b\rangle_n\coloneqq\langle a,b\rangle_{n,\epsilon_n}$ and $\|a\|_n\coloneqq\langle a,a\rangle_n^{1/2}$.
    However, the constant $\sigma^\theta_{\delta,k}>0$ from the codifferential convergence in~\Cref{thm:main_codiff} depends on normalizations in adjacent degrees, so we have the residual constant
    \begin{equation} \label{eq:residual_codiff_constant}
        \bar{\sigma}_{\delta,k} \coloneqq \frac{\sigma^\theta_{k-1}}{(\sigma^\theta_k)^2} \sigma^\theta_{\delta,k},
    \end{equation}
    such that the normalized codifferential pairing converges to $\bar{\sigma}_{\delta,k} \langle \delta \alpha, \delta \beta \rangle_M$.
    Throughout this section, for the degree-$k$ Laplacian we use the common bandwidth
    \begin{equation} \label{eq:combined_bandwidth}
        \epsilon_n=n^{-a} \quad \text{satisfying} \quad 0<a<\frac{1}{m(k+2)+2}.
    \end{equation}
    By~\Cref{cor:common_bandwidth_convergence}, this gives simultaneous convergence in probability of the inner products in degrees $k$ and $k+1$ and, when $k\geq1$, the degree-$k$ codifferential pairing.

    \subsection{Laplacians and Rayleigh Quotients}

    Recall the smooth Hodge Laplacian from~\eqref{eq:smooth_laplacian}
    \begin{equation}
        \Delta_k=d\delta+\delta d \quad \text{where} \quad 0\leq \lambda^{(k)}_0\leq \lambda^{(k)}_1\leq \cdots
    \end{equation}
    denotes the eigenvalues of $\Delta_k$, repeated with multiplicity and listed in nondecreasing order. The \emph{smooth Hodge energy} and the corresponding smooth Rayleigh quotient are defined by 
    \begin{equation} \label{eq:smooth_rayleigh_sec7}
        Q_k(\alpha)
        := \langle \Delta_k\alpha,\alpha\rangle_M
        = \|d\alpha\|_M^2+\|\delta\alpha\|_M^2 \andd R_k(\alpha):=\frac{Q_k(\alpha)}{\|\alpha\|_M^2} \quad \text{for} \quad  \alpha\neq 0,
    \end{equation}
    \phantomsection\label{not:smooth_rayleigh}
    where the $\delta$ term in $Q_k$ is omitted for $k = 0$.  
    By the usual Courant--Fisher characterization,
    \begin{equation} \label{eq:smooth_minmax_sec7}
        \lambda^{(k)}_j
        =
        \min_{\substack{V\subset \Omega^k(M)\\ \dim V=j+1}}
        \max_{0\neq \alpha\in V} R_k(\alpha).
    \end{equation}
    Next, we define the corresponding discrete objects. Let
    \begin{equation}
        \hd_n:C^\bullet(X^{(n)}_{\epsilon_n})\to C^{\bullet+1}(X^{(n)}_{\epsilon_n})
        \andd
        \hdelta_n : C^r(X^{(n)}_{\epsilon_n}) \to C^{r-1}(X^{(n)}_{\epsilon_n}), \qquad 1\leq r\leq m,
    \end{equation}
    denote the simplicial coboundary and its adjoint.
    For $0\leq k\leq m-1$, we define the \emph{normalized Rips Laplacian} by
    \begin{equation} \label{eq:rips_laplacian}
        \widehat\Delta_{k,n}=\hdelta_n\hd_n+\bar{\sigma}_{\delta,k}^{-1}\hd_n\hdelta_n \quad \text{where} \quad 0\leq \widehat\lambda^{(k)}_{0,n}\leq \widehat\lambda^{(k)}_{1,n}\leq \cdots
    \end{equation}
    are the eigenvalues of $\widehat\Delta_{k,n}$, repeated with multiplicity.  The second term is omitted when $k=0$. 

\begin{remark}
Let $\Delta^{\mathrm{std}}_{k,n}:=\hat\delta_n\hat d_n+\hat d_n\hat\delta_n$
be the unnormalized Hodge Laplacian formed using the same weighted inner products. The normalized operator $
\widehat\Delta_{k,n}
=\hat\delta_n\hat d_n+\bar{\sigma}_{\delta,k}^{-1}\hat d_n\hat\delta_n$
only rescales the down-Laplacian term. Indeed, if
$U=\hat\delta_n\hat d_n$ and $D=\hat d_n\hat\delta_n$, then
$UD=DU=0$, and the finite-dimensional Hodge decomposition splits \(C^k\) into harmonic,
exact, and coexact summands. On these summands, \(U\) and \(D\) act separately. Hence
\(\Delta^{\mathrm{std}}_{k,n}\) and \(\widehat\Delta_{k,n}\) admit a common orthonormal eigenbasis:
the normalization rescales the exact eigenvalues by \(\bar{\sigma}_{\delta,k}^{-1}\), leaves the
coexact and harmonic eigenvalues unchanged.
\end{remark}
    For $0\leq k\leq m-1$ and $c\in C^k(X^{(n)}_{\epsilon_n})$, we define the \emph{discrete Hodge energy} and the \emph{discrete Rayleigh quotient} by
    \begin{equation} \label{eq:discrete_hodge_energy_sec7}
        \widehat Q_{k,n}(c)
        :=
        \|\hd_n c\|_n^2
        +
        \bar{\sigma}_{\delta,k}^{-1}\|\hdelta_n c\|_n^2 \andd \widehat R_{k,n}(c):=\frac{\widehat Q_{k,n}(c)}{\|c\|_n^2}
        \quad \text{for} \quad  c\neq 0.
    \end{equation}
    \phantomsection\label{not:discrete_hodge_energy}
    The second term in $\widehat{Q}_{k,n}$ is omitted for $k=0$.  
    Since $C^k(X^{(n)}_{\epsilon_n})$ is finite dimensional, the usual matrix spectral theorem gives
    \begin{equation} \label{eq:discrete_minmax_sec7}
        \widehat\lambda^{(k)}_{j,n}
        =
        \min_{\substack{W\subset C^k(X^{(n)}_{\epsilon_n})\\ \dim W=j+1}}
        \max_{0\neq c\in W}\widehat R_{k,n}(c).
    \end{equation}

    \subsection{Spectral Upper Bounds}

    The first step is uniform convergence of Rayleigh quotients on each fixed finite-dimensional space of smooth forms.

    \begin{lemma} \label{lem:finite_dimensional_rayleigh_sec7}
    Let $0\leq k\leq m-1$ and $V\subset \Omega^k(M)$ with $\dim(V) < \infty$. With probability tending to one, $q|_V$ is injective and
    \begin{equation} \label{eq:finite_dimensional_rayleigh_sec7}
        \sup_{0\neq \alpha\in V}
        \left|\widehat R_{k,n}(q\alpha)-R_k(\alpha)\right|
        \xrightarrow{\PP}0.
    \end{equation}
    \end{lemma}

    \begin{proof}
    Because $V$ is finite dimensional, all smooth norms are bounded on the unit sphere $\{\alpha\in V:\|\alpha\|_M=1\}$.  The norm convergence
    \[
        \sup_{\substack{\alpha\in V\\ \|\alpha\|_M=1}}
        \left|\|q\alpha\|_n^2-1\right|
        \xrightarrow{\PP}0
    \]
        follows from~\Cref{cor:common_bandwidth_convergence}, the extension convention above, and the normalization of the degree $k$ inner product.  On the event that this supremum is less than $1/2$, every $\alpha\in V$ with $\|\alpha\|_M=1$ satisfies $\|q\alpha\|_n^2\geq 1/2$.  Hence $q|_V$ is injective on an event whose probability tends to one.

    For the exterior derivative term in the Hodge energy $Q$, \Cref{lem:cochain_map} gives
    \[
        \hd_n(q\alpha)=q(d\alpha).
    \]
    Applying the degree-$(k+1)$ conclusion of~\Cref{cor:common_bandwidth_convergence}, uniformly for $d\alpha$ with $\alpha$ in the unit sphere of $V$, gives
    \[
        \|\hd_n(q\alpha)\|_n^2
        \xrightarrow{\PP}
        \|d\alpha\|_M^2
    \]
    uniformly on that sphere.  For the codifferential term, when $m \geq k \geq 1$, applying the degree-$k$ conclusion of~\Cref{cor:common_bandwidth_convergence}, we get
    \[
        \bar{\sigma}_{\delta,k}^{-1}\|\hdelta_n(q\alpha)\|_n^2
        \xrightarrow{\PP}
        \|\delta\alpha\|_M^2
    \]
    uniformly on the unit sphere of $V$. Combining the two energy terms gives
    \[
        \sup_{\substack{\alpha\in V\\ \|\alpha\|_M=1}}
        \left|\widehat Q_{k,n}(q\alpha)-Q_k(\alpha)\right|
        \xrightarrow{\PP}0.
    \]
    Finally, the norm convergence above implies that $\|q\alpha\|_n^2$ is bounded away from zero uniformly on the unit sphere with probability tending to one; dividing the convergent energies by the convergent denominators proves \eqref{eq:finite_dimensional_rayleigh_sec7}.
    \end{proof}

    \begin{theorem} \label{thm:rips_spectral_upper_bound_sec7}
    For every fixed $j\in\N_0$, $0\leq k\leq m-1$, and $\eta>0$,
    \begin{equation} \label{eq:rips_spectral_upper_bound_sec7}
        \PP\left[\widehat\lambda^{(k)}_{j,n}\leq \lambda^{(k)}_j+\eta\right]\to 1.
    \end{equation}
    \end{theorem}

    \begin{proof}
    Let $V_j\subset \Omega^k(M)$ be the span of the first $j+1$ smooth eigenforms of $\Delta_k$, counted with multiplicity.  By \eqref{eq:smooth_minmax_sec7},
    \[
        \max_{0\neq \alpha\in V_j}R_k(\alpha)=\lambda^{(k)}_j.
    \]
    By \Cref{lem:finite_dimensional_rayleigh_sec7}, with probability tending to one,
    \[
        \max_{0\neq \alpha\in V_j}\widehat R_{k,n}(q\alpha)
        \leq
        \lambda^{(k)}_j+\eta.
    \]
    By the injectivity of $q|_{V_j}$ in \Cref{lem:finite_dimensional_rayleigh_sec7}, we also get that $\dim(q(V_j)) = j+1$ with probability tending to one.
    On the intersection of these two events, the discrete min-max formula \eqref{eq:discrete_minmax_sec7} gives
    \[
        \widehat\lambda^{(k)}_{j,n}
        \leq
        \max_{0\neq c\in q(V_j)}\widehat R_{k,n}(c)
        =
        \max_{0\neq \alpha\in V_j}\widehat R_{k,n}(q\alpha)
        \leq
        \lambda^{(k)}_j+\eta.
    \]
    This proves the claim.
    \end{proof}

    \subsection{Harmonic 1-Forms} \label{ssec:harmonic_1_forms}
    In the preceding section, we considered upper bounds for the eigenvalues of the discrete Hodge Laplacian. In order to obtain more precise spectral relationships between continuous and discrete Hodge Laplacians, we require interpolation methods to translate simplicial cochains back to differential forms. As this is beyond the scope of the present article, we focus our attention on 1-Laplacians, where we can leverage known spectral convergence results about graph Laplacians to prove convergence of harmonic 1-forms. 
    
    Throughout the remainder of this section, we assume that $m\geq2$ and $M$ is connected and orientable, and our main tool is the spectral convergence of the graph Laplacian (the $k=0$ Laplacian) to the Laplace-Beltrami operator in~\cite{calder_improved_2022}.
    Let $\theta$ be as in \Cref{def:codiff_kernel}.  
    Define $\hat{L}_n:C^0(X^{(n)}_{\epsilon_n})\to C^0(X^{(n)}_{\epsilon_n})$ to be the unnormalized weighted graph Laplacian
    \begin{equation} \label{eq:graph_laplacian_sec7}
        (\hat{L}_n a)(x_i)
        :=
        \frac{1}{n\epsilon_n^{m+2}}
        \sum_{j=1}^n
        \theta\left(\frac{\|x_i-x_j\|^2}{\epsilon_n^2}\right)
        (a(x_i)-a(x_j)), \quad a \in C^0(X^{(n)}_{\epsilon_n}),
    \end{equation}
    as defined in~\cite[Equation 2.3]{calder_improved_2022}. The degree $0$ Laplacian $\hat{\Delta}_{0,n}$ is defined in terms of our inner product in~\eqref{eq:inner_prod}, which sums over oriented edges. By direct computation, 
    \begin{equation} \label{eq:laplacian_operator_identity}
        \widehat\Delta_{0,n}
        =
        \frac{2n}{(n-1)\rho_0\sigma_1^\theta}
        \hat{L}_n
        \quad\text{on } C^0(X^{(n)}_{\epsilon_n}).
    \end{equation}
    Thus the two graph Laplacians have the same kernel, and their Rayleigh quotients differ by a scalar, which is bounded above and below uniformly in $n$.

    \begin{remark} \label{rem:open_closed_graph_sec7}
    The graph Laplacian results of~\cite{calder_improved_2022} use a closed $\epsilon$-graph, with an edge whenever $\|x_i-x_j\|\leq \epsilon$, while our Vietoris--Rips convention uses the open condition $\|x_i-x_j\|<\epsilon$. Thus,~\cite{calder_improved_2022} uses a radial profile $\eta: [0, \infty) \to [0,\infty)$ which is Lipschitz on $[0,1]$ and may be positive at the endpoint. To apply their theorem, we define the radial profile by $\eta(r) = \theta(r^2)$ whenever $r \neq 1$ and set $\eta(1) = \lim_{r\uparrow 1}\theta(r)$. For any $\epsilon > 0$, the event $\|x_i-x_j\|=\epsilon$ has probability zero for each pair $(i,j)$, and hence the closed graph Laplacian built from $\eta$ agrees almost surely with $\hat{L}_n$. Thus the spectral convergence theorem applies to $\hat{L}_n$ through this endpoint-modified profile.
    \end{remark}

    \begin{proposition} \label{prop:degree_zero_gap_sec7}
    Suppose $M$ is connected and orientable, and $\epsilon_n$ satisfies~\eqref{eq:combined_bandwidth}.  For the eigenvalue  $\widehat\lambda^{(0)}_{1,n}$ from~\eqref{eq:rips_laplacian}, there is a constant $c_0>0$ such that
    \[
        \PP\left[\widehat\lambda^{(0)}_{1,n}\geq c_0\right]\to 1.
    \]
    \end{proposition}

    \begin{proof}
    Let $0=\mu_{0,n}\leq\mu_{1,n}\leq\cdots$ denote the eigenvalues of $\hat{L}_n$, repeated with multiplicity. Note that the bandwidth in~\eqref{eq:combined_bandwidth} satisfies the hypotheses of \cite[Theorem~2.4]{calder_improved_2022}. Thus, by \cite[Theorem~2.4]{calder_improved_2022}, using the modified profile from \Cref{rem:open_closed_graph_sec7}, $\mu_{1,n}$ converges in probability to $\mu_1=c\lambda_1^{(0)}$, where $c>0$ is determined by $\theta$. Because $M$ is connected, $\lambda_1^{(0)}>0$, and hence $\mu_1>0$. Therefore
    \[
        \PP\left[\mu_{1,n}\geq \frac{\mu_1}{2}\right]\to 1.
    \]

    The comparison between Laplacians in~\eqref{eq:laplacian_operator_identity} above gives the corresponding identity for the eigenvalues, and thus with probability tending to one, 
    \begin{equation}
        \widehat\lambda^{(0)}_{1,n}= \frac{2n}{(n-1)\rho_0\sigma_1^\theta}\mu_{1,n} \geq \frac{nc}{(n-1)\rho_0\sigma_1^\theta}\lambda_1^{(0)}.
    \end{equation}
    Taking any $c_0<(c\lambda_1^{(0)})/(\rho_0\sigma_1^\theta)$ proves the claim for all sufficiently large $n$.
    \end{proof}

    \begin{corollary} \label{cor:laplacian_implies_diff_convergence_sec7}
    Let $u_n\in C^0(X^{(n)}_{\epsilon_n})$ satisfy $u_n\perp \ker \hd_n$ with respect to the degree-zero inner product.  With probability tending to one,
    \[
        \|u_n\|_n^2 \ls \|\hd_n u_n\|_n^2.
    \]
    Consequently, if $\|\widehat\Delta_{0,n}u_n\|_n\xrightarrow{\PP}0$, then $\|\hd_n u_n\|_n\xrightarrow{\PP}0$.
    \end{corollary}

    \begin{proof}
    On the event in \Cref{prop:degree_zero_gap_sec7}, the restriction of $\widehat\Delta_{0,n}$ to $(\ker\hd_n)^\perp$ has spectrum contained in $[c_0,\infty)$.  Therefore, we get
    \begin{equation}
        c_0\|u_n\|_n^2
        \leq
        \langle \widehat\Delta_{0,n}u_n,u_n\rangle_n
        = \widehat Q_{0,n}(u_n)
        =
        \|\hd_n u_n\|_n^2.
    \end{equation}
    On the same event, we apply Cauchy-Schwarz to get
    \begin{equation}
        c_0\|u_n\|_n^2 \leq \langle \widehat\Delta_{0,n}u_n,u_n\rangle_n \leq \|\widehat\Delta_{0,n}u_n\|_n \|u_n\|_n \quad \text{so that} \quad \|u_n\|_n\leq c_0^{-1}\|\widehat\Delta_{0,n}u_n\|_n.
    \end{equation}
    Using the identity defining the degree-zero energy,
    \[
        \|\hd_n u_n\|_n^2
        =
        \langle \widehat\Delta_{0,n}u_n,u_n\rangle_n
        \leq
        \|\widehat\Delta_{0,n}u_n\|_n\|u_n\|_n.
    \]
    The convergence assumption on $\widehat\Delta_{0,n}u_n$ gives the desired result.
    \end{proof}

    \begin{proposition} \label{prop:harmonic_one_form_projection_sec7}
    Assume that $M$ is connected and orientable.  Let $\omega\in \Omega^1(M)$ be a smooth harmonic one-form.  For each $n$, decompose $q\omega$ by the finite-dimensional Hodge decomposition of $C^1(X^{(n)}_{\epsilon_n})$,
    \begin{equation} \label{eq:discrete_harmonic_decomposition_sec7}
        q\omega
        =
        \homega_n+
        \hd_n\widehat\beta_n,
    \end{equation}
    where $\homega_n\in \ker \widehat\Delta_{1,n}$ and $\widehat\beta_n\perp \ker\hd_n$.  Then
    \[
        \|q\omega-\homega_n\|_n\xrightarrow{\PP}0.
    \]
    \end{proposition}

    \begin{proof}
    Since $\omega$ is a harmonic 1-form, $d\omega=0$ and $\delta\omega=0$.  By the cochain property \Cref{lem:cochain_map}, we have $\hd_n(q\omega)=q(d\omega)=0$.
    Since the finite-dimensional Hodge decomposition restricts to $\ker\widehat\Delta_{1,n}\oplus\im\hd_n$ on $\ker\hd_n$, the cochain $q\omega$ has the decomposition in \eqref{eq:discrete_harmonic_decomposition_sec7} and satisfies $\hdelta_n \homega_n = 0$.
    Next, we apply the codifferential pairing convergence \Cref{thm:main_codiff} with $\alpha=\beta=\omega$.  Since $\delta\omega=0$, it gives
    \[
        \|\hdelta_n(q\omega)\|_n\xrightarrow{\PP}0,
    \]
    by the codifferential pairing convergence.  Applying $\hdelta_n$ to \eqref{eq:discrete_harmonic_decomposition_sec7} and using $\hdelta_n\homega_n=0$, we obtain $\hdelta_n\hd_n\widehat\beta_n= \hdelta_n(q\omega)$, so that $\|\widehat\Delta_{0,n}\widehat\beta_n\|_n\to 0$ in probability.  Since $\widehat\beta_n\perp\ker\hd_n$, \Cref{cor:laplacian_implies_diff_convergence_sec7} gives
    \[
        \|\hd_n\widehat\beta_n\|_n\xrightarrow{\PP}0.
    \]
    The result follows from \eqref{eq:discrete_harmonic_decomposition_sec7}.
    \end{proof}

    \subsection{Circular Coordinates}

    We conclude this section by considering an application of our results to the convergence of the circular coordinates construction of~\cite{deSilva2011}. This construction has a topological step using persistent cohomology, which selects an integral cohomology class on the Vietoris-Rips complex of the point cloud, and an analytic smoothing step, which replaces a chosen integer cocycle by its discrete harmonic representative. The result below concerns the smoothing step, assuming that the topological step has selected the integral class corresponding to the smooth harmonic form.

    In our setting, we consider a point cloud $X^{(n)}$ sampled from an embedded compact connected orientable manifold $M \subset \R^N$. Consider a harmonic 1-form $\omega \in \Omega^1(M)$ with integral periods. After choosing a basepoint $p \in M$, this defines the smooth circular coordinate
    \begin{equation} \label{eq:smooth_circular_coordinate_sec7}
        f_\omega : M \to S^1, \qquad f_\omega(y)\coloneqq \exp\left(2\pi i\int_\gamma \omega\right),
    \end{equation}
    where $\gamma$ is any path from the basepoint $p$ to $y$. The integral-period condition makes this independent of the path. On the event $\cT_n$ that $X^{(n)}_{\epsilon_n}$ is homotopy equivalent to $M$ (see~\Cref{prop:topology_quasi_iso_convergence_scales}), the class of the discretization $q\omega$ is represented by an integer-valued cocycle $\hat{\alpha}_n$. Because $q\omega - \hat{\alpha}_n$ is exact, there exists $\zeta_n \in C^0(X^{(n)}_{\epsilon_n})$ such that \label{not:topology_event}
    \begin{align} \label{eq:smooth_circular_coord_cocycle}
        q\omega = \hat{\alpha}_n + \hat{d}_n \zeta_n \quad \text{where} \quad f_\omega(y) = \exp\left(2 \pi i\left(\zeta_n(y)+\varphi_n\right)\right)
    \end{align}
    for $y \in X^{(n)}$ and some additive constant $\varphi_n\in\R$. The integrality of $\hat{\alpha}_n$ is essential here: edge-path sums of $\hat{\alpha}_n$ are integer-valued and therefore disappear after exponentiation. In practice, we do not have access to $q\omega$; instead, we use the selected cocycle $\hat{\alpha}_n$ and $\langle \cdot , \cdot \rangle_n$ to compute the Hodge decomposition and define a discrete approximation to $f_\omega$ by
    \begin{align}
        \homega_n = \hat{\alpha}_n + \hat{d}_n \hat{\zeta}_n  \andd \hat{f}_n(y) = \exp\left( 2 \pi i \hzeta_n(y)  \right),
    \end{align}
    where $\homega_n$ is the discrete harmonic representative of the real cohomology class of $\hat{\alpha}_n$, equivalently of $q\omega$. The additive phase constant in~\eqref{eq:smooth_circular_coord_cocycle} records the choice of basepoint in~\eqref{eq:smooth_circular_coordinate_sec7}.

    \begin{corollary} \label{cor:circular_coordinate_convergence_sec7}
    On the event $\cT_n$ from \Cref{prop:topology_quasi_iso_convergence_scales}, suppose that an integer-valued cocycle $\hat{\alpha}_n$ has been selected whose image in real cohomology is the class of $q\omega$.  Let
    \begin{equation} \label{eq:discrete_harmonic_circular_sec7}
        \homega_n=\hat{\alpha}_n+\hat{d}_n\hat{\zeta}_n
    \end{equation}
    be the discrete harmonic representative of this class, and set $\hat f_n(x_i)=\exp(2\pi i\hat{\zeta}_n(x_i))$.  Then the real lifts satisfy $\inf_{c\in\R}\|\zeta_n-\hat{\zeta}_n-c\|_n\xrightarrow{\PP}0$, and consequently
    \begin{equation} \label{eq:circular_coordinate_convergence_sec7}
        \inf_{\varphi\in \R}
        \left(
            \frac1n\sum_{i=1}^n
            \left|f_\omega(x_i)-e^{2\pi i\varphi}\hat f_n(x_i)\right|^2
        \right)^{1/2}
        \xrightarrow{\PP}0.
    \end{equation}
    \end{corollary}

    \begin{proof}
    We work on the event $\cT_n$ in the statement.  Since $\hat{\alpha}_n$ and $q\omega$ determine the same real cohomology class, $\homega_n$ is also the harmonic term in the Hodge decomposition of $q\omega$.  We also restrict to the high-probability analytic events underlying \Cref{prop:harmonic_one_form_projection_sec7} and \Cref{cor:laplacian_implies_diff_convergence_sec7}.  Subtracting \eqref{eq:discrete_harmonic_circular_sec7} from \eqref{eq:smooth_circular_coord_cocycle} gives
    \[
        q\omega-\homega_n
        =
        \hat{d}_n(\zeta_n-\hat{\zeta}_n).
    \]
    By \Cref{prop:harmonic_one_form_projection_sec7}, the left hand side tends to zero in the degree-one norm.  Since the topology event makes $X^{(n)}_{\epsilon_n}$ connected, $\ker\hat{d}_n$ consists of constants.  Choose $c_n\in\R$ so that $\zeta_n-\hat{\zeta}_n-c_n\perp \ker\hat{d}_n$.  Then, \Cref{cor:laplacian_implies_diff_convergence_sec7} gives
    \[
        \|\zeta_n-\hat{\zeta}_n-c_n\|_n\xrightarrow{\PP}0,
    \]
    which proves the asserted convergence of the lifts modulo constants.  The vertex norm is equivalent to the empirical $n^{-1}\sum_i$ norm up to fixed constants, and hence, for the phase $\varphi_n+c_n$,
    \[
        \left(
            \frac1n\sum_{i=1}^n
            \left|
            f_\omega(x_i)-e^{2\pi i(\varphi_n+c_n)}\hat f_n(x_i)
            \right|^2
        \right)^{1/2}
        \leq
        2\pi
        \left(
            \frac1n\sum_{i=1}^n
            \left|\zeta_n(x_i)-\hat{\zeta}_n(x_i)-c_n\right|^2
        \right)^{1/2}
        \xrightarrow{\PP}0.
    \]
    Taking the infimum over $\varphi\in\R$ proves \eqref{eq:circular_coordinate_convergence_sec7}.  The objects in the statement are defined on events whose probabilities tend to one; extending $\hat f_n$ arbitrarily to the complements does not change convergence in probability.
    \end{proof}

    \begin{remark}
    In the case where $\dim H^1(M;\R)=1$, there is only a single harmonic $1$-form $\omega$ up to normalization. Thus, on the event $\cT_n$ from \Cref{prop:topology_quasi_iso_convergence_scales}, we get $H^1(X^{(n)}_{\epsilon_n})=H^1(M)$, and we obtain the class required for~\Cref{cor:circular_coordinate_convergence_sec7} by computing the kernel of the discrete Laplacian.
    \end{remark}

    \section{Proof of Uniform Convergence of Inner Products} \label{sec:uc_proof}
    In this section, we will prove~\Cref{thm:main_uc} by considering a uniform bound over the unit ball $\cB_\R^k$ as defined in~\eqref{eq:cB_ball}. 
    We prove this in several steps. For $\alpha, \beta \in \cB_\R^k$, we decompose the error,
    \begin{align} \label{eq:uc_bias_variance_decomposition}
        \Big|\langle q\alpha, q\beta\rangle^{\kappa,w}_{n, \epsilon_n} - \sigma^\Theta_k \langle \alpha, \beta \rangle_M^{P_\Delta}\Big| \leq \Big|\langle q\alpha, q\beta\rangle^{\kappa,w}_{n, \epsilon_n} - \E\langle q\alpha, q\beta\rangle^{\kappa,w}_{n,\epsilon_n}\Big| + \Big| \E\langle q\alpha, q\beta\rangle^{\kappa,w}_{n,\epsilon_n} - \sigma^\Theta_k \langle \alpha, \beta \rangle_M^{P_\Delta}\Big|,
    \end{align}
    where the first term is a probabilistic variance term and the second term is a deterministic bias term.  
    In~\Cref{ssec:uc_bias}, we estimate the bias using a local normal-coordinate expansion of the simplex integrals and the orthogonal invariance of the admissible kernel. For the variance term, the aim is to apply a concentration bound for $U$-statistics along with a covering argument.
    We fix a smooth symmetric weight function $w: M^{k+1} \to (c_w,\infty)$ for some $c_w>0$ and an admissible VR kernel $\kappa_\epsilon$ with defining function $\Theta$, and suppress this from the inner product notation. 

    \subsection{Bias Estimate} \label{ssec:uc_bias}

    We begin with the bias term, where we prove the following estimate.
    \begin{proposition} \label{prop:uc_bias}
    For $0 < \epsilon < \epsilon_0$, we have
    \begin{align}
        \sup_{\alpha, \beta \in \cB_\R^k} \Big| \E\langle q\alpha, q\beta\rangle^{\kappa,w}_{n,\epsilon} - \sigma^\Theta_k \langle \alpha, \beta \rangle_M^{P_\Delta}\Big| \ls \epsilon^2,
    \end{align}
    where $P(\by) = w(\by)\rho(y_0) \cdots \rho(y_k) : M^{k+1} \to \R$ and $P_\Delta(x) = P(x,\ldots, x) : M \to \R$.
    \end{proposition}

    Here, we consider the expectation of the empirical inner product 
    \begin{align}
        \E_{\mu^n}[\langle q\alpha, q\beta\rangle_{n,\epsilon}] = \E_{\mu^{k+1}}[h_\epsilon(\alpha, \beta)] = \frac{(k!)^2}{\epsilon^{k(m+2)}} \int_{D^k_\epsilon(M)} q\alpha(\by) q\beta(\by) \kappa_\epsilon(\by) P(\by) \, \dvol^{k+1}(\by),
    \end{align}
    where $h_\epsilon(\alpha, \beta)$ is the $U$-kernel defined in~\eqref{eq:euclidean_u_kernel}, $D^k_\epsilon(M)$ is the fat diagonal defined in~\eqref{eq:fat_diagonal}.
    Note that we can decompose this domain by
    \begin{align} \label{eq:DM_decomposition}
        D^k_\epsilon(M) = \bigcup_{x \in M} \{(x, \by) \, : \, \by \in D^k_\epsilon(M,x)\} \hspace{6pt} \text{where} \hspace{6pt} D^k_\epsilon(M, x) \coloneqq \{ \by \in M^{k} \, : \, \diam_\R(x, \by) < \epsilon\},
    \end{align}
    so that the expectation can be expressed as
    \begin{align} \label{eq:bias_initial_expectation}
        \E[\langle q\alpha, q\beta\rangle^\kappa_{n,\epsilon}] = \int_M \left(\int_{D^k_\epsilon(M,x)} \frac{(k!)^2}{\epsilon^{k(m+2)}} q\alpha(x,\by) q\beta(x,\by) \kappa_\epsilon(x,\by) P(x,\by) \dvol^k(\by) \right) \dvol(x).
    \end{align}
    Our aim is to bound the error between the pointwise $\langle \alpha_x, \beta_x\rangle_x$ and the inner integral 
    \begin{equation} \label{eq:inner-integral}
        I(\alpha, \beta, x) \coloneq\int_{D^k_\epsilon(M,x)} \frac{(k!)^2}{\epsilon^{k(m+2)}} q\alpha(x,\by) q\beta(x,\by) \kappa_\epsilon(x,\by) P(x,\by) \dvol^k(\by).
    \end{equation}

    Since $\epsilon<\epsilon_0$, we apply~\Cref{lem:geometric_simplex_integral} to approximate the discretizations $q\alpha(x,\by)$ with evaluations $\alpha_x(\bv)$ where $\bv = \log_x^k(\by) \in T^k_x M$.
    Under this change of coordinates, the domain becomes 
    \begin{align} \label{eq:L^k}
        L^k_\epsilon(M,x) = \log^k_x(D^k_\epsilon(M,x)) = \{\bv = (v_1, \ldots, v_k) \in T^k_x M \, : \, \diam_\R(x, \exp_x(\bv)) < \epsilon\}.
    \end{align}
    We perform a second change of variables by rescaling $\bu = \epsilon^{-1} \bv$, where the rescaled domain is
    \begin{equation} \label{eq:tL_domain}
        \tL^k_\epsilon(M,x) = L^k_\epsilon(M,x)/\epsilon = \{\bu = (u_1, \ldots, u_k) \in T^k_x M \, : \diam_\R(x, \exp_x(\epsilon\bu)) < \epsilon\}. 
    \end{equation}
    We estimate $I(\alpha, \beta, x)$ by reducing it to an invariant integral on $T^k_x M$ in three steps.
    \begin{enumerate}
        \item We use orthogonal invariance to identify this integral with a multiple of $\langle \alpha_x, \beta_x \rangle_x$.
        \item We replace the domain $\tL^k_\epsilon(M,x)$ and ambient kernel with tangent-space counterparts.
        \item We expand the joint weight-density factor $P$ about the diagonal.
    \end{enumerate}

    \begin{lemma} \label{lem:bilinear_form_trace}
    Let $\ell,r \in \N$ with $r \leq m$. Let \(V\) be an \(m\)-dimensional Hilbert space,
    \(H\subset V^\ell\) be a linear subspace invariant under the diagonal
    \(O(V)\)-action, and $d\bv$ be the induced Euclidean measure on $H$. Let \(f:H\to\mathbb R\) be invariant under this action, and let $\Phi:H\to \Lambda^r V$
    be an \(O(V)\)-equivariant map such that $|\Phi(\bv)|^2 |f(\bv)|$ is integrable over $H$. Then there exists a constant
    \(\sigma_{\Phi,f}\), depending only on \(\Phi\) and \(f\), such that for all
    \(\alpha,\beta\in \Lambda^r V^*\),
    \[
    \int_H \alpha(\Phi(\bv))\beta(\Phi(\bv))f(\bv)\,d\bv
    =
    \sigma_{\Phi,f}\langle \alpha,\beta\rangle.
    \]
    More explicitly, if \(e^1,\ldots,e^m\) is an orthonormal basis of $V^*$, then
    \[
    \sigma_{\Phi,f}
    =
    \int_H
    \left[
    (e^1\wedge\cdots\wedge e^r)(\Phi(\bv))
    \right]^2
    f(\bv)\,d\bv .
    \]
    \end{lemma}
    \begin{proof}
        Let $e^I$ be a basis for $\Lambda^r V^*$, so that $\alpha = \sum_{I} \alpha_I e^I$ and $\beta = \sum_J \beta_J e^J$.
        Let
        \begin{align}
            A_{I,J} = \int_{H} e^I(\Phi(\bv)) e^J(\Phi(\bv)) f(\bv) d\bv.
        \end{align}
        Our aim is to show that $A_{I,J} = 0$ whenever $I \neq J$ and $A_{I,I} = A_{J,J}$ for all $I, J$. First, suppose $I \neq J$ so there exists some index $a \in \{1,\ldots,m\}$ such that $a \in I$ but $a \notin J$. Define the reflection $R_a \in O(V)$ such that $R_a(e_a) = -e_a$ and $R_a(e_s) = e_s$ for all $s \neq a$. By changing variables with $R_a$, and using the invariance of $H$, the invariance of $f$, and the equivariance of $\Phi$, we get
        \begin{align}
            A_{I,J} = \int_{H} e^I(R_a \Phi(\bv)) e^J(R_a \Phi(\bv)) f(\bv) d\bv = - \int_{H} e^I(\Phi(\bv)) e^J(\Phi(\bv)) f(\bv) d \bv = -A_{I,J},
        \end{align}
        so that $A_{I,J} = 0$. Next, let $I = (i_1 < \ldots < i_r)$ and $J = (j_1 < \ldots < j_r)$. Choose $U \in O(V)$ such that $U(e_{j_q}) = e_{i_q}$ for all $q \in \{1,\ldots,r\}$, so that $e^I(U\Phi(\bv)) = e^J(\Phi(\bv))$. Then, by the same invariance and equivariance argument, we get
        \begin{align}
            A_{I,I} = \int_{H} e^I(U\Phi(\bv))^2 f(\bv) d\bv = \int_{H} e^J(\Phi(\bv))^2 f(\bv) d\bv = A_{J,J},
        \end{align}
        and we let $\sigma_{\Phi,f}$ denote this value. Then, we obtain 
        \begin{align}
            \int_{H} \alpha(\Phi(\bv))\beta(\Phi(\bv)) f(\bv) d\bv = \sigma_{\Phi,f} \sum_{I} \alpha_I \beta_I = \sigma_{\Phi,f} \langle \alpha, \beta \rangle.
        \end{align}
    \end{proof}

    To apply this lemma to the current setting, we wish to take $V = T_x M$ and $H = T_x^k M$. However, our domain $\tL^k_\epsilon(M,x)$ under consideration from \eqref{eq:tL_domain} is not $O(T_x M)$-invariant. Thus, we must replace $\tilde{L}_\epsilon^k(M)$ with one that exhibits $O(T_x M)$-invariance, see \Cref{fig:domain-transforms}. 

\begin{figure}
    \centering
    \includegraphics[width=0.8\linewidth]{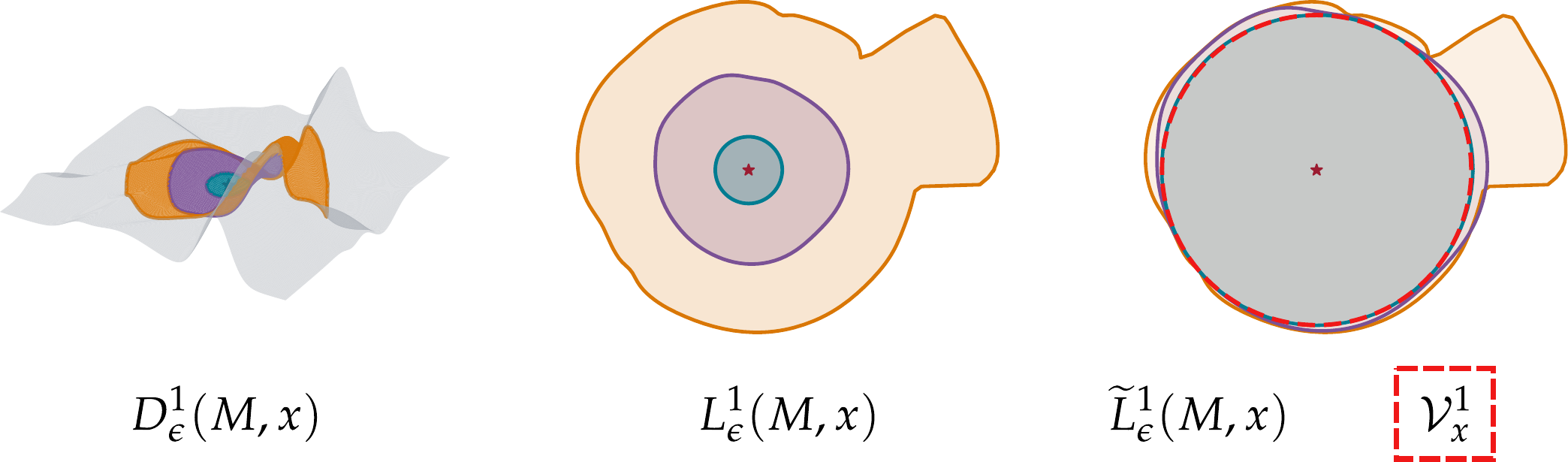}
    \captionsetup{labelfont={bf,up}}
    \caption{Domains for $m=2$ and $k=1$: $D^1_\epsilon(M,x)\subset M$ (left), its normal-coordinate image $L^1_\epsilon(M,x)$ (middle), and $\tL^1_\epsilon(M,x)=\epsilon^{-1}L^1_\epsilon(M,x)$ superimposed on $\cV_x^1$ in dashed red (right). The chosen $\epsilon$ values emphasize finite-scale asymmetry and the regularizing role of the normalization: $\tL^1_\epsilon(M,x)$ becomes closer to $\cV_x^1$ as $\epsilon$ decreases. The figure is schematic, and the displayed values of $\epsilon$ are not assumed to satisfy $\epsilon<\epsilon_0$.}
    \label{fig:domain-transforms}
\end{figure}

    \begin{lemma} \label{lem:kappa_domain_replacement}
        Let $\kappa_\epsilon : (\R^N)^{k+1} \to \R$ be an admissible VR kernel of degree $k$. For $\bu \in T_x^k M$, set $\tp(\bu) = p(0, \bu)$ and
        \begin{align}
            \cV_x^k = \{ \bu \in T^k_x M \, : \, \diam_{T_x M} (0, \bu) < 1\}.
        \end{align}
        If $f:T_x^kM\to\R$ is bounded on $\tL^k_\epsilon(M,x)\cup \cV_x^k$, then, for every $0<\epsilon<\epsilon_0$,
        \begin{align} \label{eq:kappa_domain_replacement}
            \left|\int_{T_x^kM} f(\bu)\left(\kappa_\epsilon(x,\exp_x(\epsilon\bu))\one_{\tL^k_\epsilon(M,x)}(\bu)-\Theta(\tp(\bu))\one_{\cV_x^k}(\bu)\right)d\bu\right|
            \ls \|f\|_{L^\infty(\tL^k_\epsilon(M,x)\cup \cV_x^k)}\epsilon^2,
        \end{align}
        uniformly in $x\in M$.
    \end{lemma}
    \begin{proof}
        Set $u_0=0$. Both $\tL^k_\epsilon(M,x)$ and $\cV_x^k$ are contained in a uniformly bounded subset $B\subset T_x^kM$. On this bounded set, the Taylor expansion~\eqref{eq:exp_taylor} gives, uniformly in $x$, $\epsilon$, and $\bu$,
        \begin{align} 
            \exp_x(\epsilon u_i)-\exp_x(\epsilon u_j)=\epsilon(u_i-u_j)+\epsilon^2(Q_x(u_i,u_i)-Q_x(u_j,u_j))+O(\epsilon^3).
        \end{align}
        Since $u_i-u_j\in T_xM$ and $Q_x(u_i,u_i)-Q_x(u_j,u_j)\in N_xM$, the quadratic cross term vanishes. Hence
        \begin{align} \label{eq:y_u_comparison}
            \epsilon^{-2}\|\exp_x(\epsilon u_i)-\exp_x(\epsilon u_j)\|^2=\|u_i-u_j\|^2+O(\epsilon^2), \qquad 0\leq i<j\leq k.
        \end{align}
        Consider the set of simplices in the tangent space possessing an edge of squared distance that is order $\epsilon^2$ close to one:
        \begin{align}
            S_\epsilon=\bigcup_{i<j}\left\{\bu\in B:\left|\|u_i-u_j\|^2-1\right|\ls \epsilon^2\right\}.
        \end{align}
        This is a finite union of bounded domains of thickness $O(\epsilon^2)$, so $\vol(S_\epsilon)\ls \epsilon^2$.
        On $B\setminus S_\epsilon$, the comparison~\eqref{eq:y_u_comparison} implies that membership in $\tL^k_\epsilon(M,x)$ and membership in $\cV_x^k$ agree. Moreover, either both arguments of $\Theta$ lie in the interior of the cube $[0,1)^{\binom{k+1}{2}}$, or both terms in~\eqref{eq:kappa_domain_replacement} vanish. Therefore, using the Lipschitz bound for $\Theta$ on the interior of the cube, we get
        \begin{align}
            \left|\kappa_\epsilon(x,\exp_x(\epsilon\bu))\one_{\tL^k_\epsilon(M,x)}(\bu)-\Theta(\tp(\bu))\one_{\cV_x^k}(\bu)\right|\ls\epsilon^2
        \end{align}
        on $B\setminus S_\epsilon$. On $S_\epsilon$, the integrand can be nonzero only on $S_\epsilon\cap(\tL^k_\epsilon(M,x)\cup\cV_x^k)$; there both terms are uniformly bounded, and the shell has volume $O(\epsilon^2)$. Integrating these two estimates proves the result.
    \end{proof}

    \begin{lemma} \label{lem:smooth_localization}
        Let $w: M^{k+1} \to \R$ be a smooth function. Then for any $x \in M$ and $\bv = (v_1, \ldots, v_k) \in T^k_x M$ such that $\|\bv\| < \inj_\R(M)$, 
        \begin{align}
            w(x, \exp_x(\bv)) - w_\Delta(x) = L(x,\bv) + R(x,\bv)
        \end{align}
        where $w_\Delta(x) = w(x, \ldots, x)$. Here, $L(x,\bv)$ is linear in $\bv$ and satisfies $|L(x,\bv)| \leq C_1\|\bv\|$, while $|R(x,\bv)| \leq C_2 \|\bv\|^2$, where $C_1, C_2> 0$ depend on $w$. 
    \end{lemma}
    \begin{proof}
        Let $\tw_x: T^k_x M \to \R$ be defined by $\tw_x(\bv) = w(x, \exp_x(\bv))$. Note that this is a smooth function on a vector space, so by a Taylor expansion about $\bv = 0$, we obtain
        \begin{align}
            \tw_x(\bv) = \tw_x(0) + d\tw_x(0)[\bv] + R(x,\bv),
        \end{align}
        where $|R(x,\bv)| \leq C_2 \|\bv\|^2$ is a uniform bound over $M$ since $M$ is compact. Furthermore, we let
        \begin{align}
            L(x,\bv) = d\tw_x(0)[\bv],
        \end{align}
        which is a linear function of $\bv$, and thus satisfies $|L(x,\bv)| \leq C_1 \|\bv\|$. 
    \end{proof}

    \begin{proof}[Proof of~\Cref{prop:uc_bias}]
        We begin with a change of coordinates by $\bv = \log_x(\by)$ with the Jacobian $\overline{J}(x, \bv) = \prod_{i=1}^k J(x,v_i)$ and the domain $L^k_\epsilon(M,x)$ from \eqref{eq:L^k} to get
        \begin{align}
            I(\alpha,\beta, x) = \int_{L^k_\epsilon(M,x)} \frac{(k!)^2}{\epsilon^{k(m+2)}} q\alpha(x,\by) q\beta(x,\by) \kappa_\epsilon(x,\by) P(x,\by) \overline{J}(x,\bv) d\bv.
        \end{align}
        We further rescale by setting $\bv = \epsilon \bu$. We apply~\Cref{lem:geometric_simplex_integral},~\Cref{lem:smooth_localization}, and the Jacobian remainder estimate~\eqref{eq:jacobian_remainder} to obtain the following uniform bounds (over $\cB_\R^k$ and $M$) on the bounded rescaled domain $\tL^k_\epsilon(M,x) = L^k_\epsilon(M,x)/\epsilon$,
        \begin{align}
            k! q\alpha(x,\by) &= \epsilon^k\alpha_x(\bu) + \epsilon^{k+1}L_\alpha(\bu) + O(\epsilon^{k+2}) \\
            k! q\beta(x,\by) & = \epsilon^k\beta_x(\bu) + \epsilon^{k+1} L_\beta(\bu) + O(\epsilon^{k+2}) \\
            P(x,\by) & = P_\Delta(x) + \epsilon L_P(\bu) + O(\epsilon^2) \\
            \overline{J}(x,\epsilon\bu) &= 1 + O(\epsilon^2).
        \end{align}
        We collect the terms of order $1$ and $\epsilon$ and integrate the rest over $\tL^k_\epsilon(M,x)$ to get
        \begin{align}
            I(\alpha,\beta,x) =  \int_{\tL^k_\epsilon(M,x)} \Big(\alpha_x(\bu) \beta_x(\bu) P_\Delta(x) + \epsilon \cL(\bu) \Big)\kappa_\epsilon(x,\exp_x(\epsilon\bu))d\bu + O(\epsilon^{2}).
        \end{align}
        Here, $\cL(\bu)$ is given by
        \begin{align}
            \cL(\bu) = L_\alpha(\bu) \beta_x(\bu) P_\Delta(x) + \alpha_x(\bu) L_\beta(\bu) P_\Delta(x) + \alpha_x(\bu) \beta_x(\bu) L_P(\bu),
        \end{align}
        which is homogeneous of order $2k+1$, so that $\cL(-\bu) = -\cL(\bu)$.
        Applying~\Cref{lem:kappa_domain_replacement} with $f(\bu)=\alpha_x(\bu)\beta_x(\bu)P_\Delta(x)+\epsilon\cL(\bu)$ gives
        \begin{align} \label{eq:uc_I_expanded_integral}
            I(\alpha,\beta,x) =  \int_{\cV_x^k} \Big(\alpha_x(\bu) \beta_x(\bu) P_\Delta(x) \Theta(\tp(\bu)) + \epsilon \cL(\bu) \Theta(\tp(\bu)) \Big)d\bu + O(\epsilon^{2}).
        \end{align}
        Since $\cL(\bu)$ is odd and $\Theta$ is even, we get
        \begin{align}
            \int_{\cV_x^k} \cL(\bu) \Theta(\tp(\bu)) d\bu  = 0.
        \end{align}
        Then, by applying~\Cref{lem:bilinear_form_trace} (in this case $H = V^k$ and $\Phi(\bv) = v_1 \wedge \ldots \wedge v_k$), using the $O(T_x M)$-invariance of $\Theta(\tp(\bu)) \one_{\cV_x^k}(\bu)$, we get
        \begin{align}
            \int_{\cV_x^k} \alpha_x(\bu) \beta_x(\bu) \Theta(\tp(\bu))  d\bu = \sigma^\Theta_k \langle \alpha_x, \beta_x \rangle_x \hspace{6pt} \text{where} \hspace{6pt} \sigma^\Theta_k = \int_{\cV_x^k} (e^1 \wedge \ldots \wedge  e^k)(\bu)^2 \, \Theta(\tp(\bu)) d\bu.
        \end{align}
        Therefore,
        \begin{align}
            I(\alpha, \beta, x) =  \sigma^\Theta_k \langle \alpha_x, \beta_x \rangle_x  P_\Delta(x) + O(\epsilon^{2}).
        \end{align}
        Integrating this estimate over $M$ and using~\eqref{eq:bias_initial_expectation}, we obtain
        \begin{align}
            \left|\E\langle q\alpha,q\beta\rangle_{n,\epsilon}^{\kappa,w}-\sigma_k^\Theta\langle\alpha,\beta\rangle_M^{P_\Delta}\right|
            \leq \int_M\left|I(\alpha,\beta,x)-\sigma_k^\Theta\langle\alpha_x,\beta_x\rangle_xP_\Delta(x)\right|\dvol(x)\ls\epsilon^2.
        \end{align}
        Taking the supremum over $\alpha,\beta\in\cB_\R^k$ proves the result.
    \end{proof}

    \subsection{A Geometric U-Kernel and Covering Bound} \label{ssec:uc_geometric_kernel}

    Next, we consider the variance term in~\eqref{eq:uc_bias_variance_decomposition}. For fixed $\alpha,\beta\in\Omega^k_c(\R^N)$, the empirical inner product can be written as the $U$-statistic $\langle q\alpha, q\beta \rangle_{n, \epsilon} = \U_n[h_\epsilon(\alpha, \beta)]$ of a $U$-kernel $h_\epsilon$. To prepare the later uniform concentration argument, we replace the Euclidean $U$-kernel $h^e_\epsilon=h_\epsilon$ by a geometric $U$-kernel $h^g_\epsilon$ coming from the normal-coordinate expansion of the simplex integrals. 
    The kernels differ only by a controlled higher-order error, while $h^g_\epsilon$ depends on the forms only through their restricted ambient $1$-jets along $M$. Thus the covering argument is intrinsic to $M$ rather than $\R^N$.\medskip

    Let $\by = (y_0, \ldots, y_k) \in D^k_\epsilon(M)$. For each $\ell \in \{0,\ldots,k\}$, we define
    \begin{align}
        \bv^\ell = \Big( \log_{y_\ell}(y_0), \ldots, \log_{y_{\ell}}(y_{\ell-1}), \log_{y_{\ell}}(y_{\ell+1}), \ldots, \log_{y_\ell}(y_k) \Big) \in T^k_{y_\ell} M.
    \end{align}
    For the Euclidean $U$-kernel $h_\epsilon^e$, we use~\Cref{lem:geometric_simplex_integral}, where the orientation signs cancel, to obtain for any $\ell \in \{0,\ldots,k\}$,
    \begin{align} \label{eq:unif_conv_q_approx}
        (k!)^2 q\alpha(\by)q\beta(\by) = \left(\alpha_{y_\ell}(\bv^\ell)\beta_{y_\ell}(\bv^\ell) + \alpha_{y_\ell}(\bv^\ell) L_{\beta,\ell}(\bv^\ell) + \beta_{y_\ell}(\bv^\ell) L_{\alpha,\ell}(\bv^\ell) \right) + R(\by),
    \end{align}
    where $R(\by) \ls \epsilon^{2k+2}$ uniformly over the unit ball of ambient forms $\cB_\R^k$ from \eqref{eq:cB_ball}. Now, for each $\alpha, \beta \in \Omega^k_c(\R^N)$, we define a \emph{geometric} $U$-kernel $h^g_\epsilon(\alpha,\beta) : M^{k+1} \to \R$ by symmetrizing the first order term in~\eqref{eq:unif_conv_q_approx}
    \begin{equation} \label{eq:geom-U-kernel}
        h^g_\epsilon(\alpha, \beta)(\by) \coloneqq \frac{\epsilon^{-k(m+2)}}{k+1}\sum_{\ell=0}^k \left(\alpha_{y_\ell}(\bv^\ell)\beta_{y_\ell}(\bv^\ell) + \alpha_{y_\ell}(\bv^\ell) L_{\beta,\ell}(\bv^\ell) + \beta_{y_\ell}(\bv^\ell) L_{\alpha,\ell}(\bv^\ell) \right) \kappa_\epsilon(\by) w(\by),
    \end{equation} over the choice of basepoint $y_\ell$. This has the property that
    \begin{align} \label{eq:euclidean_geometric_u_kernel_error}
        \sup_{\alpha, \beta \in\cB_\R^k} |h^e_\epsilon(\alpha, \beta)(\by) - h^g_\epsilon(\alpha, \beta)(\by)| \ls  \epsilon^{-k(m+2)} R(\by) \one_{D^k_\epsilon(M)}(\by) \ls \epsilon^{2-km} \one_{D^k_\epsilon(M)}(\by),
    \end{align}
    where $\one_{D^k_\epsilon(M)}$ denotes the indicator function.

    \subsection{Pointwise Variance Estimate} \label{ssec:uc_pointwise_estimate}

    In order to perform the pointwise concentration, our main tool will be Bernstein inequalities for $U$-statistics~\cite{hoeffding_probability_1963,arcones_bernstein-type_1995}. We cite the statement from~\cite{peel_empirical_2010}.

    \begin{theorem}{\cite{hoeffding_probability_1963,arcones_bernstein-type_1995}} \label{thm:bernstein}
        Given a $U$-kernel $h$, there exists a constant $c> 0$ depending only on $k$ such that 
        \begin{align}
            \PP\big[ |\U_n[h] - \E[h]| > t\big] \leq 4 \exp \left(\frac{-cn t^2}{\Xi + \|h\|_\infty t}\right),
        \end{align}
        where $\Xi = \Var_{x \sim \mu}\Big[\E_{\by \sim \mu^k}[h(x,\by)]\Big]$.
    \end{theorem}

    To apply this inequality to $h^g_\epsilon(\alpha, \beta)$, we must control its supremum norm and the variance of its first projection. Both estimates depend on the volume of $D^k_\epsilon(M)$. We begin with a high-probability bound for the empirical mass of this support.

    \begin{lemma} \label{lem:concentration_D_k_support}
        Let $\one_{D^k_\epsilon(M)} : M^{k+1} \to \R$ be the indicator function on $D^k_\epsilon(M)$, which is symmetric. For a sufficiently large constant $C$, define
        \begin{align}
            \cE_\epsilon \coloneqq \left\{ \U_n[\one_{D^k_\epsilon(M)}] \leq C\epsilon^{km}\right\},
        \end{align}
        Then,
        \begin{align}
            \PP[\cE_\epsilon^c] \leq 4 \exp \left(-cn\epsilon^{km}\right).
        \end{align}
    \end{lemma}
    \begin{proof}
        As $\one_{D^k_\epsilon(M)}$ is an indicator function, we have $\|\one_{D^k_\epsilon(M)}\|_\infty \leq 1$. Furthermore, by~\eqref{eq:geodesic_euclidean_balls},
        \begin{align}
            D_\epsilon^k(M,x)
            \subset \big(\ball{\R^N}{x}{\epsilon}\cap M\big)^k
            \subset \big(\ball{M}{x}{S(\epsilon)}\big)^k.
        \end{align}
        Since $S(\epsilon)\ls\epsilon$, uniform volume growth of geodesic balls on the compact manifold $M$ gives, uniformly for $x\in M$,
        \begin{align} \label{eq:fat_diagonal_fiber_volume}
            \vol^k(D_\epsilon^k(M,x))
            &\leq \vol_M\big(\ball{M}{x}{S(\epsilon)}\big)^k
            \ls \epsilon^{km}, \\
            \mu^k(D_\epsilon^k(M,x))
            &\leq \rho_+^k\vol^k(D_\epsilon^k(M,x))
            \ls \epsilon^{km}.
        \end{align}
        Consequently, by Fubini,
        \begin{align}
            \E[\one_{D^k_\epsilon(M)}]
            &=\mu^{k+1}(D^k_\epsilon(M))
            =\int_M\mu^k(D^k_\epsilon(M,x))\,d\mu(x)
            \ls\epsilon^{km}, \\
            g(x)&=\E_{\by}[\one_{D^k_\epsilon(M)}(x,\by)]
            =\mu^k(D^k_\epsilon(M,x))
            \ls\epsilon^{km}.
        \end{align}
        Therefore, we have $\Xi = \Var(g(x)) \ls \epsilon^{2km}$. Choosing $C$ large enough so that $\E[\one_{D^k_\epsilon(M)}] + \epsilon^{km} \leq C\epsilon^{km}$, applying~\Cref{thm:bernstein}, and setting $t = \epsilon^{km}$, we get (when $\epsilon < 1$) that
        \begin{align}
            \PP[\cE_\epsilon^c]  \leq 4 \exp \left( \frac{-cn\epsilon^{km}}{1 + \epsilon^{km}}\right) \leq 4 \exp \left(-cn\epsilon^{km}\right).
        \end{align}
    \end{proof}
    \begin{remark}
        The $U$-statistic $\U_n[\one_{D_\epsilon^k(M)}]$ corresponds to the proportion of possible $k$-simplices that occur in $X_\epsilon$. The preceding lemma gives an upper-tail bound for this simplex density.
    \end{remark}

    \begin{lemma} \label{lem:geokernel-variance-tail}
        For $\alpha, \beta \in \cB_\R^k$ we have 
        \begin{align}
            \PP\left[\left| \U_n[h^g_\epsilon(\alpha, \beta)] - \E[h^g_\epsilon(\alpha, \beta)]\right|  > t\right] \leq 4\exp\left(\frac{-cn t^2}{1 + \epsilon^{-km}t}\right)
        \end{align}
    \end{lemma}
    \begin{proof}
    We begin by bounding $h^g_\epsilon(\alpha, \beta)$ in~\eqref{eq:geom-U-kernel} itself. The support of $h^g_\epsilon(\alpha, \beta)(\by)$ is $D^k_\epsilon(M)$, where we have the bounds $|\alpha_x(\bv)| \ls \epsilon^k$ and $|L_\alpha(\bv)| \ls \epsilon^{k+1}$, so that
    \begin{align}
        |h^g_\epsilon(\alpha, \beta)(\by)| \ls \epsilon^{-km} \one_{D^k_\epsilon(M)} \andd \|h^g_\epsilon(\alpha, \beta)\|_\infty \ls \epsilon^{-km}.
    \end{align}
    For the first projection $g(x)=\E_{\by\sim\mu^k}[h^g_\epsilon(\alpha,\beta)(x,\by)]$, the same pointwise bound gives
    \begin{align}
        |g(x)| \ls \epsilon^{-km}\mu^k(D^k_\epsilon(M,x)) \ls 1,
    \end{align}
    where the last inequality follows from~\eqref{eq:fat_diagonal_fiber_volume}. Thus, the Bernstein projection variance satisfies $\Xi=\Var(g(x))\ls 1$. Then, we obtain our result by applying~\Cref{thm:bernstein}.

    \end{proof}

    \subsection{Uniform Variance Estimate} \label{ssec:uc_uniform_variance}

    We upgrade the fixed-form concentration estimate to a uniform estimate by covering the restricted ambient $1$-jets of forms in $\cB^k_\R$. We first bound the covering number of this jet class; then we use a union bound over the resulting finite cover to obtain the uniform estimate. For a Hilbert space $E$, let $\Lip(M,E)$ denote the space of Lipschitz maps.

    \begin{definition} \label{def:ambient_one_jet}
        For $\alpha = \sum_I \alpha_I dz^I \in \Omega^k_c(\R^N)$, we define the \emph{ambient 1-jet of $\alpha$ on $M$} to be
        \begin{align}
            j(\alpha) = (\alpha_I|_M, \partial_i \alpha_I|_M)_{i,I} \in \Lip(M, E),
        \end{align}
        where $E = \Lambda^k \R^N \oplus (\Lambda^k \R^N)^N$ contains information about the ambient form along with all ambient derivatives, restricted to $M$. Thus $j: \Omega^k_c(\R^N) \to \Lip(M, E)$. 
    \end{definition}

    The following covering bound can be obtained via a standard sup-norm entropy bound for bounded Lipschitz functions on compact $m$-dimensional domains.

    \begin{lemma} \label{lem:covering_bound}
        Let $0\leq k\leq m$, and let $\cN_\gamma$ be the smallest number of $\cC^0(M, E)$ balls of radius $\gamma > 0$ with centers lying in $j(\cB_\R^k)$ required to cover $j(\cB_\R^k) \subset \Lip(M,E)$. Then,
        \begin{align}
            \log \cN_\gamma \ls \gamma^{-m}.
        \end{align}
    \end{lemma}
    \begin{proof} For $L>0$, let $\Lip_L(M,E)$ denote the set of maps $F:M\to E$ satisfying 
        \begin{align} 
            \|F\|_{C^0(M,E)} \leq L \andd \Lip(F) \leq L. 
        \end{align} 
        Since $\alpha \in \cB_\R^k$ has uniformly bounded coefficients, first derivatives, and second derivatives, there exists $L>0$ such that $j(\cB_\R^k) \subset \Lip_L(M,E)$.
        By the standard sup-norm entropy bound for bounded Lipschitz functions on compact $m$-dimensional domains, see~\cite[Example 5.10]{wainwright_high-dimensional_2019}, there exists a $C^0(M,E)$ cover of $\Lip_L(M,E)$ by $\cN^{\Lip}_\gamma$ balls of radius $\gamma/2$, such that $\log(\cN^{\Lip}_\gamma) \ls \gamma^{-m}$. We retain only those balls which intersect $j(\cB_\R^k)$ and denote this subcollection by $\{B'_r\}_{r=1}^{\cN'_\gamma}$. 
        For each $B'_r$, choose $\eta^r \in \cB_\R^k$ such that $j(\eta^r) \in B'_r$. Because each $B'_r$ has radius $\gamma/2$, the balls $\{\ball{\cC^0(M,E)}{j(\eta^r)}{\gamma}\}_{r=1}^{\cN'_\gamma}$ cover $j(\cB_\R^k)$ and have centers in $j(\cB_\R^k)$. Therefore, 
        \begin{align} 
            \log \cN_\gamma \leq \log \cN'_\gamma  \leq \log \cN^{\Lip}_\gamma \ls \gamma^{-m}. 
        \end{align} 
    \end{proof}

    \begin{lemma} \label{lem:uc_geometric_variance}
        We have 
        \begin{align}
            \PP \left[ \sup_{\alpha, \beta \in \cB_\R^k} \left| \U_n[h^g_\epsilon(\alpha, \beta)] - \E[h^g_\epsilon(\alpha, \beta)]\right| \gtrsim t + \gamma  \right] \leq 4\cN^2 \exp\left(\frac{-c_1 n t^2}{1 +  \epsilon^{-km}t}\right) + 4\exp( - c_2n \epsilon^{km}).
        \end{align}
    \end{lemma}
    \begin{proof}
        Let $\{\ball{\cC^0(M,E)}{j(\eta^r)}{\gamma}\}_{r=1}^{\cN}$, with $\eta^r\in\cB_\R^k$, be the collection of $\cC^0(M,E)$-balls which cover $j(\cB_\R^k)$ from~\Cref{lem:covering_bound}. By applying a union bound to \Cref{lem:geokernel-variance-tail}, we get
        \begin{align} \label{eq:uc_union_variance}
            \PP\left[ \sup_{r,s \in [\cN]} \left| \U_n[h^g_\epsilon(\eta^r, \eta^s)] - \E[h^g_\epsilon(\eta^r, \eta^s)]\right| > t\right] \leq 4\cN^2 \exp\left(\frac{-c_1n t^2}{ 1+  \epsilon^{-km}t}\right)
        \end{align}
        Let $\alpha, \beta \in \cB_\R^k$. Because these balls cover $j(\cB_\R^k)$, there exist $\eta^r, \eta^s$ such that
        \begin{align}
            \|j(\alpha) - j(\eta^r)\|_{\cC^0(M,E)}< \gamma \andd \|j(\beta) - j(\eta^s)\|_{\cC^0(M,E)} < \gamma. 
        \end{align}
        Let $\by = (y_0, \ldots, y_k) \in M^{k+1}$. We note that for any $\ell \in \{0,\ldots,k\}$, we have
        \begin{align}
            |\alpha_{y_\ell}(\bv^\ell) - \eta^r_{y_\ell}(\bv^\ell)| &\leq \|j(\alpha) - j(\eta^r)\|_{\cC^0(M,E)} \|\bv^\ell\|^k\\
            |L_\alpha(\bv^\ell) - L_{\eta^r}(\bv^\ell)| &\leq\|j(\alpha) - j(\eta^r)\|_{\cC^0(M,E)}\|\bv^\ell\|^{k+1},
        \end{align}
        and similarly for $\beta$ and $\eta^s$. 
        Thus, we can bound the difference of the geometric $U$-kernels by
        \begin{align} \label{eq:geometric_kernel_error}
            |h^g_{\epsilon}(\alpha, \beta)(\by) - h^g_{\epsilon}(\eta^r, \eta^s)(\by)|  \ls \frac{\gamma}{\epsilon^{km}} \one_{D^k_\epsilon(M)},
        \end{align}
        and by Fubini and~\eqref{eq:fat_diagonal_fiber_volume}, we get
        \begin{align}
            \sup_{\alpha, \beta \in \cB_\R^k} \left| \E[h^g_\epsilon(\alpha, \beta)-h^g_\epsilon(\eta^r, \eta^s)] \right| \ls \gamma
        \end{align}
        where the choices of $\eta^r$ and $\eta^s$ depend on $\alpha, \beta$. Applying~\Cref{lem:concentration_D_k_support} to~\eqref{eq:geometric_kernel_error} gives
        \begin{align}
            \sup_{\alpha, \beta \in \cB_\R^k} \left| \U_n[h^g_\epsilon(\alpha, \beta)-h^g_\epsilon(\eta^r, \eta^s)] \right| \ls \dfrac{\gamma}{\epsilon^{km}} \U_n[\one_{D_\epsilon^k(M)}] \ls \gamma.
        \end{align}
        on the event $\cE_\epsilon$. Then by the triangle inequality, on the event $\cE_\epsilon$,
        \begin{align}
            \sup_{\alpha, \beta \in \cB_\R^k} \left| \U_n[h^g_\epsilon(\alpha, \beta)] - \E[h^g_\epsilon(\alpha, \beta)]\right| \leq  \sup_{r,s \in [\cN] }\left| \U_n[h^g_\epsilon(\eta^r, \eta^s)] - \E[h^g_\epsilon(\eta^r, \eta^s)]\right| + C_1 \gamma
        \end{align}
        Finally, by applying~\eqref{eq:uc_union_variance} and the probability of $\PP[\cE_\epsilon^c]$ from~\Cref{lem:concentration_D_k_support}, we get
        \begin{align}
            \PP \left[ \sup_{\alpha, \beta \in \cB_\R^k} \left| \U_n[h^g_\epsilon(\alpha, \beta)] - \E[h^g_\epsilon(\alpha, \beta)]\right| > t + C_1\gamma  \right] \leq 4\cN^2 \exp\left(\frac{-c_1 n t^2}{1 +  \epsilon^{-km}t}\right) + 4\exp( - c_2n \epsilon^{km}).
        \end{align}
    \end{proof}

    Our variance results thus far have been in terms of geometric kernels, which contain intrinsic information unavailable to the data science practitioner. We apply~\Cref{lem:covering_bound} to obtain the main variance bound for the Euclidean $U$-kernels as follows. 

    \begin{proposition} \label{prop:uniform_variance_euclidean}
        There is a constant $C_0>0$, depending only on the fixed data, such that the following holds. Let $p \in (0,1)$ and suppose $C_0\log(8/p) \leq n\epsilon^{km}$. With probability at least $1-p$, we have 
        \begin{align}
            \sup_{\alpha, \beta \in \cB_\R^k} \left| \U_n[h_\epsilon(\alpha, \beta)] - \E[h_\epsilon(\alpha, \beta)]\right| \ls  \sqrt{ \frac{\gamma^{-m} + \log(\frac{1}{p})}{n}} + \frac{\gamma^{-m} + \log(\frac{1}{p})}{n\epsilon^{km}} + \gamma + \epsilon^2
        \end{align}
    \end{proposition}
    \begin{proof}
        We begin by bounding the error between the Euclidean and geometric $U$-kernels. By~\eqref{eq:euclidean_geometric_u_kernel_error}, we note that
        \begin{align}
            \sup_{\alpha, \beta \in \cB_\R^k} |\E[h^e_\epsilon(\alpha, \beta)-h^g_\epsilon(\alpha, \beta)]| \ls \epsilon^2 \andd \sup_{\alpha, \beta \in \cB_\R^k} |\U_n[h^e_\epsilon(\alpha, \beta)-h^g_\epsilon(\alpha, \beta)]| \ls \epsilon^2,
        \end{align}
        where the latter holds on the event $\cE_\epsilon$ defined in~\Cref{lem:concentration_D_k_support}. We note that this is exactly the support event used in~\Cref{lem:uc_geometric_variance}, so its failure probability is charged only once. Using that lemma, along with the triangle inequality, we obtain
        \begin{align}
            \PP \left[ \sup_{\alpha, \beta \in \cB_\R^k} \left| \U_n[h_\epsilon(\alpha, \beta)] - \E[h_\epsilon(\alpha, \beta)]\right| \gtrsim t + \gamma + \epsilon^2 \right] \leq 4\cN^2 \exp\left(\frac{-c_1 n t^2}{1 +  \epsilon^{-km}t}\right) + 4\exp( - c_2n \epsilon^{km}).
        \end{align}
        We consider the two bounds on the probability separately. For the second term, if $\log(1/p) \ls n\epsilon^{km}$, then we have $4 \exp(-c_2 n \epsilon^{km}) \leq p/2$.
        For the first term, the covering bound in ~\Cref{lem:covering_bound} gives $\log(\cN_\gamma) \ls \gamma^{-m}$. Then, we use the standard Bernstein inversion and set
        \begin{align}
        t \gtrsim  \sqrt{ \frac{\gamma^{-m} + \log(\frac{1}{p})}{n}} + \frac{\gamma^{-m} + \log(\frac{1}{p})}{n\epsilon^{km}} \quad \text{to get} \quad 4\cN^2 \exp\left(\frac{-c_1 n t^2}{1 +  \epsilon^{-km}t}\right) \leq p/2.
        \end{align}
        Thus, with probability at least $1 - p$, we have
        \begin{align}
            \sup_{\alpha, \beta \in \cB_\R^k} \left| \U_n[h_\epsilon(\alpha, \beta)] - \E[h_\epsilon(\alpha, \beta)]\right| \ls  \sqrt{ \frac{\gamma^{-m} + \log(\frac{1}{p})}{n}} + \frac{\gamma^{-m} + \log(\frac{1}{p})}{n\epsilon^{km}} + \gamma + \epsilon^2
        \end{align}
    \end{proof}

    \subsection{Uniform Convergence Proof} \label{ssec:uc_proof}

    Finally, we can put our bias and variance estimates together, and balance the parameters $\epsilon, \gamma$ and $n$ to prove the main result of this section.

    \begin{proof}[Proof of~\Cref{thm:main_uc}]
        First, we combine the bias estimate from~\Cref{prop:uc_bias} and the variance estimate from~\Cref{prop:uniform_variance_euclidean}. In particular, assuming the conditions of~\Cref{prop:uniform_variance_euclidean}, with probability at least $1-p$, we have
        \begin{align}
            \xi = \sup_{\alpha, \beta \in \cB_\R^k} \Big|\langle q\alpha, q\beta\rangle^{\kappa,w}_{n, \epsilon} - \sigma^\Theta_k \langle \alpha, \beta \rangle_M^{P_\Delta} \Big| \ls  \sqrt{ \frac{\gamma^{-m} + \log(\frac{1}{p})}{n}} + \frac{\gamma^{-m} + \log(\frac{1}{p})}{n\epsilon^{km}} + \gamma + \epsilon^2,
        \end{align}
        where we use the fact that $\langle q\alpha, q\beta\rangle^{\kappa,w}_{n, \epsilon} = \U_n[h_\epsilon(\alpha,\beta)]$. Now, we wish to balance the parameters $\gamma$, $\epsilon$ and $n$ to obtain the desired rate. Here, we fix some $p \in (0,1)$, and let $\ell = \log(1/p)$, so that
        \begin{align}
            \xi \ls  \sqrt{\frac{\gamma^{-m}}{n}} + \sqrt{\frac{\ell}{n}} + \frac{\gamma^{-m}}{n\epsilon^{km}} + \frac{\ell}{n\epsilon^{km}} + \gamma + \epsilon^2.
        \end{align}
        By balancing the terms which involve $\gamma$, we obtain
        \begin{align}
            \gamma \asymp \max \left\{ n^{-1/(m+2)}, (n\epsilon^{km})^{-1/(m+1)}\right\}.
        \end{align}
        Next, we balance the bias $\epsilon^2$ term with the dominant $\epsilon$-dependent stochastic term $(n\epsilon^{km})^{-1/(m+1)}$, which gives us
        \begin{align}
            \epsilon \asymp n^{-1/(km+2m+2)}
        \end{align}
        and thus 
        \begin{align}
            \gamma \asymp n^{-r_{k,m}} \quad \text{where} \quad r_{k,m} = \min \left\{ \frac{1}{m+2}, \, \frac{2}{km+2m+2}\right\}.
        \end{align}
        Thus, the error is
        \begin{align}
            \xi \ls n^{-r_{k,m}} + \sqrt{\frac{\log(1/p)}{n}} + \frac{\log(1/p)}{n^{(2m+2)/(km+2m+2)}}.
        \end{align}
    \end{proof}

\section{Proof of Codifferential Convergence} \label{sec:codiff_proof}

In this section, we will prove~\Cref{thm:main_codiff} by expanding the discrete codifferential, and considering bias and variance estimates, similar to~\Cref{sec:uc_proof}. Unlike the previous section, we will assume that the sampling distribution $\mu$ is uniform throughout. We begin with an explicit formula for the discrete codifferential. The following computation is standard for the codifferential in a weighted simplicial complex; for instance see~\cite{horak_spectra_2013} or~\cite[Lemma 3.6]{ennaceur_hodge_2025}.

\begin{lemma} \label{lem:discrete_codifferential_formula}
    For $a \in C^k(X_\epsilon)$ and $\by \in S^{k-1}_\epsilon(X)$, we have
    \begin{align} \label{eq:discrete_codifferential_formula}
    \hdelta a(\by) = \frac{k+1}{n-k} \frac{k^2}{\epsilon^{m+2}} \sum_{\substack{x \in X - \by\\ (x,\by)\in S^k_\epsilon(X)}} a(x,\by) \frac{\kappa_{\theta,\epsilon}^k(x,\by)}{\kappa_{\theta,\epsilon}^{k-1}(\by)}.
    \end{align}
\end{lemma}

Our aim is to show that the discrete codifferential energy converges to the smooth codifferential energy up to a constant. In particular, we fix a pullback form $\alpha = \pi^*\talpha$ and show that
\begin{align} \label{eq:main_codiff_result}
    \langle \hdelta q \alpha, \hdelta q \alpha \rangle_{n,\epsilon} \to \rho_0^{k+2}\sigma^\theta_{\delta,k}\langle \delta \talpha, \delta \talpha \rangle_M
\end{align}
for some constant $\sigma_{\delta, k}^\theta$.
To obtain~\Cref{thm:main_codiff}, we will conclude at the end using a polarization identity.
We use the intrinsic unit ball $\cB_M^k$ and its pullback class $\pi^*\cB_M^k$ from~\eqref{not:intrinsic_pullback_unit_balls} throughout.
For a fixed pullback form $\alpha = \pi^*\talpha$, we define
\begin{align} \label{eq:A_epsilon}
    A_\epsilon^\alpha(x, \by) = \epsilon^{-(m+2)} q\alpha(x,\by) \frac{\kappa_{\theta, \epsilon}^k(x,\by)}{\kappa_{\theta,\epsilon}^{k-1}(\by)} \quad \text{when} \quad (x,\by) \in D^k_\epsilon(M)
\end{align}
and $A^\alpha_\epsilon(x,\by) = 0$ otherwise. 
Using~\Cref{lem:discrete_codifferential_formula}, we can decompose
\begin{align} \label{eq:codiff_inner_prod}
\langle \hdelta q \alpha, \hdelta q \alpha \rangle_{n,\epsilon} = D_n^\alpha + F_n^\alpha%
\end{align}
where $D_n^\alpha$ and $F_n^\alpha$ denote the \emph{diagonal} and \emph{off-diagonal} sums,
\begin{align}
    D_n^\alpha &= \binom{n}{k}^{-1} k^2((k+1)!)^2 \frac{\epsilon^{-(k-1)(m+2)}}{(n-k)^2}\sum_{\substack{I = (i_1 < \ldots < i_k) \\p \in [n]-I}} A^\alpha_\epsilon(x_{p}, \bx_I)^2  \kappa^{k-1}_{\theta, \epsilon}(\bx_I),  \label{eq:D_n}\\
    F_n^\alpha&= \binom{n}{k}^{-1} k^2((k+1)!)^2 \frac{\epsilon^{-(k-1)(m+2)}}{(n-k)^2} \sum_{\substack{I= (i_1 < \ldots < i_k)\\p_1 \neq p_2 \in [n] - I}} A^\alpha_\epsilon(x_{p_1}, \bx_I) A^\alpha_{\epsilon}(x_{p_2}, \bx_I) \kappa^{k-1}_{\theta, \epsilon}(\bx_I). \label{eq:F_n}
\end{align}
We will often omit the superscript $\alpha$ when the form is fixed. In these expressions, the additional factor of $\kappa^{k-1}_{\theta, \epsilon}$ arises from the degree $k-1$ inner product. The proof is then split into three parts.
\begin{enumerate}
    \item In~\Cref{prop:off_diagonal_bias}, we show that the bias of $F_n^\alpha$ converges to $\rho_0^{k+2}\sigma^\theta_{\delta,k}\langle \delta\alpha,\delta\alpha\rangle_M$.
    \item We show that the diagonal term $D_n$ is of lower order.
    \item We bound the variance $F_n-\E[F_n]$ by applying Bernstein concentration for $U$-statistics.
\end{enumerate}

The following simplex expansion will be used throughout the codifferential proof.

\begin{lemma} \label{lem:codiff_expansion}
For $0<\epsilon<\epsilon_0$, the following holds. Let $(x,\by)\in D^k_\epsilon(M)$, let $v_i=\log_x(y_i)$, and write $\bv=(v_1,\ldots,v_k)$. If $\alpha=\pi^*\talpha$ is a pullback form, then
\begin{align}
    q\alpha(x,\by) = \frac{1}{k!}\alpha_x(\bv) + \frac{1}{(k+1)!} \sum_{i=1}^k (\nabla_{v_i} \alpha)_x(\bv) + R(\bv),
\end{align}
where $|R(\bv)| \ls \|\talpha\|_{\cC^2(M)}\|\bv\|^{k+2}$.
\end{lemma}
\begin{proof}
    This follows from~\Cref{lem:geometric_simplex_integral} applied to pullback forms: the second fundamental form term vanishes since $\alpha=\pi^*\talpha$ annihilates normal vectors along $M$, and tangent ambient derivatives $\bar\nabla$ restrict to intrinsic covariant derivatives $\nabla$.
\end{proof}

\subsection{Off-Diagonal Bias Estimate} \label{ssec:off_diagonal}
The expectation of the off-diagonal term $F_n$~\eqref{eq:F_n} is 
\begin{align} \label{eq:EF_n}
    \E[F_n] = \rho_0^{k+2}  \frac{k^2((k+1)!)^2(n-k-1)}{\epsilon^{(k-1)(m+2)}(n-k)} \int_{D^{k-1}_\epsilon(M)} I(\by)^2 \dvol^k(\by),
\end{align}
where we absorb the additional factor of $\kappa_{\theta,\epsilon}^{k-1}(\by)$ from \eqref{eq:F_n} into the normalized, kernel-weighted average of the simplex integral $q\alpha(x,\by)$ over all possible cone points $x$ that form an $\epsilon$-simplex with $\by$:
\begin{align} \label{eq:original_I_by}
    I(\by)
    &= \frac{1}{\epsilon^{m+2}}
    \int_{D_\epsilon(\by)} q\alpha(x,\by) \frac{\kappa_{\theta,\epsilon}^k(x,\by)}{\sqrt{\kappa_{\theta,\epsilon}^{k-1}(\by)}} \dvol(x) \andd D_\epsilon(\by) = \{ x \in M : \diam_\R(x, \by) < \epsilon\}.
\end{align}
We call $\by=(y_1,\ldots,y_k)$ the \emph{base simplex}, the points $y_i$ its \emph{base vertices}, and $x$ the \emph{cone point}. This subsection proves the following deterministic bias estimate for this off-diagonal term.

\begin{proposition} \label{prop:off_diagonal_bias}
    Uniformly for $\alpha\in \pi^*\cB_M^k$, we have
    \begin{align} \label{eq:off_diagonal_bias_goal}
        \E[F_n] = \frac{n-k-1}{n-k}\rho_0^{k+2}k^{(m+4)/2}\sigma_\lambda \langle \delta\alpha,\delta\alpha\rangle_M + O(\epsilon),
    \end{align}
    where $\sigma_\lambda>0$ depends only on $m,k$, and $\theta$.
\end{proposition}

As in~\Cref{ssec:uc_bias}, we reduce the bias to an invariant tangent-space integral. Here, choosing a canonical center for the base simplex and normalizing by its kernel require additional arguments.
\begin{enumerate}
    \item We center the base simplex at its Riemannian center of mass, called the \emph{Karcher mean}~\cite{karcher_riemannian_1977}, and use the cochain-map identity to establish the residual-symmetry identities that produce $\delta\alpha$.
    \item We replace the ambient cone domain and kernel by tangent-space counterparts and derive the resulting local expansion of $I(\by)$, treating the cutoff layer separately.
    \item We integrate in Karcher-mean coordinates and use orthogonal invariance to identify the result with a multiple of $\langle\delta\alpha,\delta\alpha\rangle_M$.
\end{enumerate}

\subsubsection{Karcher Centering and Residual Symmetry}
The properties of the Karcher mean needed below follow from~\cite{karcher_riemannian_1977}, \cite[Lemma 3]{dyer_riemannian_2015}, and~\cite[Theorems 3.2.1 and 4.3.1]{afsari_means_2009}.
\begin{proposition} \label{prop:karcher}
    For every $0<\epsilon<\epsilon_0$ and every $\by \in D^{k-1}_\epsilon(M)$, the function
    \begin{align}
        E_\by : M \to \R, \qquad E_\by(z) = \frac12 \sum_{i=1}^k d_M(z, y_i)^2,
    \end{align}
    is strictly convex on $\ball{M}{y_1}{S(2\epsilon)}$ and has a unique critical point $z=z(\by)$ in this ball, called the \emph{Karcher mean}. It is the unique global minimizer, varies smoothly on $\by$, and satisfies
    \begin{align} \label{eq:karcher_mean_tangent_zero}
        \sum_{i=1}^k \log_{z(\by)}(y_i) = 0.
    \end{align}
\end{proposition}

For the rest of the proof, fix $\by\in D^{k-1}_\epsilon(M)$ and let $z=z(\by)$ be its Karcher mean. We write
\begin{align}
    u_i = \epsilon^{-1}\log_z(y_i), \qquad \bu=(u_1,\ldots,u_k), \qquad \hat{\bu}_i=(u_1,\ldots,\hat u_i,\ldots,u_k),
\end{align}
so that $\sum_i u_i=0$. Let $H_z \subseteq T_z^kM$ be the subspace of centered $k$-tuples in $T_zM$, and define
\begin{align} \label{eq:Hz_and_cV}
    H_z = \left\{\bu\in T_z^kM : \sum_{i=1}^k u_i=0\right\} \andd \cV_z(\bu)=\{w\in T_zM:s(w,\bu)\leq 1\}.
\end{align}
Consistently with~\eqref{eq:s_squared_diameter_convention}, here $s(\bu)=\diam_{T_zM}(\bu)^2$ and $s(w,\bu)=\diam_{T_zM}(w,\bu)^2$. 
This notation is similar to $\cV_x^k\subset T_x^kM$ used in~\Cref{sec:uc_proof} as it plays a similar role; however, $\cV_z(\bu)\subset T_zM$ is the domain of a single cone vector $w$ with $\bu$ held fixed.
To formulate the residual-symmetry identities, we define a normalized kernel $K_\bu: T_z M \to  \R$ based on tangent distances, where \label{not:normalized_tangent_kernel}
\begin{align}
    K_\bu(w) \coloneqq 0 \quad \text{when}\quad  s(\bu) \geq 1 \andd K_\bu(w) \coloneqq \dfrac{\theta(s(w,\bu))}{\sqrt{\theta(s(\bu))}} \quad \text{when} \quad s(\bu) < 1.
\end{align} with the cases preventing a vanishing denominator. 
We approximate the $q\alpha(x,\by)$ term in $I(\by)$ using the discretizations $q\alpha(z, x, \hat{\by}_i)$ on simplices based at the Karcher mean $z = z(\by)$.

\begin{lemma} \label{lem:codiff_stokes_reduction}
    With the notation above,
    \begin{align} \label{eq:codiff_I_stokes_decomposition}
        I(\by) &= \epsilon^{-(m+2)} \sum_{i=1}^k (-1)^i I_i(\by) + O(\epsilon^k) \quad \text{where} \quad 
        I_i(\by)= 
        \int_{D_\epsilon(\by)} q\alpha(z,x,\hat{\by}_i) \frac{\kappa_{\theta,\epsilon}^k(x,\by)}{\sqrt{\kappa_{\theta,\epsilon}^{k-1}(\by)}} \dvol(x).
    \end{align}
\end{lemma}
\begin{proof}
    Applying ~\Cref{lem:cochain_map} to the Euclidean $(k+1)$-simplex $(z,x,\by)$, we get
    \begin{align} \label{eq:first_stokes}
        q(d\alpha)(z, x, \by) = \hd (q\alpha)(z,x,\by) = q\alpha(x, \by) - q\alpha(z, \by) + \sum_{i=1}^k (-1)^{i+1} q\alpha(z, x, \hat{\by}_i).
    \end{align}
    We first bound the two terms not appearing in~\eqref{eq:codiff_I_stokes_decomposition}. Applying~\Cref{lem:codiff_expansion} to $q\alpha(z,\by)$, the leading term and first-order term vanish because $\sum_i u_i=0$. Hence $|q\alpha(z,\by)|\ls \epsilon^{k+2}$. Similarly, applying the same expansion to the $(k+1)$-form $d\alpha$ shows that $|q(d\alpha)(z,x,\by)|\ls \epsilon^{k+2}$: the tangent vectors $(\epsilon w,\epsilon u_1,\ldots,
    \epsilon u_k)$ are linearly dependent because $\sum_i u_i=0$. Therefore,
    \begin{align}
        q\alpha(x,\by)=\sum_{i=1}^k (-1)^i q\alpha(z,x,\hat{\by}_i)+O(\epsilon^{k+2}).
    \end{align}
    Substituting this into~\eqref{eq:original_I_by} and using $\vol(D_\epsilon(\by))\ls \epsilon^m$ gives the result.
\end{proof}

The following lemma is the local trace mechanism which produces the codifferential. It is analogous to the symmetry argument in~\Cref{lem:bilinear_form_trace}, but after fixing the centered base tuple $\bu$, the kernel $K_\bu(w)$ is invariant only under the subgroup of $O(T_zM)$ fixing $W_\bu=\spann\{u_1,\ldots,u_k\}$ pointwise. This residual symmetry kills the leading moment and makes the second moment scalar on $W_\bu^\perp$. By Karcher centering, $\sum_i u_i=0$, so alternation removes the $W_\bu$-directions and leaves the transverse trace in the normal-coordinate formula for $\delta\alpha$.

\begin{lemma} \label{lem:codiff_moment_trace}
    For any $i\in\{1,\ldots,k\}$, we have
    \begin{align} \label{eq:codiff_moment_trace_formulas}
        \int_{T_zM} \alpha_z(w, \hat{\bu}_i)K_\bu(w)dw = 0
    \end{align}
    and
    \begin{align}
        \int_{T_zM} (\nabla_w\alpha)_z(w,\hat{\bu}_i)K_\bu(w)dw
        = -\lambda_\theta(\bu)(\delta\alpha)_z(\hat{\bu}_i),
    \end{align}
    where, for any unit vector $e^\perp\in W_\bu^\perp$ and $W_\bu=\spann\{u_1,\ldots,u_k\}$,
    \begin{align}
        \lambda_\theta(\bu)\coloneqq \int_{T_zM}\langle w,e^\perp\rangle^2K_\bu(w)dw.
    \end{align}
\end{lemma}
\begin{proof}
    Let $G\subset O(T_zM)$ be the subgroup fixing $W_\bu$ pointwise. Since $K_\bu(w)$ is invariant under the action $K_\bu(w) \mapsto K_\bu(g w)$ for $g \in G$, the vector
    \begin{align}
        w_0=\int_{T_zM} w\,K_\bu(w)dw
    \end{align}
    is fixed by $G$ and hence lies in $W_\bu$. The $k$ vectors $w_0,u_1,\ldots,\hat u_i,\ldots,u_k$ all lie in $W_\bu$, which has dimension at most $k-1$, so they are linearly dependent. This proves~\eqref{eq:codiff_moment_trace_formulas} by 
    \begin{equation}
        \int_{T_zM} \alpha_z(w, \hat{\bu}_i)K_\bu(w)dw = \alpha_z \Big( \int_{T_zM}w K_\bu(w) dw, \hat{\bu}_i \Big) = \alpha_z(w_0,\hat{\bu}_i) = 0.
    \end{equation}
    
    For the second identity, let $B(v_1,v_2)=(\nabla_{v_1}\alpha)_z(v_2,\hat{\bu}_i)$, and let $e_1,\ldots,e_r$ and $e_{r+1},\ldots,e_m$ be orthonormal bases of $W_\bu$ and $W_\bu^\perp$. Then
    \begin{align}
        \int_{T_zM} B(w,w)K_\bu(w)dw = \sum_{p,q} M_{pq}B(e_p,e_q) \quad \text{where} \quad M_{pq}=\int_{T_zM}\langle w,e_p\rangle\langle w,e_q\rangle K_\bu(w)dw.
    \end{align}
    If $e_p,e_q\in W_\bu$, then $B(e_p,e_q)=0$ because $e_q,u_1,\ldots,\hat u_i,\ldots,u_k$ are linearly dependent. Mixed terms vanish by reflecting the $W_\bu^\perp$ component. On $W_\bu^\perp$, the same reflection and rotation argument gives $M_{pq}=\lambda_\theta(\bu)\delta_{pq}$. Therefore,
    \begin{align}
        \int_{T_zM} (\nabla_w\alpha)_z(w,\hat{\bu}_i)K_\bu(w)dw
        = \lambda_\theta(\bu)\sum_{e_p\in W_\bu^\perp}(\nabla_{e_p}\alpha)_z(e_p,\hat{\bu}_i).
    \end{align}
    Finally, the normal-coordinate formula for the codifferential~\cite[Eq. 3.3.47]{jost_riemannian_2013} gives
    \begin{align}
        (\delta\alpha)_z(\hat{\bu}_i)=-\sum_{p=1}^m(\nabla_{e_p}\alpha)_z(e_p,\hat{\bu}_i)
        =-\sum_{e_p\in W_\bu^\perp}(\nabla_{e_p}\alpha)_z(e_p,\hat{\bu}_i),
    \end{align}
    because the summands with $e_p\in W_\bu$ vanish as above.
\end{proof}

\subsubsection{Tangent-Space Reduction} \label{sssec:tangent_space_reduction}
We will now consider the codifferential analogue of the kernel and domain replacement in~\Cref{lem:kappa_domain_replacement} from the inner product convergence proof; however, normalization by the base kernel requires the cutoff layer to be treated separately.
We first treat base simplices lying a fixed multiple of $\epsilon^2$ inside the cutoff, for which we obtain an $O(\epsilon^2)$ tangent-space approximation.

\begin{lemma} \label{lem:codiff_I_support}
    Let
    \begin{align}
        \tL_\epsilon(\by)=\{w\in T_zM: \diam_\R(\exp_z(\epsilon w),\by)<\epsilon\}.
    \end{align}
    There is a constant $C_\partial=C_\partial(M,k)>0$, called the \emph{cutoff-separation constant}, such that, if $s(\bu)\leq 1-C_\partial\epsilon^2$ and $f:T_zM\to \R$ is bounded by $C_f> 0$ on $\cV_z(\bu)\cup \tL_\epsilon(\by)$, then
    \begin{align} \label{eq:codiff_support_integral}
        \left|\int_{T_zM} f(w)(\widetilde{K}_{\epsilon,\by}(w)-K_\bu(w))dw\right| 
        \ls C_f\epsilon^2 \quad \text{where} \quad \widetilde{K}_{\epsilon,\by}(w) = \frac{\kappa^k_{\theta,\epsilon}(\exp_z(\epsilon w),\by)}{\sqrt{\kappa^{k-1}_{\theta,\epsilon}(\by)}}\one_{\tL_\epsilon(\by)}(w).
    \end{align}
\end{lemma}
\begin{proof}
    On a common uniformly bounded set containing $\cV_z(\bu)$ and $\tL_\epsilon(\by)$, ~\eqref{eq:y_u_comparison} gives
    \begin{align}
        \epsilon^{-2}\|\exp_z(\epsilon w)-y_i\|^2 = \|w-u_i\|^2+O(\epsilon^2).
    \end{align}
    The same comparison applied to pairs of base vertices gives $\epsilon^{-2}s(\by)=s(\bu)+O(\epsilon^2)$. Therefore
    \begin{align}
        \left|\epsilon^{-2}s(\exp_z(\epsilon w),\by)-s(w,\bu)\right|\ls \epsilon^2.
    \end{align}
    Since $s(\bu)\leq 1-C_\partial\epsilon^2$ and $C_\partial$ is chosen large enough, the denominator $\sqrt{\kappa^{k-1}_{\theta,\epsilon}(\by)}$ is compared inside $[0,1)$ to $\sqrt{\theta(s(\bu))}$, and the Lipschitz bound for $\theta$ together with the lower bound $c_\theta$ gives an $O(\epsilon^2)$ pointwise error on the common support. The possible disagreement of supports is contained in the shell
    \begin{align}
        \left\{w:\left|s(w,\bu)-1\right|\ls \epsilon^2\right\},
    \end{align}
    which has volume $O(\epsilon^2)$ in the uniformly bounded region under consideration. Since $\theta$ is bounded, the shell contributes the second $O(\epsilon^2)$ error.
\end{proof}

After increasing $C_\partial$ if necessary, the comparison $\epsilon^{-2}s(\by)=s(\bu)+O(\epsilon^2)$ shows that every remaining base simplex lies in the \emph{cutoff layer} $\{|s(\bu)-1|\leq C_\partial\epsilon^2\}$. On this layer, a cutoff-adapted tangent kernel gives the uniform bound $|I(\by)|\ls\epsilon^{k-1}$ needed in later estimates.

\begin{lemma} \label{lem:codiff_uniform_I_bound}
    Uniformly for $\alpha\in\pi^*\cB_M^k$ and $\by\in D^{k-1}_\epsilon(M)$, we have $|I(\by)|\ls\epsilon^{k-1}$.
\end{lemma}
\begin{proof}
    Set $b_\epsilon(\by)=\epsilon^{-2}s(\by)<1$ and define the cutoff-adapted tangent kernel
    \begin{align}
        K^\partial_{\epsilon,\by}(w)
        =\frac{\theta\left(\max\left\{b_\epsilon(\by),\max_{1\leq j\leq k}\|w-u_j\|^2\right\}\right)}{\sqrt{\theta(b_\epsilon(\by))}}.
    \end{align}
    Its nonzero set is $\{w\in T_zM:\max_j\|w-u_j\|^2<1\}$. Both $\widetilde K_{\epsilon,\by}$ and $K^\partial_{\epsilon,\by}$ use the same quantity $b_\epsilon(\by)$ in their cutoff and have the same denominator $\sqrt{\theta(b_\epsilon(\by))}\geq\sqrt{c_\theta}$. They differ only in that the ambient distances $\epsilon^{-2}\|\exp_z(\epsilon w)-y_j\|^2$ are replaced by the tangent distances $\|w-u_j\|^2$. By~\eqref{eq:y_u_comparison}, these distances differ by $O(\epsilon^2)$, so the same common-support and thin-shell argument as in~\Cref{lem:codiff_I_support} gives the $L^1$ estimate below without any cutoff-separation condition. Consequently,
    \begin{align} \label{eq:codiff_boundary_kernel_comparison}
        \|\widetilde K_{\epsilon,\by}-K^\partial_{\epsilon,\by}\|_{L^1(T_zM)}\ls\epsilon^2.
    \end{align}

    Since $K^\partial_{\epsilon,\by}$ has the same invariance under orthogonal maps fixing $W_\bu$ pointwise as $K_\bu$, the proof of the first identity in~\Cref{lem:codiff_moment_trace} applies verbatim and gives $\int_{T_zM}\alpha_z(w,\hat\bu_i)K^\partial_{\epsilon,\by}(w)dw=0$. Combining this cancellation with \eqref{eq:codiff_boundary_kernel_comparison}, ~\Cref{lem:codiff_expansion}, and $J(z,\epsilon w)=1+O(\epsilon^2)$, we obtain $|I_i(\by)|\ls\epsilon^{m+k+1}$; the $O(\epsilon^{k+1})$ simplex-expansion remainder is the leading error, while the kernel-comparison and Jacobian errors are $O(\epsilon^{m+k+2})$. The Stokes reduction in~\Cref{lem:codiff_stokes_reduction} therefore gives $|I(\by)|\ls\epsilon^{-(m+2)}\sum_i|I_i(\by)|+\epsilon^k\ls\epsilon^{k-1}$.
\end{proof}

We now compute $I(\by)$ on the region $s(\bu)\leq1-C_\partial\epsilon^2$. Changing variables $x=\exp_z(\epsilon w)$ gives
\begin{align}
    I_i(\by)=\epsilon^m\int_{\tL_\epsilon(\by)}q\alpha(z,x,\hat{\by}_i)\frac{\kappa^k_{\theta,\epsilon}(x,\by)}{\sqrt{\kappa^{k-1}_{\theta,\epsilon}(\by)}}J(z,\epsilon w)dw.
\end{align}
Using~\Cref{lem:codiff_I_support}, the Jacobian estimate $J(z,\epsilon w)=1+O(\epsilon^2)$, and $|q\alpha(z,x,\hat{\by}_i)|\ls \epsilon^k$, we obtain
\begin{align}
    I_i(\by)=\epsilon^m\int_{T_zM}q\alpha(z,x,\hat{\by}_i)K_\bu(w)dw+O(\epsilon^{m+k+2}).
\end{align}
By~\Cref{lem:codiff_expansion},
\begin{align} \label{eq:codiff_stokes_simplex_expansion}
    q\alpha(z,x,\hat{\by}_i) = \frac{\epsilon^k}{k!}\alpha_z(w, \hat{\bu}_i) + \frac{\epsilon^{k+1}}{(k+1)!} \left( (\nabla_w \alpha)_z(w, \hat{\bu}_i) +  \sum_{j \neq i} (\nabla_{u_j} \alpha)_z(w, \hat{\bu}_i) \right) + O(\epsilon^{k+2}).
\end{align}
The first term vanishes by~\Cref{lem:codiff_moment_trace}. The third term also vanishes by applying the first identity in~\Cref{lem:codiff_moment_trace} to the $k$-form $(\nabla_{u_j}\alpha)_z$. The second term gives
\begin{align} \label{eq:I_i_approximation}
    I_i(\by) = -\frac{\epsilon^{m+k+1}}{(k+1)!}\lambda_\theta(\bu)(\delta\alpha)_z(\hat{\bu}_i)+O(\epsilon^{m+k+2}).
\end{align}
Combining~\eqref{eq:I_i_approximation} with~\eqref{eq:codiff_I_stokes_decomposition}, we get
\begin{align} \label{eq:I_by_expansion}
    I(\by) = \frac{\epsilon^{k-1}}{(k+1)!}\lambda_\theta(\bu)\sum_{i=1}^k(-1)^{i+1}(\delta\alpha)_z(\hat{\bu}_i)+O(\epsilon^k) = \frac{\epsilon^{k-1}}{(k+1)!}\lambda_\theta(\bu)(\delta\alpha)_z(\Phi(\bu))+O(\epsilon^k),
\end{align}
where $\Phi(\bu)=(u_2-u_1)\wedge\cdots\wedge(u_k-u_1)$, since by direct computation, we have
\begin{align}
    \sum_{i=1}^k(-1)^{i+1}(\delta\alpha)_z(\hat{\bu}_i)=(\delta\alpha)_z(\Phi(\bu)) = (\delta\alpha)_z(u_2-u_1, \ldots, u_k-u_1).
\end{align}
The expansion above holds uniformly for base tuples $\by\in D^{k-1}_\epsilon(M)$ whose centered Karcher coordinates satisfy $s(\bu)\leq 1-C_\partial\epsilon^2$, where $C_\partial$ is the cutoff-separation constant from~\Cref{lem:codiff_I_support}. \medskip

\subsubsection{Integration in Karcher Coordinates}
It remains to integrate the leading term in~\eqref{eq:I_by_expansion}.

\begin{lemma} \label{lem:karcher_coordinate_bias}
    Uniformly for $\alpha\in\pi^*\cB_M^k$,
    \begin{align}
        \int_{D^{k-1}_\epsilon(M)} I(\by)^2\dvol^k(\by)
        =\frac{k^{m/2}\epsilon^{(k-1)(m+2)}}{((k+1)!)^2}\int_M\int_{H_z}(\delta\alpha)_z(\Phi(\bu))^2\lambda_\theta(\bu)^2d\bu\dvol(z)+O(\epsilon^{(k-1)(m+2)+1}).
    \end{align}
\end{lemma}
\begin{proof}
    Define $\Psi_\epsilon(z,\bu)=(\exp_z(\epsilon u_1),\ldots,\exp_z(\epsilon u_k))$ on the region $\bu\in H_z$ with $s(\bu)<2$. By~\eqref{eq:y_u_comparison}, every $\by\in D^{k-1}_\epsilon(M)$ has centered Karcher coordinates in this region. Conversely, if $\Psi_\epsilon(z,\bu)=\by\in D^{k-1}_\epsilon(M)$, then
    \begin{align}
        \nabla E_\by(z)=-\sum_{i=1}^k\log_z(y_i)=-\epsilon\sum_{i=1}^ku_i=0.
    \end{align}
    Centering gives $\|u_i\|\leq\diam_{T_zM}(\bu)<\sqrt{2}$, so $z\in\ball{M}{y_1}{S(2\epsilon)}$. The uniqueness in~\Cref{prop:karcher} therefore gives $z=z(\by)$ and $u_i=\epsilon^{-1}\log_{z(\by)}(y_i)$. Thus $\Psi_\epsilon$ restricts to a smooth one-to-one coordinate map onto $D^{k-1}_\epsilon(M)$ with smooth inverse given by these formulas.

    Using the connection to split tangent vectors to the centered bundle into horizontal and vertical parts $(\xi,\eta)$, the Jacobi equation gives, after parallel transport back to $T_zM$,
    \begin{align}
        D\Psi_\epsilon(\xi,\eta)_i
        =\xi+\epsilon\eta_i+O\left(\epsilon^2\|\xi\|+\epsilon^3\|\eta_i\|\right),
    \end{align}
    uniformly on $s(\bu)<2$. Since $\sum_i\eta_i=0$, the diagonal and centered parts of the leading map are orthogonal. The diagonal part has volume factor $k^{m/2}$, the centered part has volume factor $\epsilon^{(k-1)m}$, and the displayed error changes their product by a relative factor $1+O(\epsilon^2)$. Hence
    \begin{align}
        \dvol^k(\by)=k^{m/2}\epsilon^{(k-1)m}(1+O(\epsilon^2))d\bu\dvol(z).
    \end{align}
    The comparison $\epsilon^{-2}s(\by)=s(\bu)+O(\epsilon^2)$ and the choice of $C_\partial$ show that the image of $D^{k-1}_\epsilon(M)$ is contained in $\{s(\bu)<1+C_\partial\epsilon^2\}$ in these coordinates. On the region $s(\bu)\leq 1-C_\partial\epsilon^2$, with $C_\partial$ as in~\Cref{lem:codiff_I_support}, the expansion~\eqref{eq:I_by_expansion} gives
    \begin{align}
        I(\by)^2=\frac{\epsilon^{2k-2}}{((k+1)!)^2}\lambda_\theta(\bu)^2(\delta\alpha)_z(\Phi(\bu))^2+O(\epsilon^{2k-1}),
    \end{align}
    uniformly for $\alpha\in\pi^*\cB_M^k$. The cutoff layer $\{|s(\bu)-1|\leq C_\partial\epsilon^2\}$ has volume $O(\epsilon^2)$ inside the uniformly bounded coordinate region. By~\Cref{lem:codiff_uniform_I_bound}, $|I(\by)|\ls\epsilon^{k-1}$ on this layer, so its total contribution is $O(\epsilon^{(k-1)(m+2)+2})$. The Jacobian error contributes the same or a smaller order, proving the lemma.
\end{proof}

\begin{proof}[Proof of \Cref{prop:off_diagonal_bias}]

By~\Cref{lem:bilinear_form_trace}, applied with $H=H_z$ and the $O(T_zM)$-equivariant map $\Phi:H_z\to \Lambda^{k-1}T_zM$, there exists a constant $\sigma_\lambda$ such that
\begin{align}
    \int_{H_z}(\delta\alpha)_z(\Phi(\bu))^2\lambda_\theta(\bu)^2d\bu=\sigma_\lambda\langle (\delta\alpha)_z,(\delta\alpha)_z\rangle_z.
\end{align}
Moreover, $\sigma_\lambda>0$ because $\lambda_\theta(\bu)>0$ on a positive-measure subset of $\{s(\bu)<1\}$, and the trace constant is obtained by applying the preceding identity to a nonzero $(k-1)$-form.
Integrating the above identity with respect to the Riemannian volume density gives
\begin{align}
    \int_M\int_{H_z}(\delta\alpha)_z(\Phi(\bu))^2\lambda_\theta(\bu)^2d\bu\dvol(z)=\sigma_\lambda\langle \delta\alpha,\delta\alpha\rangle_M.
\end{align}
Finally, \Cref{lem:karcher_coordinate_bias} gives
\begin{align}
    \int_{D^{k-1}_\epsilon(M)}I(\by)^2\dvol^k(\by)
    =\frac{k^{m/2}\epsilon^{(k-1)(m+2)}}{((k+1)!)^2}\sigma_\lambda\langle\delta\alpha,\delta\alpha\rangle_M+O(\epsilon^{(k-1)(m+2)+1}).
\end{align}
Substituting this into~\eqref{eq:EF_n} proves~\Cref{prop:off_diagonal_bias}:
\begin{align}
    \E[F_n] = \frac{n-k-1}{n-k}\rho_0^{k+2}k^{(m+4)/2}\sigma_\lambda\langle\delta\alpha,\delta\alpha\rangle_M+O(\epsilon).
\end{align}
\end{proof}

\subsection{Diagonal Term}
Now, we will consider the diagonal term~\eqref{eq:D_n}.

\begin{lemma} \label{lem:codiff_diagonal_exp_bound}
    On the event $\cE_\epsilon$ from~\Cref{lem:concentration_D_k_support}, we have, for $n\geq 2k$,
    \begin{align}
        \sup_{\alpha \in \pi^*\cB_M^k} D_n^\alpha \ls \frac{1}{n\epsilon^{m+2}}.
    \end{align}
\end{lemma}
\begin{proof}
    By the elementary bound $|q\alpha(x,\by)|\ls \epsilon^k$ for $\alpha\in\pi^*\cB_M^k$, the support of $\theta$, and the lower bound $\kappa_{\theta,\epsilon}^{k-1}(\by)\geq c_\theta$ whenever $\kappa_{\theta,\epsilon}^k(x,\by)\neq 0$, we have
    \begin{align}
        |A^\alpha_\epsilon(x,\by)| = \big\lvert \epsilon^{-(m+2)} q\alpha(x,\by) \frac{\kappa_{\theta, \epsilon}^k(x,\by)}{\kappa_{\theta,\epsilon}^{k-1}(\by)} \big\rvert \ls \epsilon^{k-m-2}\one_{D^k_\epsilon(M)}(x,\by).
    \end{align}
    Therefore,
    \begin{align}
        D_n^\alpha
        &\ls \binom{n}{k}^{-1}\frac{\epsilon^{-(k-1)(m+2)}}{(n-k)^2}
        \sum_{\substack{I=(i_1<\cdots<i_k)\\ p\in[n]-I}}
        \epsilon^{2k-2m-4}\one_{D^k_\epsilon(M)}(x_p,\bx_I).
    \end{align}
    Each unordered $(k+1)$-tuple is counted $k+1$ times in the last sum, and hence
    \begin{align}
        \binom{n}{k}^{-1}\frac{1}{(n-k)^2}
        \sum_{\substack{I=(i_1<\cdots<i_k)\\ p\in[n]-I}}
        \one_{D^k_\epsilon(M)}(x_p,\bx_I)
        =\frac{1}{n-k}\U_n[\one_{D^k_\epsilon(M)}].
    \end{align}
    On the event $\cE_\epsilon$, this gives
    \begin{align}
        D_n^\alpha\ls \frac{1}{n-k}\epsilon^{-(k-1)(m+2)}\epsilon^{2k-2m-4}\epsilon^{km}
        \ls \frac{1}{n\epsilon^{m+2}},
    \end{align}
    uniformly over $\alpha\in\pi^*\cB_M^k$.
\end{proof}

\subsection{Variance Term}
To control the variance term $F_n-\E[F_n]$, we largely follow the structure of~\Cref{sec:uc_proof}. We first establish $F_n$ as an order $k+2$ $U$-statistic and concentration bounds for fixed forms; then we use a covering argument to transfer this to a uniform estimate. For $\alpha=\pi^*\talpha\in\Omega^k(U_\epsilon)$, define $h^F_\epsilon(\alpha):M^{k+2}\to\R$ by
\begin{align} \label{eq:h_F}
    h^F_\epsilon(\alpha)(y_1,\ldots,y_{k+2})
    =\frac{k^2((k+1)!)^2\epsilon^{-(k-1)(m+2)}}{(k+1)(k+2)}
    \sum_{i\neq j}A^\alpha_\epsilon(y_i,\hat\by_{i,j})A^\alpha_\epsilon(y_j,\hat\by_{i,j})\kappa^{k-1}_{\theta,\epsilon}(\hat\by_{i,j}),
\end{align}
where $\hat\by_{i,j}=(y_1,\ldots,\hat y_i,\ldots,\hat y_j,\ldots,y_{k+2})$. Then,
\begin{align} \label{eq:F_as_U_stat}
    F_n^\alpha=\frac{n-k-1}{n-k}\U_n[h^F_\epsilon(\alpha)].
\end{align}

\begin{lemma} \label{lem:codiff_pointwise_variance}
    For $\alpha\in\pi^*\cB_M^k$, we have
    \begin{align}
        \PP[|F_n^\alpha-\E[F_n^\alpha]|>t]
        \leq 4\exp\left(\frac{-cnt^2}{A\epsilon^{-2}+B\epsilon^{-m(k+1)-2}t}\right),
    \end{align}
    where $A,B>0$ depend only on $M,N,k$, and $\theta$.
\end{lemma}
\begin{proof}
    First, because $|A^\alpha_\epsilon(x,\by)|\ls\epsilon^{k-m-2}$, we have
    \begin{align}
        \|h^F_\epsilon(\alpha)\|_\infty\ls \epsilon^{-m(k+1)-2}.
    \end{align}
    Next, we bound the first projection
    \begin{align}
        g(x)=\E_{\by\sim\mu^{k+1}}[h^F_\epsilon(\alpha)(x,\by)].
    \end{align}
    From the definition of $h^F_\epsilon$ in~\eqref{eq:h_F}, there are two cases to consider. First, suppose $x$ is one of the two distinguished variables. Then, after integrating the other distinguished variable, we use~\eqref{eq:original_I_by} to obtain the bound
    \begin{align}
        &\epsilon^{-(k-1)(m+2)}\int A^\alpha_\epsilon(x,\by)A^\alpha_\epsilon(z,\by)\kappa^{k-1}_{\theta,\epsilon}(\by)\,d\vol(z)d\vol^k(\by) \\
        &\qquad\ls \epsilon^{-(k-1)(m+2)}\epsilon^{k-m-2}\epsilon^{k-1}\epsilon^{km}=\epsilon^{-1},
    \end{align}
    where the integration is over the support of the integrand. After integrating in $z$, \eqref{eq:original_I_by} gives the factor $\sqrt{\kappa^{k-1}_{\theta,\epsilon}(\by)}I(\by)$; since the kernel is uniformly bounded, we used $|I(\by)|\ls\epsilon^{k-1}$ and the volume bound in~\eqref{eq:fat_diagonal_fiber_volume}. Second, suppose $x$ is part of the base $k$-tuple. Then
    \begin{align}
        &\epsilon^{-(k-1)(m+2)}\int A^\alpha_\epsilon(z_1,x,\by)A^\alpha_\epsilon(z_2,x,\by)\kappa^{k-1}_{\theta,\epsilon}(x,\by)\,d\vol(z_1)d\vol(z_2)d\vol^{k-1}(\by) \\
        &\qquad\ls \epsilon^{-(k-1)(m+2)}\int_{D^{k-1}_\epsilon(M,x)} I(x,\by)^2\,d\vol^{k-1}(\by)
        \ls \epsilon^{-(k-1)(m+2)}\epsilon^{2k-2}\epsilon^{(k-1)m}=1,
    \end{align}
    Here \eqref{eq:original_I_by}, applied to the base $k$-tuple $(x,\by)$, identifies the $z_1,z_2$-integral with $I(x,\by)^2$, and the same definition together with the lower bound for $\kappa^{k-1}_{\theta,\epsilon}$ on the support of the numerator gives $|I(x,\by)|\ls\epsilon^{k-1}$. We also use~\eqref{eq:fat_diagonal_fiber_volume} with $k-1$ in place of $k$ for the volume of $D^{k-1}_\epsilon(M,x)$. Therefore, $\|g\|_\infty\ls\epsilon^{-1}$, and
    \begin{align}
        \Xi_F=\Var(g(x))\leq \E[g(x)^2]\ls\epsilon^{-2}.
    \end{align}
    Applying~\Cref{thm:bernstein} to $\U_n[h^F_\epsilon(\alpha)]$, and using~\eqref{eq:F_as_U_stat}, proves the stated bound.
\end{proof}

\begin{lemma} \label{lem:codiff_covering_bound}
    Let $\cN_\gamma$ be the smallest number of $\cC^0(M)$-balls of radius $\gamma>0$, with centers in $\cB_M^k$, required to cover $\cB_M^k$. Then
    \begin{align}
        \log \cN_\gamma\ls \gamma^{-m/2}.
    \end{align}
\end{lemma}
\begin{proof}
        This follows from the classical entropy estimate for Hölder balls~\cite[Theorem 2.7.1]{vaart_weak_2023}. Applying this estimate in a finite atlas and finite local trivializations of $\Lambda^kT^*M$ gives the stated bound for the $\cC^2(M)$ unit ball. As in \Cref{lem:covering_bound}, after forming the cover we retain only balls meeting $\cB_M^k$ and recenter them at points of $\cB_M^k$.
\end{proof}

\begin{proposition} \label{prop:off_diagonal_covering_variance}
    There is a constant $C_0>0$, depending only on the fixed data, such that the following holds. Let $p\in(0,1)$ and suppose $C_0\log(8/p)\leq n\epsilon^{(k+1)m}$. With probability at least $1-p$, we have
    \begin{align}
        \sup_{\alpha\in\pi^*\cB_M^k}|F_n^\alpha-\E[F_n^\alpha]|
        \ls \sqrt{\frac{\gamma^{-m/2}+\log(1/p)}{n\epsilon^2}}
        +\frac{\gamma^{-m/2}+\log(1/p)}{n\epsilon^{m(k+1)+2}}
        +\gamma\epsilon^{-2}.
    \end{align}
\end{proposition}
\begin{proof}
    Let $\{\ball{\cC^0(M)}{\teta^r}{\gamma}\}_{r=1}^{\cN}$ be a collection of $\cC^0(M)$-balls which cover $\cB_M^k$, with centers $\teta^r\in\cB_M^k$. Let $\eta^r=\pi^*\teta^r$. By a union bound and~\Cref{lem:codiff_pointwise_variance},
    \begin{align} \label{eq:codiff_union_variance}
        \PP\left[\sup_{r\in[\cN]}|F_n^{\eta^r}-\E[F_n^{\eta^r}]|>t\right]
        \leq 4\cN\exp\left(\frac{-cnt^2}{A\epsilon^{-2}+B\epsilon^{-m(k+1)-2}t}\right).
    \end{align}
    Let $\alpha=\pi^*\talpha\in\pi^*\cB_M^k$. Choose $\eta^r$ such that $\|\talpha-\teta^r\|_{\cC^0(M)}<\gamma$. Then, by the definition of $A_\epsilon$ and the lower bound $\kappa_{\theta,\epsilon}^{k-1}(\by)\geq c_\theta$ on the support of $\kappa_{\theta,\epsilon}^{k}(x,\by)$,
    \begin{align}
        |A^\alpha_\epsilon(x,\by)-A^{\eta^r}_\epsilon(x,\by)|\ls \gamma\epsilon^{k-m-2}\one_{D^k_\epsilon(M)}(x,\by).
    \end{align}
    If $\cD_{F,\epsilon}(M)$ denotes the support of $h^F_\epsilon$, then $\cD_{F,\epsilon}(M)\subset D^{k+1}_{2\epsilon}(M)$. Hence
    \begin{align}
        |h^F_\epsilon(\alpha)-h^F_\epsilon(\eta^r)|\ls \gamma\epsilon^{-m(k+1)-2}\one_{D^{k+1}_{2\epsilon}(M)}.
    \end{align}
    Taking expectations gives
    \begin{align}
        \sup_{\alpha\in\pi^*\cB_M^k}|\E[h^F_\epsilon(\alpha)-h^F_\epsilon(\eta^r)]|\ls \gamma\epsilon^{-2}.
    \end{align}
    By the same argument as~\Cref{lem:concentration_D_k_support}, applied to $D^{k+1}_{2\epsilon}(M)$, there is an event
    \begin{align}
        \cF_\epsilon \coloneqq \left\{\U_n[\one_{D^{k+1}_{2\epsilon}(M)}]\leq C\epsilon^{(k+1)m}\right\}
    \end{align}
    such that $\PP[\cF_\epsilon^c]\leq 4\exp(-cn\epsilon^{(k+1)m})$. This is distinct from $\cE_\epsilon$, which controls the $D^k_\epsilon(M)$ support. On $\cF_\epsilon$,
    \begin{align}
        \sup_{\alpha\in\pi^*\cB_M^k}|\U_n[h^F_\epsilon(\alpha)-h^F_\epsilon(\eta^r)]|\ls \gamma\epsilon^{-2}.
    \end{align}
    Combining this with~\eqref{eq:codiff_union_variance} and the triangle inequality gives
    \begin{align}
        \PP\left[\sup_{\alpha\in\pi^*\cB_M^k}|F_n^\alpha-\E[F_n^\alpha]|>t+C\gamma\epsilon^{-2}\right]
        \leq 4\cN\exp\left(\frac{-cnt^2}{A\epsilon^{-2}+B\epsilon^{-m(k+1)-2}t}\right)+4\exp(-cn\epsilon^{(k+1)m}).
    \end{align}
    Using~\Cref{lem:codiff_covering_bound} and the standard Bernstein inversion with
    \begin{align}
        t\gtrsim \sqrt{\frac{\gamma^{-m/2}+\log(1/p)}{n\epsilon^2}}
        +\frac{\gamma^{-m/2}+\log(1/p)}{n\epsilon^{m(k+1)+2}}
    \end{align}
    proves the result.
\end{proof}

\subsection{Proofs of Main Results}

\begin{proof}[Proof of~\Cref{thm:main_codiff}]
    First, we prove the quadratic estimate uniformly over $\pi^*\cB_M^k$. Let
    \begin{align}
        \xi=\sup_{\alpha\in\pi^*\cB_M^k}\left|\langle\hdelta q\alpha,\hdelta q\alpha\rangle_{n,\epsilon}-\rho_0^{k+2}\sigma^\theta_{\delta,k}\langle\delta\alpha,\delta\alpha\rangle_M\right|,
        \qquad \sigma^\theta_{\delta,k}=k^{(m+4)/2}\sigma_\lambda.
    \end{align}
    Combining~\Cref{prop:off_diagonal_bias},~\Cref{lem:codiff_diagonal_exp_bound}, and~\Cref{prop:off_diagonal_covering_variance}, and using the concentration event for $\U_n[\one_{D^k_\epsilon(M)}]$ from~\Cref{lem:concentration_D_k_support}, gives with probability at least $1-p$,
    \begin{align} \label{eq:codiff_combined_rate_before_balance}
        \xi\ls \epsilon+\frac1n+\frac{1}{n\epsilon^{m+2}}
        +\sqrt{\frac{\gamma^{-m/2}+\log(1/p)}{n\epsilon^2}}
        +\frac{\gamma^{-m/2}+\log(1/p)}{n\epsilon^{m(k+1)+2}}
        +\gamma\epsilon^{-2},
    \end{align}
    provided $C_0\log(12/p)\leq n\epsilon^{(k+1)m}$ for a sufficiently large constant $C_0$ depending only on the fixed data.
    The factor $\log(12/p)$ comes from allocating failure probability $p/3$ to each prefactor-$4$ tail: the finite-cover Bernstein tail, the $D^{k+1}_{2\epsilon}(M)$ support tail, and the $D^k_\epsilon(M)$ diagonal support tail.

    We now balance $\gamma$ and $\epsilon$. Set $a=m/2$. The terms involving $\gamma$ are balanced by taking
    \begin{align}
        \gamma\asymp \max\left\{n^{-1/(a+2)}\epsilon^{2/(a+2)},\, (n\epsilon^{m(k+1)})^{-1/(a+1)}\right\}.
    \end{align}
    With this choice, the dominant $\epsilon$-dependent stochastic term is balanced with the bias term by
    \begin{align}
        \epsilon\asymp n^{-1/(m(k+1)+3a+3)}=n^{-2/(2mk+5m+6)}.
    \end{align}
    Thus, for $\epsilon_n\asymp n^{-r^\delta_{k,m}}$ with $r^\delta_{k,m}=2/(2mk+5m+6)$, the non-logarithmic terms in~\eqref{eq:codiff_combined_rate_before_balance} are bounded by $n^{-r^\delta_{k,m}}$, while
    \begin{align}
        n\epsilon_n^{(k+1)m}=n^{(3m+6)/(2mk+5m+6)}.
    \end{align}
    Therefore, with probability at least $1-p$,
    \begin{align}
        \xi\ls n^{-r^\delta_{k,m}}+\frac{\sqrt{\log(1/p)}}{n^{1/2-r^\delta_{k,m}}}+\frac{\log(1/p)}{n^{(3m+2)/(2mk+5m+6)}}.
    \end{align}

    The stated bilinear estimate follows by polarization. Indeed, the error
    \begin{align}
        B_n(\alpha,\beta)=\langle\hdelta q\alpha,\hdelta q\beta\rangle_{n,\epsilon_n}-\rho_0^{k+2}\sigma^\theta_{\delta,k}\langle\delta\alpha,\delta\beta\rangle_M
    \end{align}
    is a symmetric bilinear form in $\alpha$ and $\beta$, so $B_n(\alpha,\beta)=\frac14\left(B_n(\alpha+\beta,\alpha+\beta)-B_n(\alpha-\beta,\alpha-\beta)\right)$.
    By homogeneity, the quadratic estimate gives $|B_n(\gamma,\gamma)|\leq \xi\|\gamma\|_{\cC^2(M)}^2$
    for every pullback form $\gamma$. Hence, for $\alpha,\beta\in\pi^*\cB_M^k$, polarization and $\|\alpha\pm\beta\|_{\cC^2(M)}\leq 2$ give $|B_n(\alpha,\beta)|\leq 2\xi$, which is the claimed uniform estimate.
\end{proof}

\begin{proof}[Proof of~\Cref{cor:common_bandwidth_convergence}]
    Set $T_n=n\epsilon_n^{m(k+2)+2}=n^{1-a(m(k+2)+2)}\to\infty$ and fix $p\in(0,1)$. For each positive $r\in\{k,k+1\}$, combine~\Cref{prop:uc_bias,prop:uniform_variance_euclidean} with $\gamma_n=T_n^{-1/(2m)}$; since $n\epsilon_n^{rm}\geq T_n$ eventually, every term tends to zero. Degree zero follows from~\Cref{prop:main_uc_degree_zero}. For $k\geq1$, taking $\gamma_n=\epsilon_n^2T_n^{-1/m}$ in~\eqref{eq:codiff_combined_rate_before_balance} bounds its three $\gamma_n$-dependent terms by $T_n^{-1/4}$, $T_n^{-1/2}$, and $T_n^{-1/m}$, while the remaining terms also tend to zero. The concentration hypotheses hold eventually, and a finite union bound proves the simultaneous convergence.
\end{proof}

\appendix
\section{Degree-Zero Inner-Product Convergence} \label{app:degree_zero_inner_product}

\begin{proof}[Proof of~\Cref{prop:main_uc_degree_zero}]
    For a function $f$, the discretization is evaluation, $qf(x)=f(x)$, while $\kappa^0_\epsilon=1$. Hence the degree-zero pairing is
    \begin{align}
        \langle qf,qg\rangle_n^w
        =\frac1n\sum_{i=1}^n f(x_i)g(x_i)w(x_i),
    \end{align}
    and its expectation is exactly
    \begin{align}
        \E[\langle qf,qg\rangle_n^w]
        =\int_M f(x)g(x)w(x)\rho(x)\,\dvol(x).
    \end{align}
    Thus there is no bandwidth-dependent geometric bias.

    We use the same covering method as in the proof of~\Cref{thm:main_uc}, but only for this empirical average. By~\Cref{lem:covering_bound} with $k=0$, for $0<\gamma\leq1$ there are $\eta^1,\ldots,\eta^{\cN_\gamma}\in\cB_\R^0$ such that every $f\in\cB_\R^0$ has an index $r$ with
    \begin{align}
        \|f-\eta^r\|_{\cC^0(M)}<\gamma
        \qquad\text{and}\qquad
        \log\cN_\gamma\ls\gamma^{-m}.
    \end{align}
    The random variables $\eta^r(x_i)\eta^s(x_i)w(x_i)$ are uniformly bounded. Applying the ordinary order-one Hoeffding inequality to every pair $(r,s)$ and taking a union bound gives, with probability at least $1-p$,
    \begin{align}
        \max_{r,s\leq\cN_\gamma}
        \left|
        \frac1n\sum_{i=1}^n\eta^r(x_i)\eta^s(x_i)w(x_i)
        -\int_M\eta^r\eta^s w\rho\,\dvol
        \right|
        \ls
        \sqrt{\frac{\gamma^{-m}+\log(2/p)}{n}}.
    \end{align}
    If $\eta^r$ and $\eta^s$ approximate $f$ and $g$, respectively, the uniform bounds on $\cB_\R^0$ and $w$ give
    \begin{align}
        \|fgw-\eta^r\eta^s w\|_{\cC^0(M)}\ls\gamma.
    \end{align}
    This controls both the empirical and expected approximation errors. Therefore,
    \begin{align}
        \sup_{f,g\in\cB_\R^0}
        \left|
        \langle qf,qg\rangle_n^w-\langle f,g\rangle_M^{w\rho}
        \right|
        \ls
        \sqrt{\frac{\gamma^{-m}+\log(2/p)}{n}}+\gamma.
    \end{align}
    Taking $\gamma=n^{-1/(m+2)}$ proves the result, after absorbing the fixed $\log 2$ term into $n^{-1/(m+2)}$.
\end{proof}

\section{Further Details on AI Use} \label{apxsec:ai_use}

\textbf{Models Used.} Starting June 2025, we used the latest  reasoning models from OpenAI's ChatGPT (non-Pro) and very briefly used Anthropic's Claude. Starting February 2026, we started using the latest ChatGPT Pro models. The majority of our AI use was through interactive chat sessions. Starting May 2026, we started to use OpenAI's Codex. \medskip

\textbf{General Workflow.} We  used AI tools in an interative manner, providing drafts of the article along with notes written by the authors as context in each prompt. For the development of proofs, we would often begin with fairly specific questions\footnote{Much of the work on this project was done between June 2025 - June 2026, and the reasoning capabilities of AI tools have changed quite a lot during this period.} about how to prove a certain lemma / result. We would then use this as an outline to write a proof on our own, and then iterate back and forth with the AI tool to refine the result. We also used AI tools to help with finding stronger existing results to apply, simplification of proofs/exposition, and verification of results. \medskip

\textbf{Details on AI Use in Mathematical Content.}
\begin{itemize}
    \item The authors began the project by using ambient Gaussian kernels, for instance in~\cite{belkin_towards_2008}. AI tools suggested the use of more general kernels, by generalizing those of~\cite{garcia_trillos_error_2020,calder_improved_2022}, which resolved some issues in the codifferential convergence.
    \item The authors proved a preliminary version of~\Cref{thm:main_uc}, and AI tools were used to refine the statement by suggesting a stronger concentration inequality (the Bernstein bound used), and the use of the geometric $U$-kernel (\Cref{ssec:uc_geometric_kernel}) which allowed us to apply an additional symmetry argument to improve the rates.
    \item The general outline of the proof of~\Cref{thm:main_codiff} was developed in collaboration with AI tools, by adapting the proof of~\Cref{thm:main_uc}. In particular, AI tools suggested the variation of the trace formula in~\Cref{lem:codiff_moment_trace}.
    \item In order to deal with the discontinuity of the VR kernels, AI tools were used to refine the domain / kernel replacement lemmas in~\Cref{lem:kappa_domain_replacement} and \Cref{sssec:tangent_space_reduction}. In particular, it identified a previous gap, and suggested~\Cref{lem:codiff_uniform_I_bound}.
    \item For the codifferential proof, the authors began by using an orthogonal projection of the Euclidean center of mass as a center $z(\by)$ for the base simplex in~\Cref{ssec:off_diagonal}. AI tools suggested the use of Karcher means~\cite{karcher_riemannian_1977}, which simplified some of the proofs.
    \item AI tools were used to aid in the development of the proofs in~\Cref{sec:rips_laplacian,ssec:rips_topology}.
\end{itemize}

\section{Notation and Conventions} \label{apxsec:notation}

{\small
\begin{longtable}
    {@{}p{0.25\textwidth}p{0.64\textwidth}r@{}}
    \toprule
    Symbol & Description & Page\\
    \midrule
    \endfirsthead
    \toprule
    Symbol & Description & Page\\
    \midrule
    \endhead
    \midrule
    \multicolumn{3}{r}{\emph{Continued on next page}}\\
    \endfoot
    \bottomrule
    \endlastfoot
    \multicolumn{3}{c}{Global Parameters and Conventions} \\ \midrule
    $m$ & Dimension of the manifold $M$. & \pageref{sec:geometric_prelim}\\
    $N$ & Dimension of the ambient Euclidean space. & \pageref{sec:geometric_prelim}\\
    $n$ & Number of points in the sample $X^{(n)}$. & \pageref{sec:discretization}\\
    $k$ & Differential-form or simplicial-cochain degree. & \pageref{ssec:differential_forms}\\
    $\epsilon$, $\epsilon_n$ & Scale parameter and sample-size-dependent bandwidth. & \pageref{sec:discretization}, \pageref{thm:main_uc}\\
    $\N$, $\N_0$ & Positive and nonnegative integers, respectively. & \pageref{not:global_conventions}\\
    $M\subset\R^N$ & Compact embedded $m$-manifold, possibly nonorientable. & \pageref{sec:geometric_prelim}\\
    $\dvol$, $\vol_M$ & Riemannian volume density and its induced Borel measure. & \pageref{sec:geometric_prelim}\\
    $X^{(n)}=(x_1,\ldots,x_n)$ & Point cloud, usually an i.i.d.\ sample from $M$. & \pageref{sec:discretization}\\
    $\by=(y_0,\ldots,y_k)$ & Boldface notation for an ordered point tuple. & \pageref{ssec:euclidean_derham}\\
    $\tau_M$, $U_\epsilon$, $\pi$ & Reach, tubular neighbourhood, and nearest-point projection. & \pageref{ssec:reach_tubular}\\
    $\bar\nabla$, $\nabla$ & Ambient flat and intrinsic Levi--Civita connections. & \pageref{ssec:ambient_forms}\\
    $\|\cdot\|_{\cC^r(U)}$, $\|\cdot\|_{\cC^r(M)}$ & Ambient coefficientwise and intrinsic covariant $\cC^r$-norms. & \pageref{ssec:ambient_forms}\\ \midrule
    \multicolumn{3}{c}{Differential Forms, Cochains and Discretization} \\ \midrule
    $\Omega^k(M)$, $\Omega_c^k(\R^N)$ & Smooth intrinsic forms and compactly supported ambient forms. & \pageref{ssec:differential_forms}, \pageref{ssec:ambient_forms}\\
    $\langle\cdot,\cdot\rangle_M$, $\langle\cdot,\cdot\rangle_M^f$ & Smooth and $f$-weighted $L^2$ inner products of forms. & \pageref{ssec:differential_forms}\\
    $d$, $\delta$ & Exterior derivative and its smooth $L^2$-adjoint. & \pageref{eq:smooth_codifferential}\\
    $\Delta_k$ & Smooth Hodge Laplacian on $k$-forms. & \pageref{eq:smooth_laplacian}\\
    $X_\epsilon$ & Ambient Vietoris--Rips complex of $X$ at scale $\epsilon$. & \pageref{sec:discretization}\\
    $S_\epsilon^k(X)$ & Canonically oriented $k$-simplices of $X_\epsilon$; cochains extend alternately. & \pageref{sec:discretization}\\
    $C_k(X_\epsilon)$, $C^k(X_\epsilon)$ & Real simplicial chains and cochains. & \pageref{sec:discretization}\\
    $\hd$, $\hdelta$ & Simplicial coboundary and its adjoint for the chosen inner products. & \pageref{sec:discretization}, \pageref{not:discrete_codifferential}\\
    $\Delta^k$, $\sigma_{\by}$ & Standard oriented simplex and its affine parametrization by $\by$. & \pageref{ssec:euclidean_derham}\\
    $q$, $q_{X,\epsilon}$ & Parametric Euclidean de Rham map and its restriction to $X_\epsilon$. & \pageref{ssec:euclidean_derham}\\
    $q_n^\pi$ & Pullback discretization of intrinsic forms on $X^{(n)}_{\epsilon_n}$. & \pageref{prop:topology_quasi_iso_convergence_scales}\\ \midrule
    \multicolumn{3}{c}{Kernels, Weights and Convergence Constants} \\ \midrule
    $p(\by)$ & Vector of pairwise squared distances of $\by$. & \pageref{ssec:bilinear_forms}\\
    $\Theta$ & Defining function of an admissible VR kernel. & \pageref{def:admissible_VR_kernel}\\
    $\kappa_\epsilon$ & General admissible VR kernel. & \pageref{def:admissible_VR_kernel}\\
    $\theta$ & Squared-diameter profile, positive on $[0,1)$ and zero on $[1,\infty)$. & \pageref{not:codiff_kernel_table}\\
    $s(\bz)$ & Squared diameter of a finite tuple $\bz$. & \pageref{not:codiff_kernel_table}\\
    $\kappa^r_{\theta,\epsilon}$ & Degree-$r$ squared-diameter VR kernel. & \pageref{not:codiff_kernel_table}\\
    $w$ & Positive symmetric simplex weight in the cochain bilinear form. & \pageref{eq:inner_prod}\\
    $\langle\cdot,\cdot\rangle_{n,\epsilon}^{\kappa,w}$ & Kernel-weighted bilinear form on cochains. & \pageref{eq:inner_prod}\\
    $\mu$, $\rho$ & Sampling measure on $M$ and its volume density. & \pageref{ssec:probabilistic_inner_products}\\
    $\rho_-$, $\rho_+$, $\rho_0$ & Density bounds and the uniform density $\vol_M(M)^{-1}$. & \pageref{eq:density_bound}, \pageref{not:uniform_pullback_setting}\\
    $\U_n[h]$, $h_\epsilon(\alpha,\beta)$ & $U$-statistic and kernel representing a cochain pairing. & \pageref{def:u_statistic}\\
    $D_\epsilon^k(M)$ & Ambient $(k+1,\epsilon)$-fat diagonal in $M^{k+1}$. & \pageref{eq:fat_diagonal}\\
    $P$, $P_\Delta$ & Joint sampling weight and its diagonal restriction. & \pageref{prop:uc_bias}, \pageref{thm:main_uc}\\
    $\sigma_k^\Theta$ & Deterministic kernel moment in the inner-product limit. & \pageref{eq:inner_product_moment_constant}\\
    $\sigma^\theta_{\delta,k}$ & Deterministic constant in the codifferential limit. & \pageref{thm:main_codiff}\\

    \midrule
    \multicolumn{3}{c}{Normalized VR Laplacian} \\ \midrule
    $\langle\cdot,\cdot\rangle_{n,\epsilon}$, $\|\cdot\|_n$ & Normalized cochain inner product and its bandwidth-$\epsilon_n$ norm. & \pageref{not:normalized_inner_products}\\
    $\sigma_k^\theta$, $\bar{\sigma}_{\delta,k}$ & Kernel normalization and residual down-Laplacian constant. & \pageref{not:normalized_inner_products}\\
    $\hd_n$, $\hdelta_n$ & Sampled coboundary and its normalized adjoint. & \pageref{eq:rips_laplacian}\\
    $\widehat\Delta_{k,n}$ & Normalized Rips Hodge Laplacian. & \pageref{eq:rips_laplacian}\\
    $\lambda_j^{(k)}$, $\widehat\lambda_{j,n}^{(k)}$ & Smooth and normalized discrete Hodge eigenvalues. & \pageref{not:smooth_rayleigh}, \pageref{eq:rips_laplacian}\\
    $Q_k$, $R_k$ & Smooth Hodge energy and Rayleigh quotient. & \pageref{not:smooth_rayleigh}\\
    $\widehat Q_{k,n}$, $\widehat R_{k,n}$ & Discrete Hodge energy and Rayleigh quotient. & \pageref{not:discrete_hodge_energy}\\

\end{longtable}
}

\bibliographystyle{plain}
\bibliography{discrete_approx_merged}

\end{document}